\documentclass[11pt,,leqno,twoside]{article} 
\usepackage{amssymb}
\usepackage{amsmath}
\usepackage{amsthm}

\usepackage{mathrsfs}

\allowdisplaybreaks

\usepackage{titletoc}
\titlecontents{section}[0pt]{\addvspace{2pt}\filright}
              {\contentspush{\thecontentslabel\ }}
              {}{\titlerule*[8pt]{.}\contentspage}

\def\rr{{\mathbb R}}
\def\rn{{{\rr}^n}}

\def\nn{{\mathbb N}}

\def\bd{{\mathbb D}}

\def\cx{{\mathscr X}}

\def\ca{{\mathcal A}}

\def\fz{\infty}
\def\az{\alpha}

\def\dist{{\mathop\mathrm{\,dist\,}}}
\def\loc{{\mathop\mathrm{\,loc\,}}}
\def\lip{{\mathop\mathrm{\,Lip}}}

\def\lz{\lambda}
\def\dz{\delta}

\def\ez{\epsilon}

\def\kz{\kappa}
\def\bz{\beta}

\def\gz{{\gamma}}

\def\vz{\varphi}

\def\sz{\sigma}

\def\wz{\widetilde}

\def\ls{\lesssim}
\def\gs{\gtrsim}

\def\bint{{\ifinner\rlap{\bf\kern.35em--}
\int\else\rlap{\bf\kern.45em--}\int\fi}\ignorespaces}

\def\bbint{{\ifinner\rlap{\bf\kern.35em--}
\hspace{0.078cm}\int\else\rlap{\bf\kern.45em--}\int\fi}\ignorespaces}

\def\aplip{{\rm \,apLip\,}}

\def\osc{ \mathop \mathrm{\, osc\,} }
\def\dosc{\displaystyle\osc}
\def\diam{{\mathop\mathrm{\,diam\,}}}

\def\dint{\displaystyle\int}

\def\r{\right}
\def\lf{\left}

\newtheorem{thm}{Theorem}[section]
\newtheorem{lem}{Lemma}[section]
\newtheorem{prop}{Proposition}[section]
\newtheorem{rem}{Remark}[section]
\newtheorem{cor}{Corollary}[section]
\newtheorem{defn}{Definition}[section]

\numberwithin{equation}{section}

\begin{document}

\arraycolsep=1pt

\title{\Large\bf
Geometry and Analysis of Dirichlet forms
\footnotetext{\hspace{-0.35cm}
 {\it Key words and phases:}  Dirichlet form, intrinsic distance,
length structure, differential structure,
Sierpinski gasket,
gradient flow, Ricci curvature, Poincar\'e inequality, metric measure space
\endgraf Pekka Koskela was supported by the Academy of Finland grant 120972. Yuan Zhou was supported by
Program for New Century Excellent Talents in University of Ministry of Education of China
and National Natural Science Foundation of China (Grant No. ).
\endgraf $^\ast$ Corresponding author.
}}
\author{Pekka Koskela and Yuan Zhou\,$^\ast$
}
\date{ }
\maketitle

\begin{center}
\begin{minipage}{13.5cm}\small
{\noindent{\bf Abstract}\quad
Let $ \mathscr E $ be a regular, strongly local Dirichlet form on
$L^2(X,\,m)$ and $d$ the associated intrinsic distance.
Assume that the topology induced by $d$ coincides with the original topology
on $ X$, and that $X$ is compact, satisfies a doubling property and supports a
weak $(1,\,2)$-Poincar\'e inequality.
We first discuss  the (non-)coincidence of the intrinsic length structure and
the gradient structure.
Under the further assumption that the Ricci curvature of $X$ is bounded from
below in the sense of
Lott-Sturm-Villani,
the following are shown to be equivalent:

(i) the heat flow of $\mathscr E$  gives the unique gradient flow
of $\mathscr U_\infty$,

(ii) $\mathscr E$ satisfies the Newtonian property,

(iii) the intrinsic length structure  coincides with the gradient structure.

Moreover, for the standard (resistance) Dirichlet form on the Sierpinski gasket
 equipped with the Kusuoka measure, we  identify the intrinsic length structure
with the measurable Riemannian  and the gradient structures.
We also apply the above results to the (coarse) Ricci curvatures and
asymptotics of the gradient of the heat kernel.
}
\end{minipage}
\end{center}

\medskip

\section{Introduction\label{s1}}

It is well known that on $\rn$, associated to the Dirichlet energy
$$\int_\rn|\nabla f(x)|^2dx,$$
there is a naturally defined heat semigroup (flow).
Jordan, Kinderlehrer and Otto \cite{jko98} and Otto \cite{o01} understood this
heat flow as a gradient flow of the Boltzman-Shannon entropy
with respect to the $L^2$-Wasserstein metric on the space of probability
measures on $\rn$.
Since then this has been extended to Riemannian manifolds, Finsler manifolds,
Heisenberg groups, Alexandrov spaces and metric measure spaces;
see, for example, \cite{o01,ags,v09,e10,j11,os09,gko,ags11}.
The  gradient flow 
has also attracted considerable attention in various settings,
see, for example, \cite{ags,gko,v09,g10}
and the reference therein.
In particular, the works \cite{ags,g10,gko} in abstract setting motivate one
to extend the above phenomenon of \cite{jko98} to settings such as
metric measure spaces with Ricci curvatures of Lott-Sturm-Villani
\cite{s06a,s06b,lv09} bounded from below.

Moreover, a heat semigroup (flow) is naturally associated to any given
Dirichlet form. Via this, a notion of  Ricci curvature bounded from below was
introduced by Bakry and Emery \cite{be83}.
Observe that the Ricci curvature of Bakry-Emery essentially depends on the
differential (gradient) structure.
On the other hand, under some additional assumptions on the underlying metric
measure space,
a notion of  Ricci curvature  bounded from below  was introduced by
Lott-Villani-Sturm \cite{lv09,s06a,s06b}, purely  in terms of the length
structure.
It is then natural to analyze the connections between these different
approaches; see \cite{gko,ags11}
for seminal studies in this direction.
In this paper, we consider the intrinsic length structures and gradient
structures of Dirichlet forms.

Let $ X$ be a locally compact, connected and separable Hausdorff space and
$ m$  a nonnegative Radon measure with support $ X$.
Let $\mathscr E$ be a regular, strongly local Dirichlet form on $L^2( X)$,
$\Gamma$  the squared gradient and
$d$ the intrinsic distance induced by $\mathscr E$.
We always assume that  the topology induced by $d$ coincides with the original topology on $ X$.

In Section \ref{s2}, we  establish the coincidence of
the intrinsic length structure and the gradient structure of Dirichlet forms under a doubling property, a weak Poincar\'e inequality
and the Newtonian property.
Indeed, we prove that if $( X,\,d,\,m)$ satisfies the doubling property,
then for every $u\in\lip( X)$,
the energy measure
$\Gamma(u,\,u)$ is absolutely continuous with respect to $ m$
and $\frac d{d m}\Gamma(u,\,u)\le (\lip\, u )^2$ almost everywhere; see Theorem \ref{t2.1}.
If we further assume that $( X,\,d,\,m)$ supports a weak $(1,\,p)$-Poincar\'e inequality for some $p\in[1,\,\fz)$
and that $( X,\,\mathscr E,\,m)$ satisfies the Newtonian property introduced in this paper,
then $\frac d{d m}\Gamma(u,\,u)= (\lip\, u )^2$ almost everywhere;
see  Theorem \ref{t2.2}.

In Section \ref{sx2},
by perturbing the classical
Dirichlet energy form of $\rr^2$ on a large Cantor set,
we construct a simple example that satisfies a  doubling property and a weak
Poincar\'e inequality, but so that the  intrinsic length structure does not
coincide  with
the gradient structure; see Proposition \ref{p2.3}.
This shows that a  doubling property  and a weak Poincar\'e inequality are
not sufficient to guarantee  that $\frac d{d m}\Gamma(u,\,u)= (\lip\, u )^2$
almost everywhere. A more general construction can be found in \cite{s97}.
Moreover, the gradient (differential) structure of our perturbed Dirichlet form
does not coincide with the distinguished gradient (differential) structure of
Cheeger's; see Proposition \ref{p2.4}.
Recall that if $(X,\,d,\,m)$ satisfies a doubling property and a weak
$(1,\,p)$-Poincar\'e inequality for some
$p\in[1,\,\fz)$, then Cheeger \cite{c99} constructed a differential structure
equipped with a distinguished inner product norm, which coincides with the
gradient  structure of $\Gamma$
if $(X,\,\mathscr E,\,m)$ further satisfies the Newtonian property;
see Corollary \ref{c2.3}.

In Section \ref{s3}, with the aid of the above results,
for the standard (resistance) Dirichlet form on the standard Sierpinski gasket
 equipped with the Kusuoka measure,
we  identify the intrinsic length structure
with the measurable Riemannian  and the gradient structures.
In particular, some refined Rademacher theorems are established.
See Theorem \ref{t3.1} through Theorem \ref{t3.3} below.

In Section \ref{s4},
we assume that $( X,\,d,\,m)$ is compact and satisfies a
doubling property.
If  the entropy $\mathscr U_\fz$ is weak $\lz$-displacement convex for some $\lz\in\rr$,
then  we obtain the equivalence of the following:

(i) for all Lipschitz functions $u$,
$\frac d{d m}\Gamma(u,\,u)= (\lip\, u )^2$ almost everywhere,

(ii) $( X,\,\mathscr E,\,m)$ satisfies the Newtonian property,

(iii) the heat flow of $\mathscr E$ gives the unique gradient flow of $\mathscr U_\fz$;

\noindent  see Theorems \ref{t4.1} and \ref{t4.2} below.
Recall that the existence and uniqueness of the gradient flow
of $\mathscr U_\fz$ was already established in \cite{ags,g10}.

In Section 6, applying the results of Section 2,  we first obtain a dual fomula related to Kuwada's dual theorem and
 the boundedness from below of the coarse Ricci curvature of Ollivier [35]; this does not
require the Newtonian property.
Moreover, with some additional assumptions, relying on \cite{s07}, we obtain
that if the Ricci curvature of $(X,\,d)$ is bounded from below in the sense of
Lott-Sturm-Villani \cite{s06a,s06b,lv09}, then
the Ricci curvature of $(X,\,\mathscr E)$ is bounded from below in the sense
of Bakry-Emery \cite{b97,be83}.

In Section 7, assuming that $(X,\,\mathscr E,\,m)$ is compact and  has a
spectral gap,
we show that  the identity $\Gamma(d_x,\,d_x)=m$ for all
$x\in X$ actually reflects some short time asymptotics of the gradient of the heat kernel.

Finally, we state some {\it conventions}. Throughout the paper,
we denote by $C$ a {\it positive
constant} which is independent
of the main parameters, but which may vary from line to line.
Constants with subscripts, such as $C_0$, do not change
in different occurrences. The {\it notation} $A\ls B$ or $B\gs A$
means that $A\le CB$. If $A\ls B$ and $B\ls A$, we then
write $A\sim B$.
Denote by  $\nn$ the {\it set of positive integers}.
For any locally integrable function $f$,
we denote by $\bbint_E f\,d\mu$ the {\it average
of $f$ on $E$}, namely, $\bbint_E f\,d\mu\equiv\frac 1{\mu(E)}\int_E f\,d\mu$.


%


\tableofcontents
\contentsline{section}{\numberline{ } References}{48}

\section{Dirichlet forms: $\frac{d}{dm}\Gamma(u,\,u)= (\lip u)^2$\label{s2}}

The main aim of this section is to  establish the coincidence of
the intrinsic length structure and the gradient structure of Dirichlet forms
under a doubling property, a weak Poincar\'e inequality
and the Newtonian property; see Theorem \ref{t2.1} and Theorem \ref{t2.2} below.

Let $ X$ be a locally compact, connected and separable Hausdorff space and
$ m$ be a nonnegative Radon measure with support $ X$. In this paper,
$L^p( X)$ with $p\in(1,\,\fz]$ is
the space of integrable functions of order $p$ on $ X$;
$\mathscr C( X)$ (resp. $\mathscr C_0( X)$)  the collection of all continuous
functions
(with compact supports) on $ X$, and $\mathscr M( X)$ the collection of all signed Radon measures on $ X$.

Recall that a {\it Dirichlet form} $\mathscr {E}$ on $L^2( X)$ is a closed,
 nonnegative definite and  symmetric bilinear form
defined on a dense linear subspace $\bd$ of $L^2( X),$
that satisfies the {\it Markov property}:
for any $u\in\bd$, $v=\min\{1,\,\max\{0,\,u\}\}$, we have $\mathscr E(v,\,v)\le \mathscr E(u,\,u)$.
Then $\mathscr {E}$ is said to be {\it strongly local} if
$\mathscr E(u,\,v)=0$ whenever  $u,\,v \in \bd$ with $u$  a constant on a neighborhood of the support of $v$;
to be {\it regular} if there exists a subset of
 $ \bd\cap\mathscr C_0( X)$ which is both dense in $\mathscr C_0( X)$
with uniform norm and
in $\bd$ with the norm $\|\cdot\|_\bd$ defined by
$\|u\|_\bd=[\|u\|^2_{L^2( X)}+\mathscr E(u,\,u)]^{1/2}$ for each $u\in\bd$.
Beurling and Deny \cite{bd59} showed that a regular, strongly local Dirichlet form $\mathscr E$ can be written as
$$\mathscr E(u,\,v)=\dint_ X d\Gamma(u,\,v) $$
for all $u,\,v\in\bd$, where $\Gamma$ is an $\mathscr M( X)$-valued
 nonnegative definite and symmetric  bilinear form  defined by the formula
\begin{equation}\label{e2.x1}
 \int_ X \phi\, d\Gamma(u,\,v)\equiv\frac12\lf[\mathscr E(u,\,\phi v)+\mathscr E(v,\,\phi u)-\mathscr E(uv,\,\phi)\r]
\end{equation}
for all  $u,\,v \in \bd\cap L^\fz( X)$  and $\phi\in\bd\cap\mathscr C_0 ( X)$.
We call $\Gamma(u,\,v)$ the {\it  Dirichlet  energy measure (squared gradient)} and
  $ \sqrt{\frac{d}{d m}\Gamma(u,\,u)}$  the {\it length of the gradient}.

Observe that, since $\mathscr E$ is strongly local,
$\Gamma$ is local and satisfies the Leibniz rule and the chain rule, see for example \cite{fot}.
Then both
$\mathscr E(u,\,v)$ and $\Gamma(u,\,v$) can be defined for
$u,\,v\in\bd_\loc$, the {\it collection of all $u\in L^2_\loc( X)$
satisfying} that for each relatively compact set $K\subset X$, there exists a function $w\in\bd$
such that $u=w$ almost everywhere on $K$.
With this, the  {\it intrinsic distance on $ X$ associated  to $\mathscr  E$}  is defined by
\begin{equation}\label{e2.1}
  d(x,\,y)\equiv\sup\{u(x)-u(y):\ u\in\bd_\loc\cap \mathscr C( X),\, \Gamma(u,\,u)\le  m\}.
\end{equation}
Here $\Gamma(u,\,u)\le  m$  means that $\Gamma(u,\,u)$ is absolutely continuous with respect to
 $ m$ and $\frac{d}{d m}\Gamma(u,\,u)\le1$ almost everywhere.

In this paper, we always assume that $\mathscr E$ is a regular, strongly local Dirichlet form on $L^2( X)$,
and that the topology induced by $ d$ is equivalent to the original topology on $ X$.
Notice that,  under this assumption, $d$ is a distance, $d(x,\,y)<\fz$ for all $x,\,y\in X$, and
$( X,\, d)$ is a length space; see \cite{s94,s98b,s10}.

For such a space, the very first question is the coincidence of the
gradient structure of $ \Gamma$
and the length structure of $d$.
It is well known that for all $x\in X$, $\Gamma( d_x,\, d_x)\le m$ as proved
in \cite{s94}.
Very recently, it was observed in \cite{flw11} (see also \cite{s10}) that,
for $u\in\lip( X)$ with Lipschitz constant $1$,
we have $\Gamma(u,\,u)\le m$.
Moreover,  under a doubling assumption, we are able to establish
a pointwise relation   between $\frac d{d m}\Gamma(u,\,u)$ and $\lip\,u$ as
follows.
Here and in what follows, for a measurable function $u$,
its  {\it pointwise Lipschitz constant} is defined as
$$\lip\,u(x)\equiv\limsup_{y\to x}\frac{|u(x)-u(y)|}{ d(x,\,y)},$$
and  $\lip( X)$ stands for the {\it  collection of all measurable functions $u$}
with  $$\|u\|_{\lip( X)}\equiv\sup_{x,\,y\in X,\,x\ne y} \frac{|u(x)-u(y)|}{ d(x,\,y)}<\fz.$$
When it is necessary, we also write $\lip $ as $\lip_d$ to specify
the distance $d$.
We say that $( X,\, d,\, m)$ satifies a  {\it  doubling property} if
there exists a constant $C_0>1$ such that for all $x\in X$ and $ r>0$,
\begin{equation}\label{e2.2}
  m( B(x,\,2r))\le C_0 m(B(x,\,r))<\fz.
\end{equation}

\begin{thm}\label{t2.1}
Suppose that $( X,\, d,\, m)$ satisfies a doubling property.
Then $\lip( X)\subset\bd_\loc$ and for every $u\in\lip( X)$, $\Gamma(u,\,u)\le (\lip\, u )^2 m ,$
that is, $\Gamma(u,\,u)$ is absolutely continuous with respect to $ m$
and $$\frac d{d m}\Gamma(u,\,u)\le (\lip\, u )^2$$ almost everywhere.
\end{thm}

The proof of Theorem \ref{t2.1} relies on  the following three auxiliary lemmas.

\begin{lem}\label{l2.1} For  $n\in\nn$,
$E=\{x_i\}_{i=1}^n\subset X$ and $A=\{a_i\}_{i=1}^n\subset\rr$,
set $$ d_{A,\,E}(x)\equiv\max_{i=1,\,\cdots,\,n} \lf\{a_i- d(x_i,\,x)\r\}.$$
Then $\Gamma( d_ {A,\,E},\, d_{A,\,E})\le m.$

Moreover, if $\Gamma(d_x,\,d_x)=m$ for every $x\in X$,
then $\Gamma( d_ {A,\,E},\, d_{A,\,E})= m.$
\end{lem}

\begin{proof}
We prove this by induction.
It is easy to see that if $n=1$,
then from $\Gamma (a_1,\,v)=0$ for all $v\in\bd_\loc$ and from
$\Gamma(  d_{x_1},\,   d_{x_1})\le m$ proven in \cite{s94}, we deduce that
\begin{equation}\label{e2.3}
\Gamma(a_1- d_{x_1},\, a_1- d_{x_1})=
\Gamma(a_1 ,\, a_1- d_{x_1})-\Gamma(  d_{x_1},\, a_1- d_{x_1})
=\Gamma(  d_{x_1},\,   d_{x_1})\le m.
\end{equation}
Now assume that  the claim holds for $n$. We are going to prove it for $n+1$.
To this end, let $E_{n+1}=\{x_i\}_{i=1}^{n+1}\subset X$ and $A_{n+1}\subset\{a_i\}_{i=1}^{n+1}\in\rr$.
Notice that
\begin{eqnarray*}
  d_{A_{n+1},\,E_{n+1}} &&=\max_{i=1,\,\cdots,\,n+1} \lf\{a_i- d_{x_i}\r\}\\
&&=\max\lf\{\max_{i=1,\,\cdots,\,n} \lf\{a_i- d_{x_i}\r\},\,a_{n+1}- d_{x_{n+1}}\r\}\\
&&=\max\lf\{ d_{A_{n },\,E_{n }} ,\,a_{n+1}- d_{x_{n+1}}\r\},
\end{eqnarray*}
where $A_n=A_{n+1}\setminus\{a_{n+1}\}$ and $E_n=E_{n+1}\setminus\{x_{n+1}\}$.
Recall that the following truncation property was proven in \cite{s94}:
$$\Gamma(u\wedge v,\,u\wedge v)=1_{u<v}\Gamma(u ,\,u )+1_{u\ge v}\Gamma(v,\,v),$$
where $u\wedge v=\min\{u,\,v\}$, and $1_F$ refers to the characteristic
function
of $F.$  Denote $u\vee v=\max\{u,\,v\}$.
Then we  have
$$\Gamma(u\vee v,\,u\vee v)=\Gamma((-u)\wedge(- v),\,(-u)\wedge(- v))
=1_{u>v}\Gamma(u ,\,u )+1_{u\le v}\Gamma(v,\,v),$$
and moreover, if $\Gamma(u ,\,u )\le  m$ and $\Gamma(v,\,v)\le m$, then
 $\Gamma(u\vee v,\,u\vee v)\le m.$
Now $$\Gamma(  d_ {A_{n},\,E_{n} }  ,\,
 d_{A_{n},\,E_{n}} ) \le m$$ by induction
and
$$\Gamma( a_{n+1}- d_{x_{n+1}},\,
 a_{n+1}- d_{x_{n+1}} )\le m$$ by \eqref{e2.3}.
Hence, we have
$$\Gamma( d_{A_{n+1},\,E_{n+1}},\, d_{A_{n+1},\,E_{n+1}})\le m,$$
 as desired.

Moreover,  if $\Gamma(d_x,\,d_x)=m$ for every $x\in\cx$,
then \eqref{e2.3} holds with $\le$ replaced by $=$.
With this,  by  induction, we further obtain $\Gamma( d_ {A,\,E},\, d_{A,\,E})= m.$
\end{proof}

\begin{lem}\label{l2.2}
Let $V\subset X$ be a bounded open set. Define  $u(x)\equiv\sup_{z\in V}\{v(z)- d(z,\,x)\}.$
If $v\in\bd_\loc$, $1_V\Gamma(v,\,v)\le 1_V m$ and $\|v\|_{\lip(V)}\le1$,
then  $ \Gamma(u,\,u)\le  m$.
\end{lem}
\begin{proof}
For every $n\in\nn$,
choose  a maximal finite set of $V$, $\{x_{n,\,i}\}\subset V$, such that $ d(x_{n,\,i},\,x_{n,\,j})\ge \frac1n\diam V$,
and for all $x\in X$, set
$$u_n\equiv\max_i\{v(x_{n,\,i})- d(x_{n,\,i},\,x)\}.$$
Then, by Lemma \ref{l2.1}, $\Gamma(u_n,\,u_n)\le m$, which implies that
$\{u_n\}_{n\in\nn}$ is a locally bounded set and hence has a subsequence which converges weakly in $\bd_\loc$
to some $u_0$. Without loss of generality, we still denote this subsequence by $\{u_n\}_{n\in\nn}$.
Now $\Gamma(u_0,\,u_0)\le\lim_{n\to\fz} \Gamma(u_n,\,u_n)\le m$.

It suffices to  show that $ u =u_0$. To see this, we first notice that
$u_n(x)\le u(x)$ for all $x\in X$.  On the other hand,  obviously, for all $x\in V$, $u(x)=v(x)$ and for all $i$, $u(x_{n,\,i})=v(x_{n,\,i})=u_n(x_{n,\,i})$.
For any $x\in X$, there exists $z\in V$ such that
$u(x)\le u(z)- d(z,\,x)+\frac1n\diam V.$
By the choice of $x_{n,\,i}$, we can find $x_{n,\,i}\in B(z,\,\frac2n\diam V)$.
Since $\|v\|_{\lip(V)}\le1$, we have  $|u(z)-u(x_{n,\,i})|= |v(z)-v(x_{n,\,i})|\le d(x,\,x_{n,\,i})$.
Hence
\begin{eqnarray*}
 u(x)&&\le u(z)- d(z,\,x)+\frac1n\diam V\\
&&=u(x_{n,\,i})- d(x_{n,\,i},\,x)+
 u(z)-u(x_{n,\,i})- d(z,\,x)+d(x_{n,\,i},\,x)+\frac1n\diam V\\
&&\le u_n(x)+ 2 d(z,\,x_{n,\,i})+\frac1n\diam V\\
&&\le u_n(x)+\frac1n\diam V.
\end{eqnarray*}
So $u_n\to u$ uniformly. Thus $u_n\to u=u_0$ weakly in $\bd_\loc$,
which implies that
$$ \Gamma(u,\,u)\le\liminf_{n\to\fz} \Gamma(u_n,\,u_n)\le  m.$$
This finishes the proof Lemma \ref{l2.2}.
\end{proof}

The following lemma was established in \cite[Lemma 6.30]{c99}.
Its proof uses the Lusin theorem and relies on decay property of a doubling measure on a length space
 observed in \cite{cm98}.

\begin{lem}\label{l2.3}
Suppose that $( X,\,d,\, m)$ satisfies a doubling propery.
Then for every ball $B(x_0,\,r_0)\subset X$, there exists a constant $C_2\ge 1$ such that
for every  $n \in\nn$ and  $u\in\lip(B(x_0,\,r_0))$,
there exists 
a finite collection $\{B(x_{n,\,j},\,r_{n,\,j})\}$ of mutually disjoint balls
with $x_{n,\,j}\in B(x_0,\,r_0)$ and $r_{n,\,j}\le r_0$ satisfying that
\begin{eqnarray}
&\dist(B(x_{n,\,i},\,r_{n,\,i}),\, B(x_{n,\,j},\,r_{n,\,j}))\ge\frac12 (r_{n,\,i}+r_{n,\,j}), &\label{e2.4}\\
& m(B(x_0,\,r_0)\setminus\cup_{j}B(x_{n,\,j},\,r_{n,\,j}))\le C_2\frac1n m(B(x_0,\,r_0)),&\label{e2.5}\\
&\int_{B(x_{n,\,j}\,3r_{n,\,j})}|\lip\,u(x)-\lip\, u(x_{n,\,j})|^2\,d m\le\frac1nm(B(x_{n,\,j}\,3r_{n,\,j}))&\label{e2.6}
\end{eqnarray}
and so for all  $x,\,y\in B(x_{n,\,j},\,r_{n,\,j})$ with $ d(x,\,y)\ge \frac1 n r_{n,\,j}$,
\begin{equation}\label{e2.7}
 \frac{|u(x)-u(y)|}{ d(x,\,y)}<\lip\,u(x_{n,\,j})+\frac1n.
\end{equation}
\end{lem}

\begin{proof}[Proof of Theorem \ref{t2.1}.]
Let $u\in\lip( X)$.
It suffices to prove that for every ball $B(x_0,\,r_0)\subset  X$,
\begin{equation}\label{e2.8}
\int_ X   1_{B(x_0,\,r_0)}\,d \Gamma(u,\,u) \le \int_ X   1_{B(x_0,\,r_0)}(\lip\,u)^2\,d m.
\end{equation}
Indeed, by this and a covering argument, one can show that
$\Gamma(u,\,u)$ is absolutely continuous with respect to $m$,
and $\frac d{d m}\Gamma(u,\,u)\le (\lip\,u)^2$ almost everywhere.
We omit the details.

To prove \eqref{e2.8}, we need the following construction via the MacShane extension,
which is a slight modification of that in \cite{c99}.
For $n\in\nn$,  let $\{B(x_{n,\,j},\,r_{n,\,j})\}$ be the covering provided by Lemma \ref{l2.3}.
For every $j$, we choose a maximal set
$\{z_{n,\,j,\,k}\}\subset B(x_0,\,r_0)$ such that for $k\ne \ell$,
$$ d(z_{n,\,j,\,k},\,z_{n,\,j,\,\ell})\ge \frac1{ n}r_{n,\,j}.$$
Define a function $u_n$ on $\cup_jB(x_{n,\,j},\,r_{n,\,j})$ as follows:
 for $x\in B(x_{n,\,j},\,r_{n,\,j})$, set
$$u_n(x)\equiv\max_k\{u(z_{n,\,j,\,k})-L_j d(z_{n,\,j,\,k},\,x)\},$$
where $L_{n,\,j}\equiv\lip\,u(x_{n,\,j})+\frac1n$,
and for $x\in  X\setminus\cup_{j}B(x_{n,\,j},\,r_{n,\,j})$,
set
$$u_n(x)\equiv\sup_{z\in\cup_{j}B(x_{n,\,j},\,r_{n,\,j}) }\lf\{u_n(z)-\|u_n\|_{\lip( \cup_{j}B(x_{n,\,j},\,r_{n,\,j}))} d(z,\,x)\r\}.$$

Notice that for almost all $x\in B(x_{n,\,j},\,r_{n,\,j})$,
since
\begin{equation}\label{e2.9}
L_{n,\,j}\ge \max_{k\ne\ell}\lf\{\frac{|u(z_{n,\,j,\,k})-u(z_{n,\,j,\,\ell})|}{ d(z_{n,\,j,\,k},z_{n,\,j,\,\ell})} \r\},\end{equation}
we have
\begin{equation}\label{e2.10}\lip\,u_n(x)=\|u_n\|_{\lip(B(x_{n,\,j},\,r_{n,\,j}))}= L_{n,\,j}. \end{equation}
Then by Lemma \ref{l2.1} and the strong locality of $\Gamma$,
\begin{eqnarray}
 &&1_{B(x_{n,\,j},\,r_{n,\,j})}\Gamma \lf(\frac1{L_{n,\,j}}u_n,\,\frac1{L_{n,\,j}}u_n\r)\label{e2.11}\\
&&\quad= 1_{B(x_{n,\,j},\,r_{n,\,j})}
 \Gamma \lf(\max_k\lf\{\frac 1{L_{n,\,j}}u(z_{n,\,j,\,k})- d(z_{n,\,j,\,k},\,\cdot)\r\},\,\r.\nonumber\\
&&\quad\quad\lf.\max_k\lf\{\frac 1{L_{n,\,j}}u(z_{n,\,j,\,k})-  d(z_{n,\,j,\,k},\,\cdot)\r\}\r)\nonumber\\
&&\quad\le
 1_{B(x_{n,\,j},\,r_{n,\,j})} m.\nonumber
\end{eqnarray}
Moreover, by Lemma \ref{l2.3},
\begin{equation}\label{e2.14}
 \Gamma(u_n,\,u_n)\le  \|u \|_{\lip(\cup_{j}B(x_{n,\,j}\,r_{n,\,j}))}m=(\sup_j L_{n,\,j})^2 m\le(\|u\|_{\lip( X)}+1)^2m,
\end{equation}
which implies  that
 $\int_ X1_{B(x_0,\,r_0)}\Gamma(u_n,\,u_n)$ is bounded in $\bd$.
So there is a subsequence of $\{1_{B(x_0,\,r_0)}u_n\}_{n\in\nn}$ weakly
converging to some $v\in\bd$. Without loss of generality, we still denote the
subsequence by the sequence itself, and hence
 $$\int_ X\,d\Gamma(v,\,v)\le\liminf_{n\to\fz}\int_ X 1_{B(x_0,\,r_0)}\,d\Gamma(u_n,\,u_n).$$

On the other hand,  by \eqref{e2.9},
we have $u_n(z_{n,\,j,\,k})=u(z_{n,\,j,\,k})$ for all $j$ and $k$.
For every $x\in B(x_j,\,r_j)$, by the choice of $z_{n,\,j,\,k}$,
there exists $z_{n,\,j,\,k}$ such that $d(x,\,z_{n,\,j,\,k})\le \frac1m$,
and hence
\begin{eqnarray*}
 |u(x)-u_n(x)|&&\le|u(x)-u(z_{n,\,j,\,k})|+|u_n(x)-u_n(z_{n,\,j,\,k})|\\
&&\le (\|u\|_{\lip( X)}+L_{n,\,j})d(x,\,z_{n,\,j,\,k})
\le \frac1n(2\|u\|_{\lip( X)}+1).
\end{eqnarray*}
For $x\in B(x_0,\,r_0)\setminus\cup_j{B(x_{n,\,j},\,r_{n,\,j})}$, we have
$$|u_n(x)-u(x)|\le|u_n(x)-u_n(z_{n,\,j,\,k})|+|u(x)-u(z_{n,\,j,\,k})|\le 2(2\|u\|_{\lip( X)}+1) r_0.$$
Thus we have
\begin{eqnarray}\label{e2.15}
 \|u_n-u\|^2_{L^2(B(x_0,\,r_0))}&&\ls\|u-u_n\|_{L^2(\cup_jB(x_{n,\,j},\,r_{n,\,j}))}\\
&&\quad+
2(2\|u\|_{\lip( X)}+1) r_0  m(B(x_0,\,r_0)\setminus\cup_j B(x_{n,\,j},\,r_{n,\,j}))\nonumber\\
&&\ls C(u,\,B(x_0,\,r_0))\frac1n m(B(x_0,\,r_0)),\nonumber
\end{eqnarray}
where $C(u,\,B(x_0,\,r_0))$ is a constant independent of $n$. This means that
    $\{1_B(x_0,\,r_0)u_n\}_{n\in\nn}$
converges to  $1_B(x_0,\,r_0)u$ in $L^2( X)$,
and hence $v=1_B(x_0,\,r_0)u$,
which together with the  locality of $\Gamma$ implies that
\begin{equation}\label{e2.16}
\int_ X1_{B(x_0,\,r_0)}\,d\Gamma(u,\,u)\le\liminf_{n\to\fz}\int_ X 1_{B(x_0,\,r_0)}\,d\Gamma(u_n,\,u_n).
\end{equation}

Now we estimate $\int_ X 1_{B(x_0,\,r_0)}\Gamma(u_n,\,u_n)$ from above.
Observe that by \eqref{e2.11},
\begin{equation}\label{e2.17}
1_{B(x_{n,\,j},\,r_{n,\,j})}\Gamma ( u_n,\, u_n)\le (L_j)^21_{B(x_{n,\,j},\,r_{n,\,j})} m=(\lip\,u_n)^21_{B(x_{n,\,j},\,r_{n,\,j})} m
\end{equation}
which yields
\begin{eqnarray}\label{e2.18}
\sum_{j}\int_ X 1_{B(x_{n,\,j},\,r_{n,\,j})}\,d
\Gamma ( u_n,\, u_n)
\le\sum_{j}\int_ X 1_{B(x_{n,\,j},\,r_{n,\,j})} (\lip\,u_n)^2\,d m
\end{eqnarray}

Moreover, by the triangle inequality, Lemma \ref{l2.3} again, \eqref{e2.10},
\eqref{e2.10} and the doubling property, we have
\begin{eqnarray}\label{e2.19}
&&\lf|\lf\{\sum_{j}\int_ X 1_{B(x_{n,\,j},\,r_{n,\,j})}(\lip\, u)^2\,d m\r\}^{1/2}-
\lf\{\sum_{j}\int_ X 1_{B(x_{n,\,j},\,r_{n,\,j})}(\lip\, u_n)^2\,d m\r\}^{1/2}\r| \\
&&\ \le \lf\{\sum_{j}\int_ X 1_{B(x_{n,\,j},\,r_{n,\,j})}(\lip\, u-\lip\, u_n)^2\,d m\r\}^{1/2}\nonumber\\
&&\ \le \lf\{\sum_{j}\int_ X 1_{B(x_{n,\,j},\,r_{n,\,j})}(\lip\, u-\lip\, u(x_{n,\,j}))^2\,d m\r\}^{1/2}\nonumber\\
&&\ \quad \quad+\lf\{\sum_{j}\int_ X 1_{B(x_{n,\,j},\,r_{n,\,j})}(L_{n,\,j}-\lip\, u(x_{n,\,j}))^2\,d m \r\}^{1/2}\nonumber\\
&&\ \ls\frac1n \lf\{\sum_{j} m(B(x_{n,\,j},\,3r_{n,\,j}))\,d m\r\}^{1/2}+\frac1n [ m(B(x_0,\,2r_0))]^{1/2} \nonumber\\
&&\ \ls\frac1n m(B(x_0,\,r_0)). \nonumber
\end{eqnarray}
>From this and \eqref{e2.18}, it follows that
\begin{eqnarray*}
&&\lf\{\sum_{j}\int_ X 1_{B(x_{n,\,j},\,r_{n,\,j})}\,d\Gamma(u_n,\,u_n)\r\}^{1/2}\\
&&\quad\le
 \lf\{\sum_{j}\int_ X 1_{B(x_{n,\,j},\,r_{n,\,j})}(\lip\, u)^2\,d m\r\}^{1/2}+C\frac1n m(B(x_0,\,r_0))
\\ &&\quad\le\lf\{\int_ X 1_{B(x_0,\,r_0)}(\lip\, u)^2\,d m\r\}^{1/2}+C\frac1n m(B(x_0,\,r_0))
\end{eqnarray*}
which together with    \eqref{e2.5} and \eqref{e2.14}
yields
\begin{eqnarray*}
 \lf\{\int_ X 1_{B(x_0,\,r_0)}\,d\Gamma(u_n,\,u_n) \r\}^{1/2}&&\le  \lf\{\sum_{j}\int_ X 1_{B(x_{n,\,j},\,r_{n,\,j})}
\,d\Gamma(u_n,\,u_n) \r\}^{1/2}+ \frac Cn  m(B(x_0,\,r_0))\\
&&\le \lf\{\int_ X 1_{B(x_0,\,r_0)}(\lip\, u)^2\,d m\r\}^{1/2}+C\frac1n m(B(x_0,\,r_0)).
\end{eqnarray*}
Therefore, by \eqref{e2.16}, we obtain \eqref{e2.8}.
\end{proof}

\begin{cor}\label{c2.1}
Assume that $(X,\,d,\,m)$ satisfies a doubling property.
For every $u\in\lip(X)$,
$$\|u\|_{\lip(X)}=\sup_{x\in X}\lip\,u(x)=\|\lip\,u\|_{L^\fz(X)}=\lf\|\frac{d}{dm}\Gamma(u,\,u)\r\|_{L^\fz(X)}^{1/2}.$$
\end{cor}

\begin{proof}
By the definition of $\lip\,u(x)$, we easily have $\|u\|_{\lip(X)}\ge\lip\, u(x) $ for all $x\in X$.
The inequality
 $\sup_{x\in X}\lip\, u(x)\ge\|\lip\,u\|_{L^\fz(X)}$ is trivial.
By Theorem \ref{t2.1}, we also have $\|\lip\,u\|_{L^\fz(X)}\ge\|\frac{d}{dm}\Gamma(u,\,u)\|_{L^\fz(X)}^{1/2}$.
Now, the proof of Corollary \ref{c2.1} is reduced to proving that
$\|u\|_{\lip(X)}\le \|\frac{d}{dm}\Gamma(u,\,u)\|_{L^\fz(X)}^{1/2}$.

Fix $u\in\lip(X)$ with $\|\frac{d}{dm}\Gamma(u,\,u)\|_{L^\fz(X)}<\fz$ (by Theorem \ref{t2.1} this actually holds for each $u\in \lip(X)$).
Then, for $\ez >0$, we have $v_\ez\equiv u(\|\frac{d}{dm}\Gamma(u,\,u)\|_{L^\fz(X)}+\ez)^{-1/2}\in\bd_\loc $ and
$\Gamma(v_\ez,\,v_\ez)\le m$.
By \eqref{e2.1}, we have that for all $x,\,y\in X$, $|v_\ez(x)- v_\ez(y)|\le d(x,\,y)$,
which implies that
$$ |u(x)-u(y) |\le \lf(\lf\|\frac{d}{dm}\Gamma(u,\,u)\r\|_{L^\fz(X)}+\ez\r)^{1/2}d(x,\,y).$$
This, together with the arbitrariness of $\ez>0$, implies that
$\|u\|_{\lip(X)}\le \|\frac{d}{dm}\Gamma(u,\,u)\|_{L^\fz(X)}^{1/2}$
 as desired.
\end{proof}




\begin{rem}\rm
(i) In the proof above, we used the result that
$\Gamma(d_x,\,d_x)\le m$ from \cite{s94},
but did not use the conclusion from \cite{flw11}
that this also holds for each  $1$-Lipschitz function $u$.

(ii) The doubling property in Theorem \ref{t2.1} can be relaxed to a local doubling property: for every $x_0\in X$,
there exists $r_{x_0}>0$ and $C_{x_0}$ such that for all $x\in B(x_0,\,r_{x_0})$ and $r\le r_{x_0}$,
 $ m( B(x,\,2r))\le C_{x_0} m(B(x,\,r))<\fz.$
We would like to know if Theorem \ref{t2.1} holds for a general strongly local Dirichlet form.
\end{rem}

Applying Theorem \ref{t2.1}, we clarify the relations of two kinds of weak Poincar\'e inequalities on $ X$
with the aid of a quasi-Newtonian property.

Recall that  $( X,\,\mathscr E,\, m)$ is said to support a
 {\it weak $(1,\,p)$-Poincar\'e inequality} with $p\in[1,\,\fz)$
if there exist constants $\lz\ge1$ and $C >0$ such that for all $u\in\lip( X)$,
 $x\in X$ and $r>0$,
\begin{equation}\label{e2.20}
\bint_{B(x,\,r)}|u- u_{B(x,\,r)}|\,d m\le Cr \lf\{ \bint_{{B(x,\,\lz r)}}\,\lf[\frac d{d m}\Gamma(u,\,u) \r]^{p/2}\,d m\r\}^{1/p}.
\end{equation}
Similarly,  $( X,\, d,\, m)$  is said to support a {\it weak $(1,\,p)$-Poincar\'e inequality}
if  \eqref{e2.20} holds with $\lf[\frac d{d m}\Gamma(u,\,u) \r]^{p/2}$ replaced by $(\lip\,u)^p$.

We say that $( X,\,\mathscr E,\,m)$ satisfies a {\it $K$-quasi-Newtonian property} if
for every $u\in\lip( X)$, there exists a Borel representative $g$
of $\sqrt{\frac d{dm}\Gamma(u,\,u)}$
such that for all Lipschitz curves $\gz:[0,\,1]\to X$,
\begin{equation*}
 |u(\gz(0))-u(\gz(1))|\le K\int_\gz g\,ds.
\end{equation*}
Here
 $g$ is called a {\it Borel representative} of a measurable function $h$ if
$g$ is a  Borel measurable function and
satisfies that $g(x)\ge h(x)$ for all $x\in X$ and
$g(x)=h(x)$ for almost all $x\in X$.
If $K=1$, we say that $( X,\,\mathscr E,\,m)$ satisfies the
{\it Newtonian property};
otherwise we say that $( X,\,\mathscr E,\,m)$ satisfies a
{\it quasi-Newtonian property}.

\begin{prop}\label{p2.1}
Suppose that $( X,\, d,\, m)$ satisfies a doubling property.
Then for every $p\in[1,\,\fz)$, $( X,\,\mathscr E,\, m)$ supports a weak $(1,\,p)$-Poincar\'e inequality
if and only if $( X,\, \mathscr E,\, m)$ satisfies a quasi-Newtonian property and
$( X,\, d,\, m)$ supports a weak $(1,\,p)$-Poincar\'e inequality.
\end{prop}

To prove Proposition \ref{p2.1}, we recall the notion of an upper gradient; see \cite{hk98} and also \cite{k03,s00}.
Recall that a nonnegative Borel measurable function $g$ is called a {\it $p$-weak upper gradient of $u$} with $p\in[1,\,\fz)$ if
\begin{equation}\label{e2.21}
|u(x)-u(y)|\le \int_\gz g\,ds
\end{equation}
for all $\gz\in\Gamma_{\rm rect}\setminus \Gamma_o$,
where $x$ and $y$ are the endpoints of $\gz$,
$\Gamma_{\rm rect}$ denotes the collection of non-constant compact rectifiable  curves
and $\Gamma_o$ has $p$-modulus zero in the sense that
$$ \inf\lf\{\|\rho\|_{L^p(X)}^p:\ \rho\ \mbox{is non-negative, Borel measurable,\ } \int_\gz\rho\,ds\ge1\mbox{\ for all\ }\gz\in\Gamma_o\r\}=0.$$
We denote by $N^{1,\,p}( X)$ the collection of functions $u\in L^p( X)$ that have a
$p$-weak upper gradient $g\in L^p( X)$,
and moreover, $\|u\|_{N^{1,\,p}( X)}=\|u\|_{L^p( X)}+\inf_{g}\|g\|_{L^p( X)}$,
where $g$ is taken over all $p$-weak  upper gradients of $u$.
We denote by  $N^{1,\,p}_\loc( X)$ the class of functions $u\in L_\loc^p(X)$ that have a
$p$-weak upper gradient that belongs to $L^p(B)$ for each ball $B$.

For the following relations between the weak upper gradient and the  (approximate) pointwise Lipschitz constant,
see  \cite[Theorem 6.38]{c99} with a correction in \cite[Remark 2.16]{k04}  and also \cite{s00,k03}.
For a measurable function $u$,
 its {\it approximate pointwise Lipschitz constant} is defined as
$$\aplip u(x)\equiv\inf_A\limsup_{y\in A,\,y\to x}\frac{|u(x)-u(y)|}{ d(x,\,y)} $$
for every $x\in X$,
where the infimum is taken over all Borel sets $A\subset X$ with a point of density at $x$.
Notice that if $u\in\lip( X)$, then $\aplip u=\lip\, u$ almost everywhere.

\begin{lem}\label{l2.5}
Suppose that $( X,\,d,\, m)$ satisfies a doubling property  and
 supports a weak $(1,\,p)$-Poincar\'e inequality for some $p\in[1,\,\fz)$.
Then for every $u\in N^{1,\,p}_\loc( X)$,
there exists a unique  $p$-weak upper gradient $g_u$ of $u$ such that
 $g_u=\aplip u$ almost everywhere and $g_u\le g$ almost everywhere whenever $g$ is a $p$-weak upper gradient of $u$.
In particular, if $u\in\lip( X)$, then $g_u=\lip\, u$ almost everywhere.
\end{lem}

Proposition \ref{p2.1} follows from Theorem \ref{t2.1} and the following lemma.

\begin{lem} \label{l2.4}
Suppose that $( X,\,d,\, m)$ satisfies a doubling property  and
$( X,\,\mathscr E,\,  m)$ supports a weak $(1,\,p)$-Poincar\'e inequality for some $p\in[1,\,\fz)$.
Then there exists a constant $C_1\ge1$ such that for all $u\in N^{1,\,p}_\loc( X)$,
$$(\aplip  u )^2\le C_1\frac d{d m}\Gamma(u,\,u) $$ almost everywhere.
\end{lem}

\begin{proof}[Proof of Lemma \ref{l2.4}]
Let $u\in N^{1,\,p}_\loc( X)$ and let $g_u$ be the $p$-weak upper gradient of $u$ as in Lemma \ref{l2.5}. Set
$$g_k(x)\equiv\sup_{j\ge k}\lf\{\bint_{B(x,\,\lz 2^{-j})}\lf[\frac{d}{d m}\Gamma(u,\,u)\r]^{p/2}\,d m\r\}^{1/p}.$$
Then $g_k$ is Borel measurable; indeed, $g_k$ is lower semicontinuous.
Observe that if $g_k(x)<\fz$, then $\lim_{j\to\fz}u_{B(x,\,2^{-j})}$ exists.
In fact,
since for every $j$,
$$ \bint_{B(x,\,2^{-j})}|u -u_{B(x,\,2^{-j})}| \,d m
\ls 2^{-j} \lf\{\bint_{B(x,\,\lz 2^{-j})}\lf[\frac{d}{d m}\Gamma(u,\,u)\r]^{p/2}\,d m\r\}^{1/p},$$
 by a telescope argument, we have
$$|u_{B(x,\,2^{-j})}-u_{B(x,\,2^{-\ell})}|\ls 2^{-\min\{j,\,\ell\}} g_k(x)\to 0$$
as $j,\,\ell \to \fz$.
For such an $x$, we define
$\wz u(x)\equiv\lim_{j\to\fz}u_{B(x,\,2^{-j})}$.
Generally, for $x\in X$, if
$\lim_{j\to\fz} u_{B(x,\,2^{-j})}$ exists,  then we define
$\wz u(x)\equiv\lim_{j\to\fz} u_{B(x,\,2^{-j})}$; otherwise,
 set $\wz u(x)\equiv0$. Obviously, $u(x)=\wz u(x)$ for almost all $x\in X$,
and hence  $u$ and $\wz u$ generate the same element of $ N_\loc^{1,\,p}( X)$.

Now we are going to check that $g_k$ is a $p$-weak upper gradient
of $\wz u$.  Observe that by a telescope argument again,
for all $x,\,y\in X$ with $ d(x,\,y) \le 2^{-k-2}$,
we have
$$|\wz u(x)-\wz u(y)|\ls  d(x,\,y)[g_k(x)+g_k(y)].$$
Recall that, by \cite[Proposition 3.1]{s00}, $\wz u$ is absolutely continuous on $p$-almost every curve,
namely, $\wz u\circ \gz$ is
absolutely continuous on $[0,\,\ell(\gz)]$ for all arc-length
parameterized paths $\gz\in\Gamma_{\rm rect}\setminus \Gamma$,
where $\Gamma$ has $p$-modulus zero.
For  $\gz\in\Gamma_{\rm rect}\setminus \Gamma$, we are going to show that
\begin{equation}\label{e2.22}
|\wz u(x)-\wz u(y)|\ls \int_\gz  g_k\,ds.
\end{equation}
To this end, by the absolute continuity of $u$ on $\gz$,
it suffices to show that for $j$ large enough,
$$2^j\lf|\int_0^{2^{-j}}\wz u\circ \gz(t)\,dt-\int_{\ell(\gz)-2^{-j}}^{\ell(\gz)}\wz u\circ \gz(t)\,dt\r|\ls \int_0^{\ell(\gz)}g_k\circ\gz(t)\,dz.$$
But,  for $j$ large enough, we have that
\begin{eqnarray*}
&&2^j\lf|\int_0^{2^{-j}}\wz u\circ \gz(t)\,dt-\int_{\ell(\gz)-2^{-j}}^{\ell(\gz)}\wz u\circ \gz(t)\,dt\r|\\
&&\quad=2^j\lf|\int_0^{\ell(\gz)-2^{-j}}[\wz u\circ \gz(t+2^{-j}) -\wz  u\circ \gz(t)]\,dt\r|\\
&&\quad\le  2^j\int_0^{\ell(\gz)-2^{-j}}\lf|\wz u\circ \gz(t+2^{-j}) -\wz  u\circ \gz(t)\r|\,dt\\
&&\quad\ls\int_0^{\ell(\gz)-2^{-j}}\lf[g_k\circ \gz(t+2^{-j}) +g_k\circ \gz(t)\r]\,dt\\
&&\quad\ls\int_0^{\ell(\gz) } g_k\circ \gz(t) \,dt.
\end{eqnarray*}
This gives \eqref{e2.22} and hence   $g_k$ is a $p$-weak upper gradient of $\wz  u$.
Notice that Lemma \ref{l2.5} gives that $\aplip    u$ coincides with the unique minimal $p$-weak upper gradient of $  u$
and hence that of $\wz u$ almost everywhere.
So  $\aplip  u \ls g_k$ almost everywhere
and hence, by the Lebesgue differentiation theorem, for almost all $x\in X$,
\begin{eqnarray*}
           \aplip  u(x) &&\ls \liminf_{k\to\fz} g_k(x)\\
&&\ls \lim_{k\to\fz}\sup_{j\ge k}\lf\{\bint_{B(x,\,\lz 2^{-j})}\lf[\frac{d}{d m}\Gamma(u,\,u)\r]^{p/2}\,d m\r\}^{1/p}\\
&&\ls
\lf\{ \frac{d}{d m}\Gamma(u,\,u)(x)\r\}^{1/2}.
           \end{eqnarray*}
This finishes the proof of Lemma \ref{l2.4}.
\end{proof}

Moreover,  from Theorem \ref{t2.1} and Lemma \ref{l2.5} we conclude the
following result.

\begin{prop}\label{p2.2}
 Suppose that $( X,\,d,\,m)$ satisfies a doubling property and supports a
weak $(1,p)$-Poincar\'e inequality for some $p\in[1,\,\fz)$.
Then $( X,\,\mathscr E,\,m)$ satisfies the Newtonian property
if and only if for all $u\in\lip( X)$,
\begin{equation*}
  \frac{d}{d m}\Gamma(u,\,u)=(\lip\, u)^2
 \end{equation*} almost everywhere.
\end{prop}

When $p=2$, we further have the following conclusion.  Recall that, as proved by Sturm \cite{s96}, $( X,\,d,\,m)$ satisfies
 the doubling property and $( X,\,\mathscr E,\,m)$ supports a weak
$(1,\,2)$-Poincar\'e inequality
if and only if a scale invariant Harnack inequality
for the parabolic operator
$\frac{\partial}{\partial t}-\Delta$ on $\rr\times X$ holds true, with
$\Delta$ corresponding to $\mathscr E.$

\begin{thm}\label{t2.2}
Suppose that $( X,\,d,\,m)$ satisfies a doubling property and $( X,\,\mathscr E,\,m)$
supports a weak $(1,\,2)$-Poincar\'e inequality.  Then the following hold:

  (i) $\bd=N^{1,\,2}( X)$ with equivalent norms, $\lip( X)\cap\mathscr C_0( X)$ is dense in $\bd$, and
$\bd_\loc=N^{1,\,2}_\loc( X)$;

(ii) for all $u\in\bd_\loc$, $\Gamma(u,\,u)$ is absolutely continuous with respect to $ m$,
and there exists a constant $C_1\ge1$ such that for all $u\in\bd_\loc$,
\begin{equation}\label{e2.23}
\frac d{d m}\Gamma(u,\,u)\le  (\aplip u)^2 \le C_1\frac d{d m}\Gamma(u,\,u)
\end{equation}
almost everywhere, where $\aplip u=\lip\, u$ almost everywhere for $u\in\lip( X)$.

(iii) If $( X,\,\mathscr E,\,m)$ satisfies the
 Newtonian property, then $C_1=1$ in \eqref{e2.23}.
\end{thm}

\begin{proof}
Recall that if $( X,\,d,\,m)$ supports a weak $(1,\,2)$-Poincar\'e inequality,
then $\lip( X)\cap \mathscr C_0( X)$ is dense in $N^{1,\,2}( X)$
as proved in \cite{s00,c99}. Notice that
 Theorem  \ref{t2.1} and Lemma \ref{l2.4} implies that  $\frac d{d m}\Gamma(u,\,u)\sim  (\lip u)^2
$ holds almost everywhere  for all $u\in \lip(X)$.
Under this, it was proved in \cite{ kst04,s09} that
$\bd=N^{1,\,2}( X)$ with equivalent norms.
This  implies that $\lip( X)\cap \mathscr C_0( X)$ is dense in $\bd$
and also that $\bd_\loc=N^{1,\,2}_\loc( X)$. This gives (i).

Obviously,
for $u\in\lip( X)$, (ii)  follows from Theorem \ref{t2.1},
Lemma \ref{l2.4} and Proposition \ref{p2.2}.
For  $ u\in\bd_\loc$,  by $\bd_\loc=N^{1,\,2}_\loc( X)$, we have that
$\Gamma(u,\,u)$ is absolutely continuous with respect to $ m$,
and    Lemma \ref{l2.4} gives
$(\aplip u)^2 \le C_1\frac d{d m}\Gamma(u,\,u)$ almost everywhere.
Finally, a density argument together with the closedness of $\mathscr E$ and
the fact that (ii) holds for Lipschitz functions leads to
 $\frac d{d m}\Gamma(u,\,u)\le (\aplip u)^2$ for all  $ u\in\bd_\loc$,
which completes the proof of Theorem \ref{t2.2}.
\end{proof}

\section{Dirichlet forms: $\frac{d}{dm}\Gamma(u,\,u)\ne (\lip u)^2$}\label{sx2}

This section is a continuation of Section \ref{s2}.
By perturbing the classical
Dirichlet energy form on $\rr^2$,
we construct an example that satisfies the doubling property and a weak Poincar\'e inequality
but so that the  intrinsic length structure does not coincide  with
the gradient structure; see Proposition \ref{p2.3}.  This shows that   doubling   and Poincar\'e  are not enough  to obtain $\frac d{d m}\Gamma(u,\,u)= (\lip\, u )^2$ almost everywhere
(this fact can also be deduced from \cite{s97}; see Remark \ref{r2.x1} below).
Moreover,  the gradient (differential) structure of our perturbed Dirichlet form
does not coincide with the distinguished gradient (differential) structure
of Cheeger; see Proposition \ref{p2.4}.
Notice that the  distinguished differential structure of Cheeger
coincides with the gradient  structure of $\Gamma$
if $(X,\,\mathscr E,\,m)$ further satisfies the Newtonian property;
see Corollary \ref{c2.3}.

Our example is a perturbation  of the classical Dirichlet energy form on $\rr^2$ on a large Cantor set $E$.
Denote by $m$  the Lebesgue measure on $\rr^2$, and denote by $|\cdot-\cdot|$ the Euclidean distance.
The classical Dirichlet form $\mathscr E$ is defined by
$\mathscr E(u,\,u)=\int_{\rr^2}|\nabla u|^2\,dm $
with the domain
$\bd=W^{1,\,2}(\rr^2)$,
where $\nabla$ is the distributional gradient. 
Notice that the Euclidean distance gives the intrinsic distance associated to
$\mathscr E$.
Thus  $(\rr^2,\,\mathscr E,\,m,\,|\cdot-\cdot|)$ satisfies a doubling property,
a weak $(1,\,1)$-Poincar\'e inequality and the Newtonian property.
Moreover, the length structure coincides with the gradient structure, that is,
$|\nabla u|=\aplip u $ almost everywhere for all $u\in W^{1,\,2}(\rr^2)$.

Let $F$ be the Cantor set constructed as follows:
 $I_{i}$ are the two closed intervals obtained by removing the middle open interval with length $1/10$ from $[0,\,1]$
and are  ordered from left to right;
 when $n\ge2$,
$I_{i_1\cdots i_n}$ are the two closed intervals  obtained by removing
  the middle open interval with length $(1/10)^n$ from $I_{ i_1\cdots i_{n-1}}$,
and  are ordered from left to right; $F \equiv\cap_{n\in\nn}\cup_{i_1,\,\cdots,\,i_n}I_{i_1\cdots i_n}$.
Notice that $F$ has positive $1$-dimensional Lebesgue measure.
Set $E\equiv F\times F$.
Then $ \rr^2\setminus E$ is dense in $\rr^2$ and by the Fubini theorem, $m(E)>0$.

Now, for any $\dz\in(0,\,1)$, we define a perturbation $\mathscr E_\dz$ of  $\mathscr E$     by setting
$$\mathscr E_\dz(u,\,u)\equiv\int_{\rr^2}(1-\delta 1_E)|\nabla u|^2 \,dm.$$
It is easy to see that $\mathscr E_\dz$ is a regular, strongly local Dirichlet form with the domain $\bd=W^{1,\,2}(\rr^2)$,
and  for $u\in\bd_\loc$,
$\Gamma_E(u,\,u)=(1-\delta 1_E)|\nabla u|^2 \,m.$
Moreover, let $d_\dz$ be the intrinsic distance defined as in \eqref{e2.x1}.
Then  $$(1-\delta)|\nabla u|^2\le\frac{d}{dm}\Gamma_\dz(u,\,u)\le |\nabla u|^2$$ implies that
$$|x-y|\le d_\dz(x,\,y)\le  \frac1{1-\dz}|x-y|.$$
>From this, it is easy to see that
$(\rr^2,\,\mathscr E_\dz,\,d_\dz,\,m)$ satisfies the doubling property and
a weak $(1,\,1)$-Poincar\'e inequality.
However, the intrinsic length structure does not coincide with the gradient structure when $\dz$ is close to $1$.

\begin{prop}\label{p2.3}
There exists $\dz_E\in(0,\,1)$ such that, for every
  $\dz\in(\dz_E,\,1)$, the intrinsic length structure and the gradient structure of
 $(\rr^2,\,\mathscr E_\dz,\,d_\dz,\,m)$ do  not coincide, that is,
there exists $u\in\bd$ such that
 $\frac{d}{dm}\Gamma(u,\,u)<(\aplip_{d_\dz} u)^2$ on some set of positive measure.
\end{prop}

To prove this, we need the following crucial property.

\begin{lem}\label{l2.6}
There exists a positive constant $C_E>1$ such that
for every pair of $x,\,y\in\rr^2$, we can find a rectifiable curve $\gz$ joining $x$ and $y$
 and satisfying

\noindent(i) $\ell_{\rr^2}(\gz)\le C_E|x-y|$, where $\ell_{\rr^2}(\gz)$ is the length of $\gz$ with respect to the Euclidean distance,

\noindent(ii)  the set $ \gz\cap E$ contains at most $2$ points.
\end{lem}
\begin{proof}
It suffices to consider  all pairs of $x,\,y\in [0,\,1]^2\equiv[0,\,1]\times[0,\,1]$.
Indeed, if both $x$ and $y$ belong to $\rr^2\setminus (0,\,1)^2$, (i) and (ii) obviously hold;
if only one of $x,\,y$ belongs to $(0,\,1)^2$, say $y\in (0,\,1)^2$,
taking $z$ to be the intersection of the boundary of $(0,\,1)^2$ and the interval joining $x$ and $y$,
 and gluing the interval  joining $x,\,z$ and the assumed curve joining $y,\,z$,   we obtain the desired
rectifiable curve.

We claim that for  all pairs of $x,\,y\in[0,\,1]^2\setminus E$,
there exists $\gz$ joining $x$ and $y$ such that
 $\ell_{\rr^2}(\gz)\ls|x-y|$  and   $ \gz\cap E=\emptyset$.
Assume that this claim holds for  the moment.
If $x\in E$ and $y\notin E$, since  $\rr^2\setminus E$ is dense in $\rr^2$,   there exists a sequence $\{x_n\}_{n\in\nn}\subset [0,\,1]^2\setminus E$
of points such that $x_n\to x$ as $n\to\fz$ and $|x-x_n|\le \frac12 |x-x_{n-1}|$ for all $n\ge 1$, where $x_0=y$.
Let $\gz_{n}$ be the assumed rectifiable curve joining $x_{n-1},\,x_n$ for $n\in\nn$.
Set $\gz\equiv(\cup_{n\in\nn}\gz_n)\cup\{x,\,y\}$. Observing that
$$|x_n-x_{n-1}|\le |x-x_{n-1}|+|x-x_n|\le 2|x-x_{n-1}|\le 2^{-n+1} |x-x_0|,$$
we have that
$$\ell_{\rr^2}(\gz)\le \sum_{n\in\nn} \ell_{\rr^2}(\gz_n)\ls
\sum_{n\in\nn} |x_n-x_{n-1}|\ls  \sum_{n\in\nn} 2^{-n }|x-x_0|\ls|x-y|.$$
Obviously, $ \gz_n\cap E=\emptyset$ for each $n\in\nn$ implies that $\gz\cap E$ cointains a single point.
If $x,\,y\in E$, we pick a point  $z$ in the intesection of $\rr^2\setminus E$ and the interval joining $x$ and $y$.
This  reduces us to the case $x\in E$ and $y\notin E$.

Now we prove the above claim. Let $x,\,y\in[0,\,1]^2\setminus E$.
Let $L^{(1)}_x\equiv\{  x+t(x_1,\,0):\ t\in\rr\}$ be the line
parallel to $x_1$-axis and and
$L^{(2)}_x\equiv\{ x+t(0,\,x_2):\ t\in\rr\}$ parallel to $x_2$-axis.
Observe that at least one of $L^{(1)}_x$ and $L^{(2)}_x$ does not intersect  $E$.
Otherwise, if both $L^{(1)}_x$ and $L^{(2)}_x$  intersect    $E$,
then $(x_1+t_1x_1,\,x_2), (x_1,\,x_2+t_2x_2)\in E$ for some $t_1,\,t_2\in\rr$,
and hence, $x_1,\,x_2\in F$.
Thus $x=(x_1,\,x_2)\in E$, which is a contradition.
Similarly,  define $L^{(1)}_y$ and $L^{(2)}_y$ and then
  at least one of $L^{(1)}_y$ and $L^{(2)}_y$ does not intersect   $E$.
If  $L^{(1)}_x$ and $L^{(2)}_y$ do  not intersect $E$,
since $L^{(1)}_x\cap L^{(2)}_y\ne\emptyset,$ there exists a unique
$z\in (L^{(1)}_x\cap L^{(2)}_y)\cap [0,\,1]^2$. Then we take $\gz$ as the union
of the interval joining $x$ and $z$ and
the interval joining $y$ and $z$. Obviously, $\gz$ is as desired.
We reason analogously if $L^{(2)}_x$ and $L^{(1)}_y$ do  not intersect  $E$.
However, it may happen that only $L^{(1)}_x$ and $L^{(1)}_y$ do not intersect
$E$.
In this case, we take $z=(z_1,\,x_2)\in L^{(1)}_x$ such that  $z_1\in[0,\,1]\setminus F$ but $|z_1-x_1|\le |x_1-y_1|/2$.
Notice that the fact  that $L^{(1)}_x$ and $L^{(1)}_y$ do not intersect $E$
implies that $x_2,\,y_2\in [0,\,1]\setminus F$,
and that $z_1,\,x_2\in[0,\,1]\setminus F$ implies that $L_z^{(2)}$ does not intersect   $E$.
Hence $w\equiv (z_1,\,y_2)\in L_z^{(2)}\cap  L_y^{(1)}$ does not belong to  $E$.
The desired rectifiable curve $\gz$ is given by the union of the interval
joining $x,\,z$,
the one joining $z,\,w$  and the one joining $w,\,y$. Indeed, obviously, we
have $\gz\in[0,\,1]\setminus E$, and
moreover, $$\ell_{\rr^2}(\gz)\le |x-z|+|z-w|+|w-y|\le \frac12|x_1-y_1| +|x_2-y_2|+|z_1-y_1|\ls |x-y|.$$
We reason analogously   if $L^{(2)}_x$ and $L^{(2)}_y$ do  not intersect   $E$.
This finishes the proof of Lemma \ref{l2.6}.
\end{proof}

\begin{proof}[Proof of Proposition \ref{p2.3}.]
We first prove that for all $x,\,y\in\rr^2$,
\begin{equation}\label{e2.25}
 d_\dz(x,\,y)\le  { C_E}|x-y|,
\end{equation}
where $C_E$ is the constant from Lemma \ref{l2.6}.
For every $u\in\bd_\loc$ with $\Gamma_\dz(u,\,u)\le m$,
we have $u\in\lip(\rr^2)$ (with respect to the Euclidean distance) and $|\nabla u(x)|\le1$ for almost all $x\in\rr^2\setminus E$.
Thus $u$ is locally $1$-Lipschitz outside of $E$.
For a pair of points $x,\,y\in\rr^2$, let  $\gz$ be a curve as in Lemma \ref{l2.6}
of length at most $C_E|x-y|$.
We conclude that $|u(x)-u(y)|\le C_E|x-y|.$

Let $u$ be a smooth function with compact support. Then $u\in\lip_{d_\dz}(\rr^2)\subset\bd_\loc$.
For every $x\in E\cap(0,\,1)^2$,
 since
$$\limsup_{r\to0}\sup_{{d_\dz}(x,\,y)\le r}
\frac{|u(x)-u(y)-\nabla u(x)\cdot(x-y)|}{{d_\dz}(x,\,y)}=0,$$
we have that
\begin{eqnarray*}
 \lip_{d_\dz} u(x)&&\ge \liminf_{r\to0}\sup_{d_\dz(x,\,y)\le r}\frac{|u(x)-u(y)|}{d_\dz(x,\,y)}
\ge \liminf_{r\to0}\sup_{d_\dz(x,\,y)\le r}\frac{|\nabla u(x)\cdot(x-y)|}{d_\dz(x,\,y)}.
\end{eqnarray*}
Assume that $\nabla u(x)\ne0$. Observe that there exists a sequence $\{y_i\}_{i\in\nn}\subset\rr^2\setminus E$ such that
$ y_i=x+\ez_i\nabla u(x)$  and $\ez_i\to 0$.
Choose $\dz_E\equiv1-\frac1{C_E}$.
Then by \eqref{e2.25} and $\dz\in(1-\frac1{C_E},\,1)$,
\begin{eqnarray*}
 \lip_{d_\dz} u(x)&&
\ge |\nabla u(x)|\liminf_{i\to\fz} \frac{| x-y_i|}{d_\dz(x,\,y_i)}\ge\frac1{ C_E}|\nabla u(x)|
>(1-\dz)|\nabla u(x)|=\frac{d}{dm}\Gamma(u,\,u)(x).
\end{eqnarray*}
Since  $ E\cap(0,\,1)^2$  has positive measure,
if $u$ has non-vanishing gradient   on this set, then
$\frac{d}{dm}\Gamma(u,\,u)<(\lip_{d_\dz} u)^2$ on $ E\cap(0,\,1)^2$ as desired.
\end{proof}

\begin{rem}\label{r2.x1}\rm
It can be also deduced from Sturm \cite{s97} that
doubling  and  Poincar\'e  are not sufficient  to guarantee that
$\frac d{d m}\Gamma(u,\,u)\equiv (\lip\, u )^2$ almost everywhere.
Indeed, Sturm \cite[Theorem 2]{s97} constructed a Dirichlet form
$$\mathscr E_a(u,\,u)=\int_{\rr^2}a(x)|\nabla u(x)|^2\,dm(x),$$
where $a(x)$ satisfies $0<c\le a(x)<1$ for all $x\in\rr^2$,
for which the intrinsic distance $d_a$ is exactly the Euclidean distance.
Our construction is motivated by the Cheeger differential structure below.
\end{rem}

Now we recall the distinguished differential structure constructed by Cheeger \cite {c99}.
Assume that $(X,\,d,\,m)$ satisfies the doubling property and a weak $(1,\,p)$-Poincar\'e inequality
for some $p\in[1,\,\fz)$.    Cheeger \cite[Theorem 4.38]{c99}
 proved the existence of an atlas that    consists of a countable collection
 $\{(U_\az,\,y^\az,\,k(\az))\}_{\az\in\ca}$
 of {\it charts}, where

 \noindent (i) $U_\az$'s
 are measurable sets and $m(X\setminus \cup_\az U_\az)=0$;

  \noindent (ii) $k(\az)\in\nn$, $\sup_{\az\in\ca}k(\az)<\fz$ and if $m(U_\az\cap U_\bz)>0$;
 then $k(\az)=k(\bz)$,

  \noindent  (iii) $y^\az\equiv(y^\az_1,\,\cdots,\,y^\az_{k(\az)}):\ U_\az\to \rr^{k(\az)}$ is Lipschitz;

  \noindent (iv) for every $\az\in\mathcal A$ and $u\in\lip(X)$,
   there exist  $V_\az(u)\subset U_\az$ and a collection
 $$\lf\{\frac{\partial u}{\partial y^\az_i}:\ U_\az\to\rr\r\}_{ 1\le i\le k(\az)}$$
 of bounded Borel measurable functions uniquely determined almost everywhere
 such that
 $m(U_\az\setminus V_\az(u))=0$ and for all $z\in V_\az(u)$,
\begin{equation}\label{e2.26}
  u(w)=u(z)+\sum_{j=1}^{k(\az)}\frac{\partial u}{\partial y^\az_i}(y^\az_j(w)-y_j^\az(z))
 +o(d(w,\,z));
\end{equation}


\noindent (v) if $m(U_\az\cap U_\bz)>0$, then the matrix of  $\frac{\partial y^\az}{\partial y^\bz}$
 is  invertible almost everywhere in $U_\az\cap U_\bz$.

The above atlas yields a bi-Lipschitz invariant
 measurable tangent bundle $TX$ and cotangent bundle $T^\ast X$.
 In fact, for every Lipschitz function $u:\ X\to\rr$,
 its differential $du$ is defined as
 $(\frac{\partial u}{\partial y^\az_1},\,\cdots,\,\frac{\partial u}{\partial y^\az_{k(\az)}})$ and its derivative as
 $Du=\sum_{i=1}^{k(\az)}\frac{\partial u}{\partial y^\az_i}\frac{\partial}{ \partial y^\az_i}$ on each $U_\az$.
Notice that $TX$ is the dual of $T^\ast X$.
 For a Lipschitz function $u$, its derivative
 $Du:\ TX\to\rr$ coincides with its differential $du$ in the sense that
 $$\langle Du(z),\,v\rangle_z\equiv\sum_{i=1}^{k(\az)}\frac{\partial u}{\partial y^\az_i}(z)v_i=du(v)(z)$$
 for every $z\in V_\az(u)\subset U_\az$ and $v=\sum_{i=1}^{k(\az)}v_i\frac{\partial}{ \partial y^\az_i}\in
 T_z X$.
Moreover, for each $z\in V_\az(u)\subset U_\az$, a natural norm  $\|\cdot \|_{T_zX}$ on $T_zX$ is defined by setting
$\|v \|_{T_zX}\equiv\lip_d\lf(\sum_{i=1}^{k(\az)} v_iy_i^\az\r)$
for   $v=\sum_{i=1}^{k(\az)}v_i\frac{\partial}{ \partial y^\az_i}\in
 T_z X$ and hence  $\|Du \|_{T_zX}=\lip_d\lf(\sum_{i=1}^{k(\az)}  \frac{\partial u}{\partial y^\az_i}y_i^\az\r)$
for every Lipschitz function $u$.
Generally, $\|v \|_{T_zX}$ is not Hilbertian.  Cheeger \cite[p. \,460]{c99} introduced
a distinguished inner product norm $ \||v  \||_{T_zX}$ associated to it as follows.

Let $V$ be a $k$-dimensional vector space and $\|\cdot\|$  be a norm on $V$.
Denote by $V^\ast$  the dual space of $V$, endowed with the norm $\|\cdot\| ^\ast$ induced by $\|\cdot\| $.
Then a distinguished inner product norm $ \||\cdot \||^\ast$ on $V^\ast$ is obtained by identifying the functions of $V^\ast$
with their restriction to the unit ball $B_{\|\cdot\|}(0,\,1))$ (with respect to $\|\cdot\| $) and
regarding the functions so obtained as elements of $L^2(B_{\|\cdot\|}(0,\,1),\, (k+1)\frac{{\rm Vol}(k)}{{\rm Vol}(k+2)}H^k_{\|\cdot\| })$.
Here ${\rm Vol}(n)$ denotes the volume of the Euclidean unit ball of $\rr^n$
and $H^k_{\|\cdot\| }$ is the $k$-dimensional Hausdorff measure associated to the metric induced by $\|\cdot\| $.
In other words, for $v^\ast\in V^\ast$, we define
\begin{equation}\label{e2.x30}
\||v^\ast\||^\ast\equiv\lf((k+1)\frac{{\rm Vol}(k)}{{\rm Vol}(k+2)}\int_{B_{\|\cdot\| }(0,\,1)}|v^\ast(v)|^2\,dH^k_{\|\cdot\| }(v)\r)^{1/2}.
\end{equation}
Then the inner product norm $\||\cdot\||$ on $V$
is defined by $\||v\||=\sup_{\||v^\ast\||^\ast\le1}v^\ast(v)$.
Notice that if $\|\cdot\|  $ is an inner product norm,
 then $\||\cdot\||^\ast= \|\cdot\|^\ast$ and $\||\cdot\||= \|\cdot\|$.

Now we have two differential (gradient) structures on $(X,\,\mathscr E,\,d,\,m)$:
the original one of $\Gamma$ induced by $\mathscr E$ and the distinguished one $\||\cdot\||_{TX}$
induced from the intrinsic distance in the sense of Cheeger.
Under some reasonable assuptions, they coincide  as a corollary to
Theorem \ref{t2.2}.

\begin{cor}\label{c2.3}
Suppose that $( X,\,d,\,m)$ satisfies a doubling property,
and that $( X,\,\mathscr E,\,m)$
supports a weak $(1,\,2)$-Poincar\'e inequality and the Newtonian property.
Then for all $u\in\bd$, $\frac {d}{dm}\Gamma(u,\,u)=\||D u\||_{TX}^2=(\aplip_d u)^2$ almost everywhere.
\end{cor}

However, generally, $\frac d{dm}\Gamma$  and $\||\cdot\||^2_{TX}$ do not
necessarily coincide.
This will be illustrated by the above example in Proposition \ref{p2.4} below.
To this  end, notice that since $(\rr^2,\,d_\dz,\,m)$ satisfies the doubling property and
a weak $(1,\,2)$-Poincar\'e inequality,  by \cite{c99}, there exists some atlas.
Up to some change of the coordinate functions,
we can take the atlas with a single chart $U\subset\rr^2$ with $m(\rr^2\setminus U)=0$, and  naturally,
choose $x=(x_1,\,x_2)$ as the coordinate.
Indeed, let  $\{(U_\az,\,y^\az,\,k(\az))\}_{\az\in\ca}$ be the atlas determined by \cite{c99} as above.
We will compare it with the usual coordinates via  the classical Rademacher theorem:
for every $u\in\lip(\rr^2)$ and almost all $z\in\rr^2$,
\begin{equation}\label{e2.xx1}
u(w)=u(z)+\sum_{i=1}^2\frac{\partial u(z)}{\partial x_i}(w_i-z_i)+o(|w-z|).
\end{equation}
Notice that this formula also holds when  $|w-z|$ is  replaced by $d_\dz(w,\,z)$
since the two distances are equivalent.
Now, for $\az\in\mathcal A$, applying \eqref{e2.xx1} to $(y_i^\az)_{i=1}^{k(\az)}$,
we get a Jacobian matrix $\frac{\partial y^\az}{\partial x}\equiv(\frac{\partial y_i^\az}{\partial x_j})_{i=1,\,\cdots,k(\az);\, j=1,\,2}$
almost everywhere;
while applying \eqref{e2.26} to $(x_1,\,x_2)$, we get
$\frac{\partial x}{\partial  y^\az}\equiv(\frac{\partial x_i}{\partial y_j^\az})_{i=1,\,2;\,j=1,\,\cdots,k(\az)}$ on $V_\az(x_1)\cap V_\az(x_2)$.
Then for almost all $z\in U_\az$, we   have
$$w=z+ \frac{\partial x}{\partial  y^\az}(z)\frac{\partial  y^\az}{\partial x}(z) (w-z)+ o(d_\dz(w,\,z)),$$
which implies that
$\frac{\partial x}{\partial  y^\az} \frac{\partial  y^\az}{\partial x}=Id_{2}$ almost everywhere in $U_\az$.
Similarly, $ \frac{\partial  y^\az}{\partial x}\frac{\partial x}{\partial  y^\az}=Id_{k(\az)}$ almost everywhere in $U_\az$.
This implies that $k(\az)=2$,  and $\frac{\partial x}{\partial  y^\az}=(\frac{\partial  y^\az}{\partial x})^{-1}$ almost everywhere.
Therefore on $U_\az$, and hence on $\cup_{\az\in\mathcal A} U_\az$, we can use the uniform coordinate function $x$.

Under the above atlas $\{(U,\,x,\,2)\}$, from the above argument, we also see that
the Cheeger derivative $D_\dz u$ coincides with $\nabla u$ for all Lipschitz functions  $u$, namely,
$D_\dz u=\frac{\partial u}{\partial x_1}\frac{\partial}{\partial x_1}+\frac{\partial u}{\partial x_2}\frac{\partial}{\partial x_2}$.
For almost all $z\in U $ and $v= v_1\frac{\partial}{\partial x_1}+v_2\frac{\partial}{\partial x_2}\in T_z\rr^2$, $\|v\|_{T_z {\rr^2}}=\lip_{d_\dz}(v_1x_1+v_2x_2)$.
But when $\dz$ is close to $1$, the following result
 shows that $ \||D_\dz u\||_{T\rr^2}^2$ does not coincide  with the squared gradient $\frac{d}{dm}\Gamma_\dz(u,\,u)$ and hence,
the distinguished differential structure of Cheeger does not coincide with the original differential structure on
$(\rr^2,\,\mathscr E_\dz,\,d_\dz,\,m)$.

\begin{prop}\label{p2.4}
There exists a $\wz\dz_E\in(0,\,1)$ such that for every $\dz\in(\wz\dz_E,\,1)$,
we can find a function $u\in\lip(\rr^2)$ such that
$\frac{d}{dm}\Gamma_{\dz}(u,\,u)<\||D_\dz u\||^2_{T\rr^2}$ on some set with positive measure.
\end{prop}

To this end, we need the following result, whose notation is that of the paragraphh containing formula \eqref{e2.x30}.

\begin{lem} \label{l2.7}
Assume that $\|\cdot\| $ and $\|\cdot\|_o$ are two norms on $V$
that satisfy $M^{-1}\|\cdot\|  \le \|\cdot\|_o\le M\|\cdot\| $ for some $M\ge1$.
Then there exists a positive constant $C(M,\,k)$ such that
$$\frac1{C(M,\,k)}\||\cdot\||  \le \||\cdot\||_o\le C(M,\,k)\||\cdot\||.$$
\end{lem}

\begin{proof}
We first notice that
$H^k_{\|\cdot\|_o}$ (and also $H^k_{\|\cdot\|}$) is a constant multiple of the Lebesgue measure on $\rr^k$
 due to the its translation invariance.  Then
$H^k_{\|\cdot\|_o}=c_oH^k_{\|\cdot\| }$ for some $c_o>0$.
We claim that $
      M^{-k}\le c_o\le M^k.
      $
  Indeed, recall that  for any set $F$,
its $k$-dimensional Hausdorff measure with respect to the norm $\|\cdot\|_o$ is defined by
 $H^k_{\|\cdot\|_o}(F)\equiv \lim_{\ez\to0+} H^k_{\|\cdot\|_o,\,\ez}(F)$
with $$H^k_{\|\cdot\|_o,\,\ez}(F)\equiv\inf\lf\{\sum_{i}(\diam_{\|\cdot\|_o}U_i)^k\r\},$$
where the infimum is taken over all  covers $\{U_i\}_i$ of $F$ with $\diam_{\|\cdot\|_o}U_i\le\ez.$
Notice  that $M^{-1}\|\cdot\| \le \|\cdot\|_o\le M\|\cdot\| $ implies that
$$M^{-1}\diam_{\|\cdot\| } U\le \diam_{\|\cdot\|_o} U\le M\diam_{\|\cdot\| } U.$$
Then it follows that $H^k_{\|\cdot\|_o,\,\ez}(F)\le M^{k} H^k_{\|\cdot\| ,\,M\ez}(F)$
and hence $H^k_{\|\cdot\|_o}(F)\le M^kH^k_{\|\cdot\| }(F)$. Similarly, $H^k_{\|\cdot\| }(F)\le M^kH^k_{\|\cdot\|_o }(F)$
as desired.

Moreover, observing that
$ B_{\|\cdot\|_o}(0,\,1)\subset B_{\|\cdot\| }(0,\,M)$,  by the scaling property of the
Lebesgue measure and hence of $H^k_{\|\cdot\| }$,
we obtain \begin{eqnarray*}
\||v^\ast\||_o^\ast&&
\le\lf(c_o(k+1)\frac{{\rm Vol}(k)}{{\rm Vol}(k+2)}\int_{B_{\|\cdot\| }(0,\,M)}|v^\ast(v)|^2\,
dH^k_{\|\cdot\| }(v)\r)^{1/2}\\
&&
\le\lf(c_o(k+1)M^{k+2}\frac{{\rm Vol}(k)}{{\rm Vol}(k+2)}\int_{B_{\|\cdot\| }(0,\,1)}|v^\ast(v)|^2\,
dH^k_{\|\cdot\| }(v)\r)^{1/2}\\
&&\le\lf(  c_o M^{k+2}\r)^{1/2}
\||v^\ast\|| ^\ast,
\end{eqnarray*}
which implies that  $\||v \||\le \lf(  c_o M^{k+2}\r)^{1/2}\||v \||_o$.
Similarly, we have $$\||v \||_o\le \lf( \frac1{ c_o} M^{k+2}\r)^{1/2}\||v \||,$$
which finishes the proof of Lemma \ref{l2.7}.
\end{proof}

\begin{proof}[Proof of Proposition \ref{p2.4}.]
For almost every $z\in U$ and $v=v_1\frac{\partial}{\partial x_1}+v_2\frac{\partial}{\partial x_2}\in T_z\rr^2,$
following Cheeger's definition,
 we have
 $\|v\|_{T_z\rr^2}=\lip_{d_\dz}(v_1x_1+v_2x_2)(z).$
Set
 $$\|v\|\equiv|\nabla (v_1x_1+v_2x_2)(z)|=(v_1^2+v_2^2)^{1/2}.$$
 Since \eqref{e2.25} implies that
$\frac1{C_E}|\nabla u(z)|\le\lip_{d_\dz}u(z)\le |\nabla u(z)|$
for $u\in\lip(\rr^2)$,
we have that
$\frac1{C_E}\|v\| \le\|v\|_{T_z\rr^2}\le \|v\|,$
which together with Lemma \ref{l2.7} leads to
$$\frac1{C(C_E,\,2)}\||v\||\le\||v\||_{T_z\rr^2}\le C(C_E,\,2)\||v\||.$$
Notice that $\|\cdot\|=\||\cdot\||$.
We choose $\wz \dz_E\equiv1-\frac1{C(C_E,\,2)}$.
For  every $\dz\in( \wz\dz_E,\,1)$ and $z\in E\cap U$,
we have
\begin{eqnarray*}
&&\frac d{dm}\Gamma_\dz(v_1x_1+v_2x_2,\,v_1x_1+v_2x_2)(z) \\
&&\quad =
(1-\dz)^2|\nabla (v_1x_1+v_2x_2)(z)|^2=(1-\dz)^2
\|v\|^2\\
&&\quad < \||v\||^2_{T_z\rr^2}=
(\lip_{d_\dz}(v_1x_1+v_2x_2)(z))^2.
\end{eqnarray*}
Since the set $E\cap U$ has positive measure, this finishes the proof of Proposition \ref{p2.4}.
\end{proof}

\section{A Sierpinski gasket  with $\frac d{dm}\Gamma(u,\,u)=({\lip}_d\, u)^2$ }\label{s3}

In this section, for the standard (resistance) Dirichlet form on the Sierpinski gasket
 equipped with the Kusuoka measure,
we will identify the intrinsic length structure
with the measurable Riemannian  and the gradient structures;
see Theorem \ref{t3.1} through Theorem \ref{t3.3} below.
We begin with the definition of the Sierpinski gasket $K$.
\begin{defn}\label{d3.1}\rm
Let $V_0\equiv\{p_1,\,p_2,\,p_3\}\in\rr^2$ be the set of the three vertices of an equilateral triangle,
and for $p_i\in V_0$, define $F_i(x)\equiv(x+p_i)/2$ for all $x\in\rr^2$.
The {\it Sierpinski gasket} $K$ is defined as the self-similar set associated with the family of contractions $\{F_i\}_{i=1}^3$,
namely, $K$ is the unique non-empty compact set satisfying $K=\cup_{i=1}^3 F_i(K)$.
\end{defn}

On the Sierpinski gasket $K$, there is a standard  resistance  form $(\mathscr E,\,\mathbb F)$.
Before defining it, we recall the following standard notation and notions.

(i) Let $S\equiv\{1,\,2,\,3\}$.
Set $W_0\equiv S^0=\{\emptyset\}$ and for $n\in\nn$,
 $W_n\equiv S^n=\{i_1i_2\cdots i_n|i_j\in S\}$.
 Let $W_\ast\equiv \cup_{n\in\nn\cup\{0\}}W_n$.
Set  $\Sigma\equiv S^\nn=\{i_1i_2\cdots|i_n\in S\}$
and for $w\in W_n$,  $\Sigma_w\equiv\{v\in\Sigma|v_1\cdots v_n=w\}$.

(ii) For $w=w_1\cdots w_n\in W_n$ with $n\in\nn \cup\{0\}$, define $|w|=n$,
write $F_w\equiv F_{w_1}\circ\cdots \circ F_{w_n}$ if $n\ne1$ and $F_w=Id_K$ if $n=0$,
and set $K_w\equiv F_w(K)$.
For $n\in\nn\cup\{0\}$, $V_n\equiv\cup_{w\in W_n}F_w(V_0)$. $V_\ast=\cup_{n\in\nn\cup\{0\}} V_n$.

(iii) For $w=w_1w_2\cdots\in\Sigma$, define $\pi(w)=\cap_{n\ge1} K_{w_1\cdots w_n}$.
Then $\pi:\Sigma\to K$ is continuous, surjective, and $\sharp(\pi^{-1}(x))=2$ if $x\in\cup_{n=1}^\fz V_n$
and $\sharp\pi^{-1}(x)=1$ otherwise. For $w\in W_\ast$, $\pi(\Sigma_w)=K_w$ and  $\Sigma_w=\pi^{-1}(K_w)$.

For $n\in\nn\cup\{0\}$ and each pair of $u,\,v\in \mathscr C(K)$,
define  $$\mathscr E_n(u,\,v)\equiv \frac18 \lf(\frac53\r)^{n} \sum_{p,\,q\in V_n,\,p\sim q}[u(p)-u(q)][v(p)-v(q)],$$
where $p\sim q$ if and only if $p,\,q\in F_w(V_0)$ for some $w\in W_n$.
Notice that $\mathscr E_n$ is a non-negative definite symmetric quadratic form on $\mathscr C(K)$ and
$\mathscr E_n(u,\,u)\le  \mathscr E_{n+1}(u,\,u)$ for all $u\in \mathscr C(K)$.
Then the following resistance  form  $(\mathscr E,\,\mathbb F)$ is well defined.

\begin{defn}\label{d3.2}Let $\mathbb F\equiv\{u\in\mathscr C(K)| \lim_{n\to\fz}\mathscr E_n(u,\,u)<\fz \}$
and  define $\mathscr E(u,\,v)\equiv\lim_{n\to\fz}\mathscr E_n(u,\,v)$ for all $u,\,v\in\mathbb F$.
\end{defn}

Observe that, associated to the above resistance form $(\mathscr E,\,\mathbb F)$,
the square gradient $\Gamma(u,\,v)$ is well defined by \eqref{e2.x1} as a signed Radon measure on $K$
for each pair of $u,\,v\in\mathbb F$.

Kusuoka \cite{k89} endowed $(K,\,\mathscr E,\,\mathbb F)$ with a  ``Riemannian volume'' measure.
Here we recall its definition via the harmonic embedding $\Phi$ of $K$ into $\rr^2$ see \cite{k93,ka10}.
We say that $h\in\mathbb F$ is  an $E$-{\it harmonic function} for some compact set $E\subset K$
if $\mathscr E(h,\,u)=0$
for all $u\in\mathbb F$ with $u=0$ on $E$.
Let $h_1,\,h_2\in\mathbb F$  be $V_0$-harmonic functions satisfying
$$h_1(p_1)=h_2(p_1)=0,\ h_1(p_2)=h_1(p_3)=1,\ \mbox{and}\ -h_2(p_2)=h_2(p_3)=\frac1{\sqrt 3}.$$
For  the existence of such functions see, for example, \cite[Section 3.2]{k01}.
Observe that by \cite[Theorem 3.6]{k93}, the harmonic embedding $\Phi \equiv(h_1,\,h_2)$ actually induces a
homeomorphism between $K$ and $\Phi(K)$.
Due to this, $\Phi(K)$ is called the harmonic  Sierpinski gasket.
\begin{defn}\label{d3.3}\rm The {\it Kusuoka measure} $m$ on $K$ is  defined by
$ m\equiv \Gamma(h_1,\,h_1)+\Gamma(h_2,\,h_2).
$
\end{defn}

Notice that the Kusuoka measure $m$ is non-atomic and satisfies  $m(U)>0$ for all open nonempty sets $U\subset K$;
see \cite{k89,k93} and below.
Then by \cite[Theorem 3.4.6]{k01}, $(K,\,\mathscr E,\,m)$ is a strongly local, regular Dirichlet
form on $L^2(K,\,m)$ with   domain $\bd=\mathbb F$.
The strong locality  obviously follows  from the definition of $\mathscr E$.
The intrinsic distance $d$ associated to $(K,\,\mathscr E,\,m)$  is then defined as in
\eqref{e2.1}. To distinguish it from the Euclidean distance on $\rr^2$,
we let $B_d(x,\,r)=\{y\in K| d(x,\,y)<r\}$ for $x\in K$ and $r>0$,
and denote by $\lip_d(K)$ the space of Lipschitz functions and
by $\lip_d\, u$ (resp. ${\rm apLip}_d\, u$)
the (resp. approximative) pointwise Lipschitz constant with respect to $d$.

The following result identifies
the intrinsic length structure with the gradient structure on $(K,\,\mathscr E,\,d,\,m)$.

\begin{thm}\label{t3.1}
For every  $u\in \bd$,  the energy measure $\Gamma(u,\,u)$ is absolutely continuous
with respect to the Kusuoka measure $m$ and  $\frac d{dm}\Gamma(u,\,u)=({\lip}_d\, u)^2$ almost everywhere.
\end{thm}

To prove this, we first recall the following properties of $(K,\,\mathscr E,\,d,\,m)$.

\begin{prop}\label{p3.1}
The topology induced by $d$ coincides with the original topology on $K$
 inherited from $\rr^2$,
$(K,\,d,\,m)$ satisfies a doubling property, and
 $(K,\,\mathscr E,\,m)$ supports a weak $(1,\,2)$-Poincar\'e inequality.
\end{prop}

Proposition \ref{p3.1} has been proved in
 \cite[Theorem 6.2]{k08} and \cite[Lemma 3.7 and Proposition 3.20]{ka10}
with the aid of the  dual formula:
for all $x,\,y\in K$,
\begin{equation}\label{e3.1}
 d(x,\,y)=\inf\{\ell_{\rr^2}(\Phi\circ\gz)|\gz: [0,\,1]\to K,\, \gz \  \mbox{is continuous},\,\gz(0)=x,\, \gz(1)=y\},
\end{equation}
where $\ell_{\rr^2}(\Phi\circ\gz)$ denotes the length of $\Phi\circ\gz:[0,\,1]\to\rr^2$
with respect to the Euclidean distance.  Recall that \eqref{e3.1}  is
proved in \cite[Theorem 4.2]{ka10}, and
the right hand side of \eqref{e3.1}
is first introduced in \cite{k08} as the harmonic geodesic metric.
>From \eqref{e3.1}, it easily follows that
\begin{equation}\label{e3.2}
d_\Phi(x,\,y)\equiv|\Phi(x)-\Phi(y)|\le d(x,\,y).
\end{equation}
But, as pointed out in \cite[p.\,800]{k08}, $d_\Phi$ is not comparable to $d$; indeed,
there
exists a double sequence  $\{x_n,\,y_n\}_{n\in\nn}\subset K$ such that
 $ d_\Phi(x_n,\,y_n)/d(x_n,\,y_n) \to0$ as $n\to\fz$.

We also need the following Rademacher theorem on $K$, which is a corollary to Theorem \ref{t3.2} below.

\begin{prop}\label{p3.2}
For every $u\in\bd$,
there exists a unique measurable vector field $\wz\nabla u$   such that
 $\frac d{dm}\Gamma(u,\,u)=|\wz\nabla u|^2$ almost everywhere, and
 for almost all $x\in K$ and  all $y\in K$,
\begin{equation} \label{e3.3}
 |u(y)-u(x)-\wz\nabla u(x)\cdot(\Phi(y)-\Phi(x))|=o(d(x,\,y)).
\end{equation}
\end{prop}

With the help of the results in Section \ref{s2} and the above Proposition \ref{p3.1}, Proposition \ref{p3.2} and \eqref{e3.2},
we now prove Theorem \ref{t3.1}.

\begin{proof}[Proof of Theorem \ref{t3.1}.]
Combining Proposition \ref{p3.1} and Theorem \ref{t2.2}, we have that for all $u\in\bd$,
$\Gamma(u,\,u)$ is absolutely continuous with respect to $ m$ and \eqref{e2.23} holds with some $C_1\ge1$.
By Theorem \ref{t2.2} and Proposition \ref{p2.2} or some density arguments,
the proof of Theorem \ref{t3.1} is reduced to verifying that we may take $C_1=1$.
For almost all $x\in K$ satisfying  \eqref{e3.3}  for all $y$,
applying Proposition \ref{p3.2} and \eqref{e3.2}, we have
\begin{eqnarray*}
 \frac{|u(y)-u(x)|}{d(x,\,y)}
&&\le \frac{|u(y)-u(x)-\wz\nabla u(x)\cdot(\Phi(y)-\Phi(x))|}{d(x,\,y)} +
\frac{|\wz\nabla u(x)\cdot(\Phi(y)-\Phi(x))|}{d(x,\,y)}\\
&&\le o(1)+\lf|\wz\nabla u(x) \r|\frac{|\Phi(y)-\Phi(x) |}{d(x,\,y)}\\
&&\le  o(1)+\lf|\wz\nabla u(x)\r|,
\end{eqnarray*}
which implies that $\lip_du(x)\le|\wz\nabla u(x)|$. This  finishes the proof of Theorem
\ref{t3.1}.
\end{proof}

To
prove Proposition \ref{p3.2}, we  need  several
geometric properties of $(K,\,\mathscr E,\,d,\,m)$.
We first recall the geometric description of $\Phi(K)$;
see \cite{k89,k93} and also \cite{k08,ka10}.
Let $\{T_i\}_{i=1}^{3}$ be the linear
 transformation on $\rr^2$ with the  matrix representations:
$$T_1\equiv\lf(\begin{array}{cc}
            3/5\ & 0\\
0 & 1/5
          \end{array}\r),\ \
T_2\equiv\lf(\begin{array}{cc}
            3/10\ & -\sqrt3/10\\
-\sqrt3/10 & 1/2
          \end{array}\r) \
\mbox{and}\ \
T_3\equiv\lf(\begin{array}{cc}
            3/10\ &  \sqrt3/10\\
 \sqrt3/10 & 1/2
          \end{array}\r).$$
Define $H_i(x)=\Phi(p_i)+T_i(x-\Phi(p_i))$ for all $x\in\rr^2$, $i=1,\,2,\,3$. Then
$\Phi(K)$ is exactly the self-similar set determined by the system $\{H_i\}_{i=1}^3$,
namely, $\Phi(K)=\cup_{i=1}^3H_i(\Phi(K))$.
Moreover, $H_i\circ \Phi=\Phi\circ H_i$ and  $\Phi:\ K\to\Phi(K)$ is a homeomorphism.

We recall the ``Riemannian volume'' $m$ on $\Sigma$ introduced in \cite{k89},
namely, the Kusuoka measure via geometric description.
There exists a unique Borel regular probability measure $m_\Sigma$ on $\Sigma$
such that for all  $w=w_1\cdots w_n\in W_\ast$, $m_\Sigma(\Sigma_w)= (5/3)^n\|T_w\|^2$,
where $T_w=T_{w_1}\circ\cdots\circ T_{w_n}$ and $\|T_w\|$ denotes its Hilbert-Schmidt norm;
see \cite{k89}.
The pushforward measure $\pi_\ast m_\Sigma=m_\Sigma\circ\pi^{-1}  $
is exactly the Kusuoka measure $m$  as in Definition \ref{d3.3}.
Indeed, for $w=w_1 \cdots w_n\in W_\ast$,
\begin{eqnarray}\label{e3.4}
 \pi_\ast m_\Sigma(K_w)&&=m_\Sigma(\pi^{-1}(K_w))=m_\Sigma( \Sigma_w) =\lf(\frac 53\r)^n\|T_w\|^2=
m(K_w);
\end{eqnarray}
see \cite{k93} and also \cite[Proposition 2.14]{ka10}.

Now we collect some further properties, which will be used later.
See \cite{k89,k93,k08,ka10}
for their proofs or details.

\begin{lem}\label{l3.1}
(i) If  $u,\,v\in\bd$, then $u\circ F_i,\,v\circ F_i\in\bd$ for $i=1,\,2,\,3$  and
$$\mathscr E(u,\,v)=\frac 53\sum_{i=1}^3\mathscr E(u\circ F_i ,\,v\circ F_i).$$

(ii)  There exists a constant $C_2\ge1$ such that for  $u \in\bd$,
$$\osc_  K u \le C_2\sqrt{\mathscr E(u,\,u)},$$
where $\osc_  E u \equiv \sup_{x\in E}u(x)-\inf_{x\in E}u(x)$ for any   set $E$.
\end{lem}

For $s\in(0,\,1]$,
denote by $\Lambda(s)$ the collection of all
 $w=w_1 \cdots  w_n\in W_\ast$ such that $\|T_w\|\le s<\|T_{w_1}\circ\cdots\circ T_{w_{n-1}}\|$
when $n\ge2$ and $\|T_w\|\le s$ when $n=1$.
For $x\in K$ and $s\in(0,\,1]$,
set
$$K(x,\,s)\equiv\bigcup_{w\in\Lambda(s),\,x\in K_w}K_w,\ \mbox{and}\
\ U(x,\,s)\equiv\bigcup_{w\in\Lambda(s),\, K_w\cap K(x,\,s)\ne\emptyset}K_w.$$

Then  we have the following results; see \cite{k08} and \cite{ka10}.
\begin{lem}\label{l3.x2}
(i) For all $x\in K$, $s\in(0,\,1]$ and $w\in\Lambda(s)$,
\begin{equation}\label{e3.5}
\sharp\{v\in\Lambda(s)| K_v\cap K(x,\,s)\ne\emptyset\}\le 6\ \
\mbox{and}\ \
\sharp\{v\in\Lambda(s)| K_v\cap K_w\ne\emptyset\}\le 4,
\end{equation}
and that
\begin{equation}\label{e3.6}
B_d(x,\,\sqrt2s/50)\subset U(x,\,s)\subset B_d(x,\,10s).
\end{equation}
(ii) There exists a positive constant $C_3\ge1$ such that
if $w,\,v\in\Lambda(s)$ and $K_w\cap K_v\ne\emptyset$, then
\begin{equation}\label{e3.7}
C_3^{-1}m(K_v)\le m(K_w)\le C_3m(K_v).
\end{equation}
(iii) There exists $C_4\ge 1$ such that for all $w\in W_\ast$ and $i\in\{1,\,2,\,3\}$,
\begin{equation}\label{e3.8}
 C_4^{-1}m(K_w)\le m(K_{wi})\le  m(K_w).
\end{equation}

\end{lem}

Applying the properties above, we  obtain the following results.

\begin{lem}\label{l3.2}
(i) There exists a positive integer $N$ such that for all
$x\in K$ and $s\in (0,\,1)$, if  $w,\,v\in \Lambda(s)$ satisfy  $x\in K_w$
and $K_v\cap K(x,\,s)\ne\emptyset$, then
$\max\{|w|-N,\,0\}\le |v|\le |w|+N$.

(ii) There exists a constant $C_6\ge1$ such that for all
$x\in K$ and $s\in (0,\,1)$, if $w\in \Lambda(s)$ and  $x\in K_w$, then
$m(B_d(x,\,s))\le C_6m(K_w)$.
\end{lem}

\begin{proof} (i) Without loss of generality, we may assume that $K_w\cap K_v\ne\emptyset$.
Indeed, there must exist $\sz\in \Lambda(s)$ such that $K_\sz\cap K_w\ne\emptyset$
and $K_\sz\cap K_v\ne\emptyset$.
By \eqref{e3.4},  $ v\in \Lambda(s)$ and \eqref{e3.8}, we have
$$C_4^{-1}\lf(\frac53\r)^{|v|-1} s^2\le C_4^{-1}\lf(\frac53\r)^{|v|-1} \|T_{v_1\cdots v_{n-1}}\|^2
\le \lf(\frac53\r)^{|v|} \|T_v\|^2\le \lf(\frac53\r)^{|v|}s^2,$$
and the same inequality also holds with $v$ replaced by $w$.
>From this, it is easy to see that $\|T_w\|\le s< \|T_{w_1 \cdots w_{n-1}} \|\ls s $.
Moreover, by \eqref{e3.4}, $w,\,v \in \Lambda(s)$ and \eqref{e3.7}, we also have
$$C_3^{-1}\lf(\frac53\r)^{|v|} \|T_v\|^2\le \lf(\frac53\r)^{|w|} \|T_w\|^2 \le C_3 \lf(\frac53\r)^{|v|} \|T_v\|^2
, $$
which gives that
$  \|T_w\|^2 \sim \lf(\frac53\r)^{|v|-|w|}s^2.$ Hence, $\lf(\frac53\r)^{|v|-|w|}\sim 1$,
which yields (i).

(ii) By the doubling property and \eqref{e3.6},
$$m(B_d(x,\,s))\ls m(B_d(x,\,\sqrt 2s/25))\ls m(U(x,\,s)).$$
Then, by \eqref{e3.5}, it suffices to show that for all $v\in \Lambda(s)$
such that
 $K_v\cap K(x,\,s)\ne\emptyset$, we have $m(K_v)\ls m(K_w)$.
But this follows from \eqref{e3.7} and the fact that
there must exist $\sz\in \Lambda(s)$ such that $K_\sz\cap K_w\ne\emptyset$
and $K_\sz\cap K_v\ne\emptyset$.  This finishes the proof of Lemma \ref{l3.2}.
\end{proof}

\begin{lem}\label{al.3}
Let $h\equiv h_1$ or $h\equiv h_2$. Then for all $u,\,v,\,f,\,g\in \bd$ and for almost all $x\in K$,
\begin{equation}\label{e3.9}
  \frac{d\Gamma(u,\,v)}{d\Gamma(h,\,h)}(x)\frac{d\Gamma(f,\,g)}{d\Gamma(h,\,h)}(x)=\frac{d\Gamma(u,\,g)}{d\Gamma(h,\,h)}(x)
\frac{d\Gamma(f,\,v)}{d\Gamma(h,\,h)}(x),
\end{equation}
and for all $c\in\rr$,
\begin{equation}\label{e3.10}
  \frac{d\Gamma(cu+v,\,g)}{d\Gamma(h,\,h)}(x)=
c\frac{d\Gamma(u,\,g)}{d\Gamma(h,\,h)}(x)+\frac{d\Gamma( v,\,g)}{d\Gamma(h,\,h)}(x).
\end{equation}
Moreover, \eqref{e3.9} and \eqref{e3.10} also hold with $\Gamma(h,\,h)$ replaced by $m$.
\end{lem}

\begin{proof}
By \cite[Theorem 5.6]{h10}, for every $u\in\bd$, $\Gamma(u,\,u)$ is absolutely continuous with respect to
$\Gamma(h,\,h)$.  Thus $\Gamma(h ,\,h )$ and $m$ are mutually absolutely continuous,
and moreover, for every $u\in \mathbb F$, by \eqref{e2.x1} and the Cauchy-Schwarz inequality,
both $\Gamma(u,\,h )$ and $\Gamma(u,\,u) $
are absolutely continuous with respect to $m$ and $\Gamma(h ,\,h )$.
Let $\{f_i\}_{i\in\nn}$ be an arbitrary complete orthonormal system of $\bd$.
By \cite[Proposition 2.12]{h10} and the fact that the index of $(K,\,\mathscr E)$ is $1$ (see \cite[Section 4]{h10}),
there exists a sequence of functions, $\{\zeta^i\}_{i\in\nn}$, such that for all $i,\,j\in\nn$, $\frac {d\Gamma(f_i,\,f_j)}{d\Gamma(h,\,h)}=\zeta^i\zeta^j$
almost everywhere. Recall that $u$ has a unique representation
$u=\sum_{i\in\nn}a_i(u) f_i$ with $\sum_{i\in\nn}a_i(u)^2<\fz$.
Then $\gz(u)\equiv\sum_{i\in\nn}a_i(u)\zeta^i\in L^2(K,\,\Gamma(h ,\,h ))$
is well defined. Indeed, for $u,\,v\in\bd$,
\begin{eqnarray*}
 \lf(\sum_{i=1}^ N a_i(u)\zeta^i\r)\lf(\sum_{j=1}^ Na_j(v)\zeta^j\r)
&&=\sum_{i,\,j=1}^Na_i(u)a_j(v)\frac {d\Gamma(f_i,\,f_j)}{d\Gamma(h,\,h)}\\
&&=\frac{d\Gamma(\sum_{i=1}^ Na_i(u)f_i,\,\sum_{j=1}^Na_j(v)f_j )}{d\Gamma(h,\,h)} \to
\frac{d\Gamma(u,\,v)}{d\Gamma(h,\,h)}
\end{eqnarray*}
as $N\to\fz$ in $L^2(K,\,\Gamma(h,\,h))$ and hence almost everywhere. Moreover, from this, we deduce that
$\gz(u)\gz(v)=\frac{d\Gamma(u,\,v)}{d\Gamma(h,\,h)}$ almost everywhere,
which implies \eqref{e3.9}.

The linearity property  $a_i(cu+v)=ca_i(u)+a_i(v)$ implies linearity of $\gz$
and hence also \eqref{e3.10}
by repeating the above argument.

Since $\Gamma(h ,\,h )$ and $m$ are mutually absolutely continuous, \eqref{e3.9} (resp. \eqref{e3.10})
with $\Gamma(h,\,h)$ replaced by $m$ follows from the Radon-Nikodym theorem and \eqref{e3.9} (resp. \eqref{e3.10}).
Indeed, for almost all $x\in X$,
\begin{eqnarray*}
  \frac{d\Gamma(u,\,v)}{dm}(x)&&=\lim_{r\to0}   \frac{ \int_X 1_{B(x,\,r)}\,d\Gamma(u,\,v)}{\int_X 1_{B(x,\,r)}\,  dm} \\
&&= \lim_{r\to0}   \frac{ \int_X 1_{B(x,\,r)}d \,\Gamma(u,\,v)}
{\int_X 1_{B(x,\,r)}\, d\Gamma(h,\,h)}\cdot \lim_{r\to0} \frac{ \int_X 1_{B(x,\,r)}d \,\Gamma(h,\,h)}
{\int_X 1_{B(x,\,r)}\, dm}\\
&&=
\frac{d\Gamma(u,\,v)} {d\Gamma(h,\,h)} (x) \frac{d\Gamma(h,\,h)}{dm}(x).
\end{eqnarray*}
This finishes the proof of Lemma \ref{al.3}.
\end{proof}

\begin{rem}\rm
The set of points $x\in K$ for which \eqref{e3.10} holds  is independent of $c\in\rr$.
\end{rem}

The following Rademacher theorem is an improvement
on   \cite[Theorem 5.4]{h10} and \cite[Theorem 2.17 (ii)]{ka10}.

\begin{prop}\label{p3.3}
Let $h\equiv h_1$ or $h\equiv h_2$. For every $u\in\bd$,
there exists a unique measurable function $\frac{du}{dh}$   such that
 for almost all $x\in K$ and  all $y\in K$,
\begin{equation}\label{e3.11}
 |u(y)-u(x)-\frac{du(x)}{dh}(h(y)-h(x))|=o(d(x,\,y)).
\end{equation}
Moreover,  $\frac {d\Gamma(u,\,u)}{d\Gamma(h,\,h)}=|\frac{du}{dh}|^2$ almost everywhere.
\end{prop}

\begin{proof}[Proof of Proposition \ref{p3.3}]
Set
$$\frac{du }{dh}\equiv\frac{d\Gamma(u,\,h)}{d\Gamma(h,\,h)},$$
and
$$R_x(y)\equiv u(y)-u(x)-\frac{du(x)}{dh} (h(y) -h(x))$$
for all $y\in K$  whenever $\frac{du(x)}{dh}$ exists.
Then $R_x(\cdot)\in\mathbb \bd$ and
  \eqref{e3.9} implies that
\begin{equation}\label{e3.12}
\frac{d\Gamma( u  ,\,u ) }{d\Gamma(h,\,h)}(x)=  \frac{d\Gamma( u  ,\,u ) }{d\Gamma(h,\,h)}(x)
 \frac{d\Gamma( h  ,\,h )}{d\Gamma(h,\,h)}(x)= \lf(\frac{d\Gamma(u,\,h)}{d\Gamma(h,\,h)}(x)\r)^2= \lf(\frac{du(x)}{dh}\r)^2.
\end{equation}

Now it suffices to  prove that
for almost all  $x\in K$ and all $s\in(0,\,1)$,
\begin{equation}\label{e3.13}
\sup_{y\in B_d(x,\,s)}|u(y)-u(x)-\frac {du(x)}{dh}(h(y)-h(x))|=o(s).
\end{equation}
Recall that for all $x\in K$ and $s\in(0,\,1/10]$, $B_d(x,\,s)\subset U(x,\,10s)$; see \eqref{e3.6}.
Therefore, for each $y\in B_d(x,\,s)$,
there exist  $w,\,v\in \Lambda(10s)$ such that $x\in K_w$,
$K_w\cap K_v\ne\emptyset$ and $y\in K_v$,
where $w$ and $v$ may be equal. Taking $y_\ast\in K_w\cap K_v=F_w(V_0)\cap F_v(V_0)$ and using $R_x(x)=0$,
we obtain
\begin{eqnarray*}
|R_x(y)|&&\le |R_x(y)-R_x(y_\ast)|+ |R_x(y_\ast)-R_x(x)|\\
&&\le   \dosc _{ K_w}R_x +\dosc _{ K_v} R_x  =   \dosc _{ K}R_x\circ F_w +\dosc_{ K } R_x\circ F_v .
\end{eqnarray*}
By Lemma \ref{l3.1},  we have
\begin{eqnarray*}
|R_x(y)| &&\ls \sqrt{\Gamma(R_x\circ F_w,\,R_x\circ F_w)(K )}+\sqrt{\Gamma(R_x\circ F_v,\,R_x\circ F_v)(K )}\\
&&\ls \sqrt{\lf(\frac35\r)^{|w|}\Gamma(R_x,\,R_x)(K_w)}+\sqrt{\lf(\frac35\r)^{|v|}\Gamma(R_x,\,R_x)(K_v)}.
\end{eqnarray*}
Noticing that $|v|\sim|w|$  by Lemma \ref{l3.2} and
$U(x,\,10s)\subset B_d(x,\,500s/\sqrt2)$ by \eqref{e3.6}, for $s\in(0,\,1/500]$, we have
$$\sup_{y\in B_d(x,\,s)}|R_x(y)|\ls \sqrt{\lf(\frac35\r)^{|w|}\Gamma(R_x,\,R_x)(B_d(x,\,500s/\sqrt2))}.$$
On the other hand, for $s\in(0,\,1)$,
\begin{eqnarray*}
\frac{ \Gamma(R_x,\,R_x)(B_d(x,\,s))}{ \Gamma(h,\,h)(B_d(x,\,s))}
&&=\frac{\Gamma( u-\frac{du(x)}{dh} h ,\,u-\frac{du(x)}{dh}h)(B_d(x,\,s))}{\Gamma(h,\,h)(B_d(x,\,s))} \\
&&=\frac{\Gamma( u  ,\,u )(B_d(x,\,s))}{\Gamma(h,\,h)(B_d(x,\,s))}-2\frac{du(x)}{dh}
\frac{\Gamma(u,\,h)(B_d(x,\,s))}{\Gamma(h,\,h)(B_d(x,\,s))}+ \lf(\frac{du(x)}{dh}\r)^2.
\end{eqnarray*}
By this, \eqref{e3.12}, the definition of $\frac{du}{dh}$ and the Radon-Nikodym theorem, we conclude that
\begin{eqnarray*}
\lim_{s\to0}\frac{ \Gamma(R_x,\,R_x)(B_d(x,\,s))}{\Gamma(h,\,h)(B_d(x,\,s))}&&= \frac{d\Gamma( u  ,\,u )(x)}{d\Gamma(h,\,h) }-2\frac{du(x)}{dh}
\frac{d\Gamma(u,\,h)(x)}{d\Gamma(h,\,h)}+ \lf(\frac{du(x)}{dh}\r)^2  =0
\end{eqnarray*}
for almost all $x\in K$.
Therefore,
\begin{equation}\label{e3.14}
\sup_{y\in B_d(x,\,s)}|R_x(y)|=o\lf(\sqrt{\lf(\frac35\r)^{|w|}\Gamma(h,\,h)(B_d(x,\,500s/\sqrt2))}\r).
\end{equation}
Observe that, by the doubling property, definition of $m$  and   Lemma \ref{l3.1} (ii)
and Lemma \ref{l3.2},  we have
$$\Gamma(h,\,h)(B_d(x,\,500s/\sqrt2))\le m(B_d(x,\,500s/\sqrt2))\ls m(B_d(x,\,10s))\ls m(K_w).$$
Thus by \eqref{e3.4} and $w\in\Lambda(10s)$, we arrive at
$$\sup_{y\in B_d(x,\,s)}|R_x(y)|=o\lf(\sqrt{\lf(\frac35\r)^{|w|}m(K_w)}\r)=o(  \|T_w\|)=o(s),$$
 as desired.

To see the uniqueness, assume that $ a$ is a measurable function such that
\eqref{e3.13} holds with $\frac{du(x)}{dh}$
replaced by $ a(x)$, for almost all $x\in K$.
Then \begin{equation}\label{e3.15}
\sup_{y\in B_d(x,\,s)}\lf|  \frac {du(x)}{dh}-a(x) \r| |h(y)-h(x)|=o(s).
\end{equation}
Take $x=\pi^{-1}(w)\in K_w$ such the above holds, and
for $s\in(0,\,1)$, and  $w_1\cdots w_{n_s}\in\Lambda(\sqrt s/25)$.
Observe that \eqref{e3.6} gives  $K_{w_1\cdots w_{n_s}}\subset U(x,\,\sqrt2s/25)\subset B_d(x,\,s )$.
Since
\begin{eqnarray*}
 \sup_{y\in B_d(x,\,s)}| h(y)-h(x) |&&\ge\frac12\osc_{ B_d(x,\,s)}h \\
&&\ge
\frac12\osc_{ K_{w_1\cdots w_{n_s}}}h \ge\frac12\mathscr E(h\circ F_{w_1\cdots w_{n_s}},\,h\circ F_{w_1\cdots w_{n_s}}),
\end{eqnarray*}
it suffices to show that
$\mathscr E(h\circ F_{w_1\cdots w_{n_s}},\,h\circ F_{w_1\cdots w_{n_s}})\gs s$.
Indeed, this would imply that $\frac {du(x)}{dh}-a(x)=0$.

By the martingale convergence theorem  and the mutual absolute
continuity of $m$ and $\Gamma(h,\,h)$,
for almost all $x\in K$, we have
\begin{eqnarray*}
 \lim_{s\to0}
\frac{\mathscr E(h\circ F_{w_1\cdots w_{n_s}},\,h\circ F_{w_1\cdots w_{n_s}})}{m(K_{w_1\cdots w_{n_s}})}
&& = \lim_{s\to0}\frac{\Gamma(h\circ F_{w_1\cdots w_{n_s}},\,h\circ F_{w_1\cdots w_{n_s}})(K)}{m(K_{w_1\cdots w_{n_s}})}\\
&&=   \lim_{s\to0}
\frac{\Gamma( h ,\,h)(K_{w_1\cdots w_{n_s}})}{m(K_{w_1\cdots w_{n_s}})}=\frac{d\Gamma(h,\,h)}{dm}(x)>0,
\end{eqnarray*}
which together with \eqref{e3.8} and $w_1\cdots w_{n_s}\in\Lambda(\sqrt2 s/25)$
implies that $$\mathscr E(h\circ F_{w_1\cdots w_{n_s}},\,h\circ F_{w_1\cdots w_{n_s}})\gs m( K_{w_1\cdots w_{n_s}})\gs
m( K_{w_1\cdots w_{n_s-1}})\gs s.$$
This finishes the proof of Proposition \ref{p3.3}.
\end{proof}

Now, we are going to identify the intrinsic length structure
with the ``measurable Riemannian structure'' of Kusuoka \cite{k89}  in Theorem \ref{t3.2} below.

Recall that Kusuoka further introduced the ``measurable Riemannian metric'' $Z$ on $K$.
Indeed,  for $w\in W_\ast$, set $Z_m(w)\equiv T_wT_w^\ast/\|T_w\|^2$.
Kusuoka \cite{k89} proved that
\begin{equation}\label{e3.16}
 Z(w)\equiv\lim_{n\to\fz} Z_m(w_1\cdots w_n)
\end{equation}
  exists for almost all $w=w_1w_2\cdots\in\Sigma$,
and moreover, rank $Z(w)=1$ and $Z$ is orthogonal projection onto its image for almost all $w\in\Sigma$.
Since $m_\Sigma(V_\ast)=0$, the pushforward mapping $\pi_\ast Z=Z\circ\pi^{-1}$, which is still denote by $Z$
by abuse of notation, is well defined on $K$.

The above measurable Riemannian structure is identified with the   gradient structure in the following sense;
see \cite{k89,k93} and \cite[Theorem 4.8]{k08}.
Let $\mathscr C^1(K)\equiv\{v\circ\Phi| v\in\mathscr C^1(\rr^2)\}$.
Then $\mathscr C^1(K)$ is dense in $(\bd,\,\|\cdot\|_\bd)$. Moreover, for every $u\in\mathscr C^1(K)$,
define $\nabla u\equiv (\nabla v)\circ\Phi$,  which is well defined
since it is independent of the choice of $v\in\mathscr C^1(\rr^2)$ with
$u=v\circ\Phi$; see \cite[Section 4]{k93}.
Here $\nabla v(x)\equiv(\frac{\partial v(x)}{\partial x_1},\, \frac{\partial v(x)}{\partial x_2})$ denotes the usual gradient
for $v\in\mathscr C^1(\rr^2)$.

\begin{prop}\label{p3.4}
For every $u\in \bd$,
there exists a measurable vector field $Y(u)\in\rr^2$ such that $Y(u) \in {\rm Im\,}Z$,
 $\frac d{dm}\Gamma(u,\,u) =|Y(u) |^2$ almost everywhere, and hence
$\mathscr E(u,\,u)=\int_K |Y(u)(x)|^2 \,dm(x).$
Moreover,   if $u\in\mathscr C^1(K)$, then $Y(u) =Z\nabla u $.
\end{prop}

Applying Propositions \ref{p3.3} and \ref{p3.4}, we have a formula for the projection $Z$ via the
harmonic embedding (or coordinate) $\Phi$.
\begin{lem}\label{l3.5}
The pushforward $\pi_\ast Z$ to $K$ of the projection $Z$ on $\Sigma$ as in \eqref{e3.16} is given by
$$
\pi_\ast Z=\lf(\begin{array}{cl}
           \frac1{1+a^2}\quad&  \frac a{1+a^2}\\
 \frac a{1+a^2}\quad&\frac{a^2}{1+a^2}
          \end{array}\r)
$$
almost everywhere, where  $a=\frac{dh_2}{dh_1}=\frac{d\Gamma(h_2,\,h_1)}{d\Gamma(h_1,\,h_1)}$.
The eigenvalues of $\pi_\ast Z$ are $\lz_1=0$ and $\lz_2=a^2+1$, the corresponding eigenvectors are
$\xi_1=(-\frac a{1+a^2},\,\frac 1{1+a^2})$ and $\xi_2=(\frac 1{1+a^2},\,\frac a{1+a^2})$.
The projection space is ${\rm Im} Z=(\frac 1{1+a^2},\,\frac a{1+a^2})\rr$.
\end{lem}

\begin{proof}
By the Radon-Nikodym theorem and Proposition \ref{p3.3}, we have
$$
a^2=\frac{d\Gamma(h_2,\,h_1)}{d\Gamma(h_1,\,h_1)}\frac{d\Gamma(h_2,\,h_1)}{d\Gamma(h_1,\,h_1)}=
\frac{d\Gamma(h_1,\,h_1)}{d\Gamma(h_1,\,h_1)}\frac{d\Gamma(h_2,\,h_2)}{d\Gamma(h_1,\,h_1)}
=\frac{d\Gamma(h_2,\,h_2)}{d\Gamma(h_1,\,h_1)}$$
and
$$ \frac{dm}{d\Gamma(h_1,\,h_1)}=\frac{d\Gamma(h_1,\,h_1)}{d\Gamma(h_1,\,h_1)}+
\frac{d\Gamma(h_2,\,h_2)}{d\Gamma(h_1,\,h_1)}=1+a^2.
 $$
Hence by Proposition \ref{p3.4} and the Radon-Nikodym theorem,
$$Ze_i\cdot e_j=Z\nabla h_i\cdot\nabla h_j= \frac{d\Gamma(h_1,\,h_1)}{dm}=\frac{d\Gamma(h_i,\,h_j)}{d\Gamma(h_1,\,h_1)}\frac{d\Gamma(h_1,\,h_1)}{dm}
=\frac1{1+a^2}\frac{d\Gamma(h_i,\,h_j)}{d\Gamma(h_1,\,h_1)},
$$
which implies that
 $$Ze_1\cdot e_1=\frac1{1+a^2},\ \  Ze_1\cdot e_2=Ze_2\cdot e_1= \frac a{1+a^2}\ \
\mbox{and}\ \ Ze_2\cdot e_2=\frac{a^2}{1+a^2}.$$
The other conclusions follow from  this by  standard computations.
This finishes the proof of Lemma \ref{l3.5}.
\end{proof}

Now we improve Proposition \ref{p3.4} and \cite[Theorem 2.17 (i)]{k10} as follows.
Notice that Proposition \ref{p3.2}  follows from Theorem \ref{t3.2}.
\begin{thm}\label{t3.2}
For every   $u\in\bd$, there exists a unique measurable vector field $\wz\nabla u$
such that for almost all $x\in K$,
and for all $s\in(0,\,1)$,
\begin{equation} \label{e3.17}
\sup_{y\in B_d(x,\,s)}|u(y)-u(x)- \wz\nabla u(x) \cdot(\Phi(y)-\Phi(x))|=o(s).
\end{equation}
Moreover, $\wz\nabla u=Y(u)\in {\rm Im}Z$ and
$\frac d{dm}\Gamma(u,\,u) =|Y(u) |^2=|\wz\nabla u |^2 $
almost everywhere; in particular, if $u\in\mathscr C^1(K)$,
then $\wz\nabla u=Y(u)=Z\nabla u$ almost everywhere.
\end{thm}

\begin{proof}
We first observe that, by Proposition \ref{p3.4}, $\nabla h_1=e_1$, and
mutual  absolute continuity of $\Gamma(h_1,\,h_1)$ and $m$,
we obtain $$ |Z(x)e_1|^2=|Z(x)\nabla h_1(x)|^2=\frac{d\Gamma(h_1,\,h_1)}{dm}(x) >0$$
for almost all $x\in K$. Similarly, we have that $ |Z(x)e_2|^2>0$ for almost all $x\in K$.
For such an $x\in K$, let $\zeta\equiv (\zeta _1,\,\zeta _2)=\zeta_1e_1+\zeta_2e_2$
with   $\zeta_1\equiv|Ze_1|$ and
$\zeta_2\equiv Ze_1\cdot e_2/ |Ze_1|$.
Take $\wz\nabla u \equiv\frac{du }{dh_1}\zeta_1 \zeta $.
Then for almost all $x\in K$, obviously, $\wz\nabla u(x)\in {\rm Im}Z(x)$.
Since $|\zeta(x)|^2=1$ and $(\zeta_1(x))^2=\frac{d\Gamma(h_1,\,h_1)}{dm}(x)$, applying Lemma \ref{al.3}
and \eqref{e3.12},
we further have
$$|\wz\nabla u(x)|^2= \lf(\frac{du(x)}{dh_1} \r)^2 \frac{d\Gamma(h_1,\,h_1)}{dm}(x)= \frac{d\Gamma(u,\,u)}{d\Gamma(h_1,\,h_1)}  (x) \frac{d\Gamma(h_1,\,h_1)}{dm}(x)=\frac{d\Gamma(u,\,u)}{dm}(x).
 $$

Whenever $\wz\nabla u(x) $ exists,
write
\begin{eqnarray*}
&&u(y)-u(x)- \wz\nabla u(x) \cdot(\Phi(y)-\Phi(x))\\
&&\quad= (\zeta_1(x))^2[u(y)-u(x)- \frac{du(x)}{dh_1}(h_1(y)-h_1(x))]\\
&&\quad\quad+  [(1-(\zeta_1(x))^2)(u(y)-u(x))-\zeta_1(x) \zeta_2(x)\frac{du(x)}{dh_1}(h_2(y)-h_2(x))]\\
&&\quad\equiv\wz R^{(1)}_x(y)+\wz R^{(2)}_x(y)
\end{eqnarray*}
for  all $y\in K$.
Observe that Proposition \ref{p3.3} implies that
$\sup_{y\in B_d(x,\,s)}|\wz R^{(1)}_x(y)|=o(s)$.
Then \eqref{e3.17} is reduced to proving
$\sup_{y\in B_d(x,\,s)}|\wz R^{(2)}_x(y)|=o(s)$.
To this end, observe that by Proposition \ref{p3.4} and the Radon-Nikodym theorem,
$$1-(\zeta_1(x))^2=1- \frac{d\Gamma(h_1,\,h_1)}{dm}(x)= \frac{d\Gamma(h_2,\,h_2)}{dm}(x)$$
and
$$\zeta_1(x)\zeta_2(x)=Z\nabla h_1(x)\cdot\nabla h_2(x)=\frac{d\Gamma(h_1,\,h_2)}{dm}(x)=\frac{d\Gamma(h_1,\,h_1)}{dm}(x)
\frac{d\Gamma(h_1,\,h_2)}{d\Gamma(h_1,\,h_1)}(x)
$$
for almost all $x\in K$.
Also by Proposition \ref{p3.3} and the Radon-Nikodym theorem, for almost all $x\in K$,
\begin{eqnarray*}
\frac{du}{dh_1}(x)&&=\frac{d\Gamma(u,\,h_1)}{d\Gamma(h_1,\,h_1)}(x)\\
&&
= \frac{d\Gamma(u,\,h_1)}{d\Gamma(h_2,\,h_2)}(x)\frac {d\Gamma(h_2,\,h_2)}{d\Gamma(h_2,\,h_2)}(x)\frac {d\Gamma(h_2,\,h_2)}{d\Gamma(h_1,\,h_1)}(x)\\
&& = \frac{d\Gamma(u,\,h_2)}{d\Gamma(h_2,\,h_2)}(x)\frac {d\Gamma(h_1,\,h_2)}{d\Gamma(h_2,\,h_2)}(x)\frac {d\Gamma(h_2,\,h_2)}{d\Gamma(h_1,\,h_1)}(x)\\
&&= \frac{d\Gamma(u,\,h_2)}{d\Gamma(h_2,\,h_2)}(x)\frac {d\Gamma(h_1,\,h_2)} {d\Gamma(h_1,\,h_1)}(x)\\
&&=\frac{du}{dh_2}(x)\frac{dh_2}{dh_1}(x).
\end{eqnarray*}
Therefore,
\begin{eqnarray*}
 \zeta_1(x)\zeta_2(x)\frac{du}{dh_1}(x)
&& = \frac{d\Gamma(h_1,\,h_1)}{dm}(x)  \frac{d\Gamma(h_1,\,h_2)}{d\Gamma(h_1,\,h_1)}(x)\frac {d\Gamma(h_1,\,h_2)} {d\Gamma(h_1,\,h_1)}(x)\frac{d\Gamma(u,\,h_2)}{d\Gamma(h_2,\,h_2)}(x)\\
&& = \frac{d\Gamma(h_1,\,h_1)}{dm}(x)  \frac{d\Gamma(h_2,\,h_2)}{d\Gamma(h_1,\,h_1)}(x) \frac{d\Gamma(u,\,h_2)}{d\Gamma(h_2,\,h_2)}(x)\\
&& = \frac{d\Gamma(h_2,\,h_2)} {dm}(x) \frac{d\Gamma(u,\,h_2)}{d\Gamma(h_2,\,h_2)}(x) \\
&&= (1-(\zeta_1(x))^2) \frac{du}{dh_2}(x).
\end{eqnarray*}
Thus
\begin{eqnarray*}
\wz R^{(2)}_x(y) &&= (1-(\zeta_1(x))^2)[u(y)-u(x)-\frac{du(x)}{dh_2}(h_2(y)-h_2(x))].
\end{eqnarray*}
and hence by Proposition \ref{p3.3},  $\sup_{y\in B_d(x,\,s)}|\wz R^{(2)}_x(y)|=o(s)$ as desired.

The uniqueness of $\wz \nabla u$ follows from  exactly the same argument that was used in the proof of Proposition \ref{p3.3}.
\end{proof}

We also extend Theorems \ref{t3.1} and \ref{t3.2} to the case of
an energy measure $\Gamma(h,\,h)$,
where $h$ is a nontrivial $V_0$-harmonic function with $\mathscr E(h,\,h)=1$.
Notice that $\Gamma(h,\,h)$ and $m$ are mutually absolutely continuous.
Recall that $(K,\,\mathscr E,\, \Gamma(h,\,h))$ is a strongly local Dirichlet form on
$L^2(K,\,\Gamma(h,\,h))$ with domain $\bd=\mathbb F$.
Denote by $d_h$ the associated intrinsic distance. Proposition \ref{p3.1} and Proposition \ref{p3.3} still hold for
$(K,\,\mathscr E,\, \Gamma(h,\,h),\,d_h)$ and a dual formula similar to \eqref{e3.1} is still available.
For the above see \cite{k93,k08,ka10}.
The following result identifies
the length structure with the length of the gradient on $(K,\,\mathscr E,\, \Gamma(h,\,h),\,d_h)$.

\begin{thm}\label{t3.3}
For every  $u\in \bd$, the energy measure $\Gamma(u,\,u)$ is absolutely continuous
with respect to the Kusuoka measure $\Gamma(h,\,h)$ and the square of
the length of the gradient  satisfies
$ \frac {d\Gamma(u,\,u)}{d\Gamma(h,\,h)}=({\lip}_{d_h}  u)^2$ almost everywhere.
Moreover,  for almost all $x\in K$ and  all $y\in K$,
$$
 |u(y)-u(x)-\frac{du(x)}{dh}(h(y)-h(x))|=o(d_h(x,\,y)).
$$
where $\frac{du }{dh}=\frac {d\Gamma(u,\,h)}{d\Gamma(h,\,h)}$
and $|\frac{du }{dh}|^2=\frac {d\Gamma(u,\,u)}{d\Gamma(h,\,h)}$.
\end{thm}

We point out that Theorem \ref{t3.3}  can be proved by repeating the above arguments as in
 Theorems \ref{t3.1} and \ref{t3.2}, and all the properties needed in the arguments
are available by \cite{k93,k08,ka10}. We omit the details.

\section{Heat flow, gradient flow and $\frac d{dm}\Gamma(u,\,u)=({\lip} \, u)^2$ }\label{s4}

In this section, under the Ricci curvature bounds
of Lott-Sturm-Villani, we will clarify the relations between the coincidence
of the intrinsic length structure and the gradient structure  and
the identification of the heat flow of $\mathscr E$ and the gradient flow of entropy;
 see Theorem \ref{t4.1} and \ref{t4.2}. We begin with the definition of Wasserstein distance.

On a given  metric space  $( X,\,d)$,
denote by $\mathscr P( X)$ the {\it collection of all Borel  probability measures} on $ X$ and
endow it with weak $\ast$-topology, that is,  $\mu_i\to\mu$ if and only if
for all $f\in\mathscr C( X)$, $\int_ X f\,d\mu_i\to\int_ X f\,d\mu$.
For $p\in[1,\,\fz)$, denote by $\mathscr P_p(X)$ the collection of all measures $\mu\in\mathscr P(X)$
such that $\int_X d^p(x_1,\,x)\,d\mu(x)<\fz$.
Moreover, for every pair of $\mu,\,\nu\in\mathscr P( X)$, define the {\it  $L^p$-Wasserstein distance} as
\begin{equation}\label{e4.1}
 W_p(\mu,\,\nu)\equiv\inf_\pi\lf(\int_{ X\times X}[d(x,\,y)]^p\,d\pi(x,\,y)\r)^{1/p},
\end{equation}
where the infimum is taken over all couplings $\pi$ of $\mu$ and $\nu$.
Recall that a {\it  coupling }$\pi$ of $\mu$ and $\nu$
is a probability measure $\pi\in\mathscr P( X\times X)$
with the property that for all measurable sets $A\subset X$, $\pi(A\times X)=\mu(A)$ and $\pi( X\times A)=\nu(A)$.
There always exists (at least)
one optimal coupling, and so the above infimum can be replaced by minimum; see for example \cite[Proposition 2.1]{v09}.

In the rest of this section, we always assume that $X$ is   compact,
$\mathscr E$ is a regular, strongly local Dirichlet form
on $X$ and $m$ is a probability measure, namely, $m( X)=1$.
Let $d$ be the associated intrinsic distance as in \eqref{e2.1} and assume that
the topology induced by $d$ coincides with the original topology on $X$.
Then $(X,\,d)$ is a compact length space   by \cite{s94,s98b},
and hence, $\mathscr P_2( X)= \mathscr P( X)$ equipped with the distance $ W_2$
is a compact length space (hence a geodesic space); see \cite{lv09}.
Notice that the topology induced by $ W_2$
coincides with the above weak $\ast$-topology (see for example \cite{v09}).

Let $U:[0,\,\fz)\to [0,\,\fz)$ be a continuous convex function with $U(0)=0$
and define the associated {\it functional} $\mathscr U:\mathscr P_2( X)\to\rr\cup\{+\fz\}$ by setting
$$\mathscr U(\mu)\equiv\int_ X U\lf(\frac{d\mu}{dm}\r)\,dm+U'(\fz)\mu_ {\rm sing }( X),$$
where $\mu_{\rm sing }$ is the  singular part
of the Lebesgue decomposition of $\mu$ with respect to $m$, and
$U '(\fz)\equiv\lim_{r\to\fz}\frac1r U (r)$.
If $U'(\fz)=\fz$, then $\mathscr U(\mu)<\fz$ means that $\mu$ is absolutely continuous with respect to $m$,
namely, $\mu_{\rm sing }=0$.  If $U'(\fz)<\fz$, this is not necessarily the case.


\begin{defn}\label{d4.1}\rm
Let $U$ be  a continuous convex function with $U(0)=0$ and $\lz\in\rr$. Then
$\mathscr U$ is called  {\it  weakly $\lz$-displacement convex} if for all $\mu_0,\,\mu_1\in\mathscr P_2( X)$,
there exists some Wasserstein geodesic
$\{\mu_t\}_{t\in[0,\,1]}$  along which
\begin{equation}\label{e4.2}
 \mathscr U(\mu_t)\le t\mathscr U(\mu_1)+(1-t)\mathscr U(\mu_0)-\frac12\lz t(1-t) W_2(\mu_0,\,\mu_1)^2.
\end{equation}

\end{defn}

\begin{rem}\rm
If for every pair of $\mu_0,\,\mu_1\in\mathscr P_2( X)$ that are absolutely continuous
with respect to $m$ and have continuous densities,
there exists some Wasserstein geodesic
$\{\mu_t\}_{t\in[0,\,1]}$ along which \eqref{e4.2} holds, then as shown in \cite[Lemma 3.24]{lv09},
 $\mathscr U$ is weakly $\lz$-displacement convex.
\end{rem}

A curve $\{\mu_t\}_{t\in I}\subset\mathscr P_2( X)$ on an interval $I\subset\rr$
is {\it  absolutely continuous} if there exists a function $f\in L^1(I)$ such that
\begin{equation}\label{e4.3}
 W_2(\mu_t,\,\mu_s)\le\int_t^sf(r)\,dr
\end{equation}
for all $s,\,t\in I$ with $t\le s$. Obviously, an absolutely continuous curve is  continuous.
For an absolutely continuous curve $\{\mu_t\}_{t\in I}\subset\mathscr P_2( X)$, its {\it metric derivative}
$$|\dot\mu_t|\equiv\lim_{s\to t}\frac{ W_2(\mu_t,\,\mu_s)}{|t-s|}$$
is well defined for almost all $t\in I$; see \cite[Theorem 1.1.2]{ags}.
Moreover,  $|\dot\mu_t|\in L^1(I)$, and it is the minimal function such that
\eqref{e4.3} holds.
For $\mu\in\mathscr P_2( X)$,
define the {\it local slope} of $\mathscr U$ at $\mu$ as
$$|\nabla^-\mathscr U|(\mu)\equiv\limsup_{\nu\to\mu,\,\nu\ne\mu}\frac{
[\mathscr U(\mu)-\mathscr U(\nu)]_+}{ W_2(\mu,\nu)},$$
where $a_+=\max\{a,\,0\}$.

Now we recall the definition of a gradient flow
of a weakly $\lz$-displacement convex functional  $\mathscr U$,  via the energy dissipation identity.
\begin{defn}\label{d4.2}\rm
Let $\mathscr U:\mathscr P_2( X)\to\rr\cup\{\fz\}$ be
weakly $\lz$-displacement convex for some $\lz\in\rr$.
An absolutely continuous curve $\{\mu_t\}_{t\in[0,\,\fz)}\subset\mathscr P_2( X)$
is called a {\it gradient flow} of $\mathscr U $ if $\mathscr U(\mu_t)<\fz$ for all $t\ge0$,
and for all $0\le t<s$,
\begin{equation}\label{e4.4}
 \mathscr U(\mu_t)=\mathscr U(\mu_s)+\frac12\int_t^s |\dot\mu_r|^2\,dr+
\frac12\int_t^s|\nabla^-  \mathscr U|^2(\mu_r)  \,dr.
\end{equation}
\end{defn}

Associated with the convex function  $U_\fz:\ [0,\,\fz)\to[0,\,\fz)$ defined by
$U_\fz(r)=r\log r$ for $r>0$ and $U(0)=0$,
we have the functional
$\mathscr U_\fz:\ \mathscr P_2( X)\to\rr\cup\{+\fz\}$,
which is well-defined and  lower semicontinuous on $\mathscr P_2( X)$;
see, for example, \cite[Theorem B.33]{lv09}.
Denote by $\mathscr P_2^\ast( X)$ the collection of $\mu\in\mathscr P(X)$ such that $\mathscr U_\fz(\mu)<\fz$.
Since $U'_\fz(\fz)=\fz$,  $\mu\in\mathscr P_2^\ast( X)$ implies that $\mu$ is absolutely continuous
with respect to $m$ and $U_\fz(\frac{d\mu}{dm})\in L^1( X)$.
Recall that if $\mathscr U_\fz$ is weakly $\lz$-displacement convex on $\mathscr P_2(X)$ for some $\lz\in\rr$,
then $(X,\,d,\,m)$ is said to have {\it Ricci curvature bounded from below}  in the sense of Lott-Sturm-Villani \cite{s06a,s06b,lv09}.

Recently, under the compactness of $ X$ and weak $\lz$-displacement convexity
of $\mathscr U_\fz$,
Gigli \cite{g10} obtained the existence, uniqueness and stability of the gradient flow of
$\mathscr U_\fz$; for the basics of the theory of gradient flows see
\cite{ags}.
Under some  further additional conditions,
we are going to prove in Theorem \ref{t4.1} below that this gradient flow is
actually given by the heat flow.
Recall that the heat flow is the unique gradient flow of the Dirichlet energy functional $\mathscr E$ on
the Hilbert space $L^2( X).$
Moreover, the heat flow can be represented by the strongly continuous group $\{T_t\}_{t\ge 0}$ on $L^2( X)$
generated by the unique selfadjoint operator $\Delta $, which is determined by
 $$-\int_ X u \Delta v \,d m =\mathscr E(u,\,v)=\int_ X\,d\Gamma(u,\,v)$$
for all $u,\,v\in\bd$.
Indeed, for every $\mu\in\mathscr P( X)$,
the {\it heat flow} $\{T_t\mu\}_{t\in[0,\,\fz)}$ is  given by
 $T_0\mu=\mu$  and when $ t>0$,
 $T_t\mu$ is defined as the unique nonnegative Borel regular measure satisfying
$$\int_ X \phi dT_t\mu=\int_ X\int_ X \phi(x)T_t(x,\,y)\,d\mu(y)\,d m(x).$$
Then by $T_t1=1$, we have $T_t\mu\in\mathscr P( X)$, and hence
$\{T_t\mu\}_{t\in[0,\,\fz)}$ is a curve in $\mathscr P_2( X)$.
Notice that if $\mu =f m$, then $T_t\mu=(T_tf) m$ for all $t\ge0$.

\begin{thm}\label{t4.1}
Assume that $( X,\,d,\,m)$ is compact and satisfies a doubling property,
and that $( X,\,\mathscr E,\,m)$ satisfies the Newtonian property.
If $\mathscr U_\fz$ is weak $\lz$-displacement convex for some $\lz\in\rr$,
then for every $\mu\in\mathscr P_2^\ast( X)$, the  heat flow $\{T_t\mu\}_{t\in[0,\,\fz)}$
gives the unique gradient flow of $\mathscr U_\fz$ with initial value $\mu$.
\end{thm}

We follow the procedure of \cite{gko} to prove Theorem \ref{t4.1}.
Let $\{\mu_t\}_{t\in[0,\,\fz)}\subset\mathscr P_2( X)$ be an absolutely continuous curve
 that satisfies $\mathscr U_\fz(\mu_t)<\fz$ for all $t\ge0$.
To  prove that $\{\mu_t\}_{t\in[0,\,\fz)}$ is a gradient flow of $\mathscr U_\fz$,
we observe that since $\mathscr U_\fz$ is
 weak $\lz$-displacement convex   and lower semicontinuous,
by \cite[Corollary 2.4.10]{ags}, for all $s>t\ge0$,
we have
\begin{equation}\label{e4.x5}
|\mathscr U_\fz(\mu_t)-\mathscr U_\fz(\mu_s)|
\le\int_t^s|\nabla^- \mathscr U_\fz|(\mu_r)|\dot\mu_r|\,dr,
\end{equation}
which implies, by Young's inequality, that
\begin{equation} \label{e4.5}
\mathscr U_\fz(\mu_t)\le \mathscr U_\fz(\mu_s)+\frac12\int_t^s |\dot\mu_r|^2\,dr+
\frac12\int_t^s|\nabla^-  \mathscr U_\fz|^2(\mu_r)  \,dr.
\end{equation}
So it suffices to check that for all $s>t\ge0$,
$$ \mathscr U_\fz(\mu_s)+\frac12\int_t^s |\dot\mu_r|^2\,dr+
\frac12\int_t^s|\nabla^-  \mathscr U_\fz|^2(\mu_r)  \,dr\le \mathscr U_\fz(\mu_t),$$
which is further reduced to proving
\begin{equation}\label{e4.6}
 \frac12\ |\dot\mu_r|^2+
\frac12 |\nabla^-  \mathscr U_\fz|^2(\mu_r)   \le -\frac d{dr}\mathscr U_\fz(\mu_r)
\end{equation}
for almost all $r\ge0$.

Therefore, with the aid of \cite{r11,r12}, Theorem \ref{t4.1} will follow from  Lemma  \ref{l4.1},
Proposition \ref{p4.2} and Proposition \ref{p4.3} below.

\begin{lem}\label{l4.1}
Assume that $( X,\,d,\,m)$ is compact and satisfies a doubling property,
and that $( X,\,\mathscr E,\,m)$  supports a weak $(1,\,2)$-Poincar\'e inequality.
If $\mathscr U_\fz$ is weakly $\lz$-displacement convex for some $\lz\in\rr$,
then there exists a constant $C_6\ge1$ such that
for $\mu=fm\in\mathscr P^\ast_2 ( X)$,
\begin{equation}\label{e4.7}
 |\nabla^- \mathscr U_\fz|^2(\mu)\le 4 C_6\int_ X \,d\Gamma(\sqrt f,\,\sqrt f).
\end{equation}
Moreover, if  $( X,\,\mathscr E,\,m)$ satisfies the Newtonian property, then $C_6=1$.
\end{lem}

Lemma \ref{l4.1} follows from the following result; see \cite[Theorem 20.1]{v09}.
\begin{prop}\label{p4.1}
Let $U$ be a continuous convex function on $[0,\,\fz)$.
Let $\{\mu_t\}_{t\in[0,\,1]}\subset\mathscr P_2( X)$ be an absolutely continuous geodesic with density
$\{\rho_t\}_{t\in[0,\,1]}$,
and $U( \rho_t)\in L^1( X)$ for all $t\in[0,\,1]$.
Further assume that $\rho_0\in\lip( X)$,
$U(\rho_0),\,\rho_0U'(\rho_0)\in L^1( X)$
and $U'$ is Lipschitz on $\rho_0( X)$. Then
$$\liminf_{t\to0}\lf[\frac{\mathscr U(\mu_t)-\mathscr U(\mu_0)}{t}\r]\ge -\int_{ X\times X}
U''(\rho_0(x_0))|\nabla^-\rho_0|(x_0)d(x_0,\,x_1)\,d\pi(x_0,\,x_1),$$
where $\pi$ is an optimal coupling of $\mu_0$ and $\mu_1$.
\end{prop}

In Proposition \ref{p4.1} and below, for a measurable function $f$ on $ X$, set
$$|\nabla^-f|(x)\equiv\limsup_{y\to x}\frac{[f(x)-f(y)]_+}{d(x,\,y)}.$$
Obviously, $|\nabla^-f|(x)\le \lip\,f(x)$ for all $x\in X$.
However,    if $( X,\,d,\,m)$ satisfies a doubling property and
 supports a weak $(1,\,p)$-Poincar\'e inequality for some $p\in[1,\,\fz)$,
 then $|\nabla^-f|=\lip\,f$ almost everywhere. See \cite[Remark 2.27]{lv07}.

\begin{proof}[Proof of Lemma \ref{l4.1}.]
We first assume that $f\in \lip( X)$ and $f$ is bounded away from zero.
By the definition of $|\nabla^- \mathscr U_\fz|(\mu)$,
it suffices to consider $\nu\in\mathscr P_2( X)$ with
$\mathscr U_\fz(\nu)<\mathscr U_\fz(\mu)$.
Since
$\mathscr U_\fz(\nu)<\fz$,
we have $\nu\in\mathscr P_2^\ast( X)$.
By the convexity of $\mathscr U_\fz$,
there exists a curve $\{\mu_t\}_{t\in[0,\,1]}\subset\mathscr P_2( X)$ such that
$\mu_0=\mu$ and $\mu_1=\nu$, along which
\eqref{e4.2} holds. Moreover, by \eqref{e4.2}, we have $\mathscr U(\mu_t)<\fz$ for all $t\in[0,\,1]$,
which further means that $\mu_t$ is absolutely continuous with respect to $m$.
Denote the density  by $\rho_t$.

Notice that $U_\fz$ and $\{\mu_t\}_{t\in[0,\,1]}$ fulfill all the conditions required in Proposition
\ref{p4.1}.  So  by $U''(s)=\frac1s$, optimality of $\pi$ and the H\"older inequality, we have
\begin{eqnarray*}
\liminf_{t\to0}\lf[\frac{\mathscr U_\fz(\mu_t)-\mathscr U_\fz(\mu_0)}{t}\r]&&\ge -\int_{ X\times X}
 \frac1{[\rho_0(x_0)]^2}|\nabla^-\rho_0|(x_0)d(x_0,\,x_1)\,d\pi(x_0,\,x_1)\\
&& \ge -\lf\{\int_{ X\times X}
 \frac1{[\rho_0(x_0)]^2}|\nabla^-\rho_0|^2(x_0)\,d\pi(x_0,\,x_1)  \r\}^{1/2} \\
&&\quad\quad\times\lf\{\int_{ X\times X}d(x_0,\,x_1)^2\,d\pi(x_0,\,x_1)\r\}^{1/2}\\
&& =-W_2(\mu_0,\,\mu_1)\lf\{\int_ X
 \frac1{[\rho_0(x_0)]^2}|\nabla^-\rho_0|^2(x_0)\,d\mu_0(x_0)  \r\}^{1/2} \\
&& =-W_2(\mu_0,\,\mu_1)\lf\{\int_ X
 \frac1{f}|\nabla^-f|^2 \,dm  \r\}^{1/2},
\end{eqnarray*}
which together with $|\nabla^-f|^2\le |\lip\, f|^2\le C_1^2\Gamma(f,\,f) $ almost everywhere implies that
\begin{eqnarray}\label{e4.8}
\limsup_{t\to0}\lf[\frac{\mathscr U_\fz(\mu_0)-\mathscr U_\fz(\mu_t)}{tW_2(\mu_0,\,\mu_1)}\r]
\le C_1\lf\{\int_ X
 \frac1{f}\,d\Gamma(f,\,f)  \r\}^{1/2}.
\end{eqnarray}
On the other hand, by the weak displacement convexity of $\mathscr U_\fz$, we   have
$$
\lf[\frac{\mathscr U_\fz(\mu_0)-\mathscr U_\fz(\mu_1)}{ W_2(\mu_0,\,\mu_1)}\r]
\le \lf[\frac{\mathscr U_\fz(\mu_0)-\mathscr U_\fz(\mu_t)}{tW_2(\mu_0,\,\mu_1)}\r]-\frac12\lz(1-t) W_2(\mu_0,\,\mu_1),$$
which together with \eqref{e4.8} yields that
 $$
\lf[\frac{\mathscr U_\fz(\mu_0)-\mathscr U_\fz(\mu_1)}{ W_2(\mu_0,\,\mu_1)}\r]
\le  C_1\lf\{\int_ X
 \frac1{f}\,d\Gamma(f,\,f)  \r\}^{1/2}+ \frac12|\lz| W_2(\mu_0,\,\mu_1),$$
and hence, letting $\mu_1\to\mu_0$ with respect to $ W_2$,
$$|\nabla^-\mathscr U_\fz|(\mu_0)
\le  C_1\lf\{\int_ X
 \frac1{f}\,d\Gamma(f,\,f)  \r\}^{1/2}=2C_1\lf\{\int_ X
  \,d\Gamma(\sqrt f,\,\sqrt f)  \r\}^{1/2} .$$
This is as desired.

For $f\in \lip ( X)$ with $\sqrt f\in\bd$,
letting $f_n=(f\vee\frac1n)\wedge n$, we have
$f_n\in\lip( X)$ and $\frac{f_nm}{\|f_n\|_{L^1( X)}}\in\mathscr P_2( X)$.
Moreover, since $f_n\ge \frac1n$, by the above argument,
we have
 $$|\nabla^-\mathscr U_\fz|\lf(\frac{f_nm}{\|f_n\|_{L^1( X)}}\r)
\le  2C_1\frac1{{\|f_n\|_{L^1( X)}}}\lf\{\int_ X
 \,d\Gamma(\sqrt{f_n},\,\sqrt{f_n})  \r\}^{1/2}.$$
Moreover, recall that the lower semicontinuity of $\mathscr U_\fz$ implies that of $|\nabla^-\mathscr U_\fz|$; see \cite[Corollary 2.4.11]{ags}.
Since  $\sqrt {f_n}\to\sqrt f$ in $\bd$,
$\|f_n\|_{L^1( X)}\to1$
and $\frac{f_nm}{\|f_n\|_{L^1( X)}}\to fm$ in $\mathscr P( X)$,
we have
\begin{eqnarray*}
 |\nabla^-\mathscr U_\fz|(fm)&& \le\liminf_{n\to\fz} |\nabla^-\mathscr U_\fz|\lf(\frac{f_nm}{\|f_n\|_{L^1( X)}}\r)
\le 2C_1\lf\{\int_ X
 \,d\Gamma(\sqrt f,\,\sqrt f)  \r\}^{1/2}.
\end{eqnarray*}

Generally, for $\mu=fm\in\mathscr P_2^\ast( X)$, we may assume  $\sqrt f\in\bd$
without loss of generality. By Theorem \ref{t2.2},
we know that $\lip( X)$ is dense in $\bd$. So there exists a sequence
$\{g_n\}_{n\in\nn}\subset\lip( X)$ such that $g_n\to \sqrt f$ in $\bd$.
Since $f\ge 0$ almost everywhere, we still have $|g_n|\to \sqrt f$ in $\bd$.
Notice that $|g_n|^2\in\lip( X)$ and $\|g_n\|_{L^2( X)}\to1$. By the
lower semicontinuity of $\mathscr U_\fz$ again and the above result for Lipschitz functions,
we have
\begin{eqnarray*}
 |\nabla^-\mathscr U_\fz|(\mu)&& \le\liminf_{n\to\fz} |\nabla^-\mathscr U_\fz|\lf(\frac{|g_n|^2m}{\|g_n\|_{L^2( X)}}\r)\\
&&\le 2C_1\liminf_{n\to\fz}\frac1{\|g_n\|_{L^2( X)}}\lf\{\int_ X
 \,d\Gamma(|g_n|,\,|g_n|)  \r\}^{1/2}\\
&&=2C_1\lf\{\int_ X
 \,d\Gamma(\sqrt f,\,\sqrt f)  \r\}^{1/2}.
\end{eqnarray*}
This finishes the proof of Lemma \ref{l4.1}.
\end{proof}

\begin{prop}\label{p4.2}
Assume that $( X,\,d,\,m)$ is compact and satisfies a doubling property,
and that $( X,\,\mathscr E,\,m)$ supports a weak $(1,\,2)$-Poincar\'e inequality.
Let $\mu=fm\in\mathscr P^\ast_2( X)$.
Then $\{T_t\mu\}_{t\in[0,\,\fz)}\subset\mathscr P^\ast_2( X)$,
  $\{\sqrt {T_tf}\}_{t\in[0,\,\fz)}\subset\bd$ with a locally uniform bound on $ (0,\,\fz)$,
and
$\mathscr U_\fz(T_t\mu) $ is locally Lipschitz on $(0,\,\fz)$
and for almost all $t\in(0,\,\fz)$,
\begin{equation}\label{e4.9}
\frac d{dt}\mathscr U_\fz(T_t\mu)=-\int_ X\frac1{T_tf}\,d\Gamma(T_tf,\,T_tf).
\end{equation}
\end{prop}

\begin{proof}
Recall that, under the assumptions of Proposition \ref{p4.2}, it was proved in \cite{s95} that for every $t>0$,
the kernel $T_t$ is locally H\"older continuous in each variable and satisfies
\begin{equation}\label{e4.10}
C^{-1}\frac1{m(B(x,\,{\sqrt t}))}e^{-\frac{d^2(x,\,y)}{ c_1t}}\le T_t(x,\,y)\le C\frac1{m(B(x,\,{\sqrt t}))}e^{-\frac{d^2(x,\,y)}{ c_2t}}.
\end{equation}
So, for every $t\ge\frac1n$, $T_t f$ is continuous, and moreover,
 $0<C(n)^{-1}\le T_t(x,\,y)\le C(n) $ implies that  $ C(n)^{-1}\le T_tf(x)\le C(n) $ for all $x\in X$. From this, it is easy to  see that
$0\le \mathscr U_\fz(T_t\mu)\le C(n)\log C(n)<\fz$ for all $t\ge 1/n$.
For $t\ge 1/n$,  by $T_{1/n}f\in\bd$  and the fact that
 $\mathscr E(T_tf,\,T_tf)$  is  decreasing in $t$ (both of these facts follow
by functional calculus),
we have
$$\int_ X\frac1{T_tf}\,d\Gamma(T_tf,\,T_tf)\le C(n)\int_ X \,d\Gamma(T_tf,\,T_tf)
\le C(n)\int_ X \,d\Gamma(T_{1/n}f,\,T_{1/n}f),$$
which together with the chain rule implies that $\sqrt {T_tf}\in\bd$
with locally uniform bound on $(0,\,\fz)$.

Observe that the function $U_\fz(s)=s\log s$ is smooth on the interval $(\frac1n,\,n)$ for all $n$,
and that $T_t f$, as  $L^2( X)$-valued function in $(0,\,\fz)$,
is locally Lipschitz on $(0,\,\fz)$. So  $\mathscr U_\fz(T_t\mu)$ is locally Lipschitz in $(0,\,\fz)$.
Therefore, by the chain rule for $\Gamma$
and the fact that $\Gamma (1,\,h)=0$ for all $h\in\bd$, we have
\begin{eqnarray*}
\frac d{dt}\mathscr U_\fz(T_t\mu)&&=\int_ X U_\fz'(T_tf)\Delta T_tf\,dm
=\int_ X (\log T_tf+1)\Delta T_tf\,dm\\
&&=-\int_ X\,d\Gamma(\log T_tf+1,\, T_tf)=-\int_ X\,d\Gamma(\log T_tf ,\, T_tf)\\
&&=-\int_ X\frac1{T_tf}\,d\Gamma( T_tf ,\, T_tf).
\end{eqnarray*}
This is as desired.
\end{proof}

The following result is essentially proved in \cite[Proposition 3.7]{gko}.
We point out that,
comparing with the assumptions of Theorem \ref{t4.1}, we  can get rid of the Newtonian property
here since in the proof, instead of $\frac d{dm}\Gamma(u,\,u)=(\lip\, u)^2$,
it is enough to use $\frac d{dm}\Gamma(u,\,u)\le(\lip\, u)^2$
(after writing the first version of this paper, we learned that this was also realized in \cite[Lemma 6.1]{ags11}).
For completeness, we give its proof.

\begin{prop}\label{p4.3}
Assume that $( X,\,d,\,m)$ is compact and satisfies a doubling property,
and
that $( X,\,\mathscr E,\,m)$ supports a weak $(1,\,2)$-Poincar\'e inequality.
For  every $\mu=fm\in\mathscr P_2^\ast( X)$,
$\{T_t\mu\}_{t\in[0,\,\fz)}$ is an absolutely continuous curve in $\mathscr P_2( X)$
and for almost all $t\in[0,\,\fz)$,
\begin{equation}\label{e4.11}
|\dot T_t\mu|^2\le\int_ X\frac1{T_tf}\,d\Gamma(T_tf,\,T_tf).
\end{equation}
\end{prop}

To prove this, we recall the following result about the Hamilton-Jacobi semigroup established in \cite{lv07,behm}.
For $\phi\in\lip( X)$, set $Q_0f=f$ and for $t\ge0$, define
$$Q_t\phi(x)\equiv\inf_{y\in X}\lf[\phi(y)+\frac{1}{2t}d^2(x,\,y)\r].$$

\begin{prop}\label{p4.4}
Assume that $( X,\,d,\,m)$  satisfies a doubling property and supports
a weak $(1,\,p)$-Poincar\'e inequality for some $p\in[1,\,\fz)$.
Then the following hold:

(i)  For all $t,\,s\ge0$ and all $x\in X$, $Q_tQ_s\phi(x)=Q_{t+s}\phi(x)$;

(ii) For all $t\ge0$, $Q_t\phi\in\lip( X)$;

(iii) For all $t\in (0,\,\fz)$ and almost  all $x\in X$,
$$\frac d{dt}Q_t\phi(x)+\frac12|\nabla^-\, Q_t\phi|^2(x)=0.$$
\end{prop}

\begin{proof}[Proof of Proposition \ref{p4.3}.]
Let $t,\,s>0$.   By the Kantorovich duality,
\begin{eqnarray*}
\frac12  W^2_2(T_t\mu,\,T_{t+s}\mu)&&
=\lf[\int_ X Q_1\phi\,dT_ {t+s}\mu-\int_ X  \phi\,dT_t\mu\r];
\end{eqnarray*}
for some $\phi\in L^1( X)$; see, for example, \cite[Theorem 5.10]{v09} and \cite[Section 6]{ags}.
Moreover, by checking the proof (see, for example, \cite[p.\,66]{v09}), we know that $|\phi|$ is bounded
and for all $x\in X$,
$$\phi(x)=\sup_{y\in X}[Q_1\phi(y)-\frac12d^2(x,\,y) ].$$
Since $ X$ is compact and hence bounded, we further have that $\phi,\, Q_1\phi \in \lip( X)$;  we omit the details.
Observe that, by Proposition \ref{p4.4},
$Q_r\phi$  as an $L^2( X)$-valued function of $r$ is
Lipschitz on $[0,\,1]$ and hence is differentiable almost everywhere. 
Similarly, $T_{t+rs}f$  as an $L^2( X)$-valued function of $r$ is
Lipschitz on $[0,\,1]$ and hence is differentiable almost everywhere. 
Therefore, $(Q_r\phi)T_{t+rs}f$ as an $L^1( X)$-valued function of $r$
is Lipschitz in $[0,\,1]$ and hence is differentiable almost everywhere. Thus
\begin{eqnarray*}
\frac12  W^2_2(T_t\mu,\,T_{t+s}\mu)&& =\int_0^1\frac d{dr}\int_ X (Q_r\phi)(T_ {t+rs}f)\,dm\,dr\\
&& =\int_0^1\int_ X\lf[-\frac12|\nabla^-\, Q_r\phi|^2 (T_{t+rs}f)+s(Q_r\phi)\Delta T_{t+rs}f  \r]\,dm\,dr.
\end{eqnarray*}
Since $$\frac d{dm}\Gamma(Q_r\phi,\,Q_r\phi)\le (\lip\, Q_r\phi)^2=|\nabla^-\, Q_r\phi|^2$$ almost everywhere
as given by Theorem \ref{t2.1} and \cite[Remark 2.27]{lv07}, we have
\begin{eqnarray*}
\frac12  W^2_2(T_t\mu,\,T_{t+s}\mu)
&& \le-\frac12\int_0^1\int_ X    T_{t+rs}f \,d\Gamma(Q_r\phi,\,Q_r\phi)\,dr
+ s\int_0^1\int_ X (Q_r\phi)\Delta T_{t+rs}f \,dm\,dr.
\end{eqnarray*}
Moreover, by the Cauchy-Schwarz inequality for Dirichlet forms,  we have
\begin{eqnarray*}
  &&\int_ X  (Q_r\phi)\Delta T_{t+rs}f \,dm
 =
 \int_ X\,d\Gamma(Q_r\phi,\, T_{t+rs}f)\\
&&\quad=
 \int_ X (T_{t+rs}f)^{1/2}\cdot\frac1{(T_{t+rs}f)^{1/2}}\,d\Gamma(Q_r\phi,\, T_{t+rs}f)\\
&&\quad\le \frac1{2s} \int_ X T_{t+rs}f\, d\Gamma(Q_r\phi,\,Q_r\phi)
+ \frac s{2 } \int_ X\,\frac1{T_{t+rs}f}\, d\Gamma(T_{t+rs}f,\,T_{t+rs}f),
\end{eqnarray*}
which implies that
\begin{equation}\label{e4.12}
 W^2_2(T_t\mu,\,T_{t+s}\mu)\le  {s^2}  \int_0^1\int_ X\,\frac1{T_{t+rs}f}\, d\Gamma(T_{t+rs}f,\,T_{t+rs}f)\,dr.
\end{equation}
Since $T_{1/n}f\in\bd$, by an argument as in the proof of Proposition \ref{p4.2}, we have for $t\ge 1/n$,
\begin{eqnarray*}
 W^2_2(T_t\mu,\,T_{t+s}\mu)&&
\le  {s^2}C(n)   \int_0^1\int_ X\ \, d\Gamma(T_{t+rs}f,\,T_{t+rs}f)\,dr\\
&&
\le  {s^2} C(n) \int_ X\, d\Gamma(T_{1/n}f,\, T_{1/n}f),
\end{eqnarray*}
which  implies that $T_t\mu$ is locally Lipschitz continuous
and hence, $\{T_t\mu\}_{t\ge0}$ is an absolutely continuous curve in $\mathscr P_2( X)$. Moreover,
\eqref{e4.12} also implies that
$$ |\dot T_t\mu|^2\le  \int_ X\,\frac1{T_tf}\, d\Gamma(T_tf,\,T_tf),$$
which is as desired.
\end{proof}

For the proof of Theorem \ref{t4.1},
we still need a very recent result of Rajala \cite{r11,r12}:
 $\lambda$-displacement convexity of $\mathscr U_\fz$ implies
that $(X,\,d,\,m)$ supports a weak $(1,\,1)$- and hence a
weak $(1,\,2)$-Poincar\'e inequality.
This combined with Proposition \ref{p2.2} allows us to use
Lemma \ref{l4.1} and Propositions \ref{p4.2} and \ref{p4.3}
 in the proof of Theorem \ref{t4.1}.

\begin{proof}[Proof of Theorem \ref{t4.1}]
For $\mu=fm\in\mathscr P_2^\ast( X)$, it was proved in Proposition \ref{p4.3} that
$\{T_t\mu\}_{t\in[0,\,\fz)}$ is an absolutely continuous curve in $\mathscr P_2^\ast( X)$.
To prove that $\{T_t\mu\}_{t\in[0,\,\fz)}$ is a gradient flow of $\mathscr U_\fz$,
since \eqref{e4.5} follows from the displacement convexity of $\mathscr U_\fz$,
it suffices to check the reverse inequality
which is further reduced to \eqref{e4.6}. But
\eqref{e4.6} follows from \eqref{e4.7} with $C_3=1$, and \eqref{e4.9} and \eqref{e4.11}.
\end{proof}

Under the assumptions of Proposition \ref{p4.2},
for every $\mu\in\mathscr P_2( X)$ and $t>0$,
$T_t\mu$ is absolutely continuous with respect to $m$ and its density is continuous and bounded away from zero.
Indeed, for $t>0$ and every nonnegative $\phi\in\mathscr C( X)$, by $\mu( X)=1$ and \eqref{e4.10},
we have
\begin{eqnarray*}
\int_ X\phi\,dT_t\mu=\int_ X\int_ X\phi(x)T_t(x,\,y)\,d\mu(y)\,dm(x)\le
C(t)  \int_ X\phi(x)\,dm(x),
\end{eqnarray*}
which implies the  absolute  continuity of $T_t\mu$.
Let $f_t=\frac{d}{dm}T_t \mu$ for $t>0$.
Then $f_t\in L^1( X)$, and moreover, by the semigroup property,
$f_t=T_{t/2}f_{t/2}$, which together with \eqref{e4.10} and the continuity of the
kernel of $T_{t/2}$ implies that
$f_t$ is continuous and bounded away from zero.
Relying on the observations above  and Theorem \ref{t4.1}, we conclude the following result.

\begin{cor}\label{c4.1}
Let all the assumptions be as in Theorem \ref{t4.1}.
For every $\mu\in\mathscr P( X)$,
the curve $\{T_t\mu\}_{t\in(0,\,\fz)}\subset\mathscr P^\ast(X)$
is absolutely continuous on each $[\ez,\,\fz)$ for all $\ez>0$,
and $\mathscr U_\fz(T_t\mu)<\fz$ for all $t>0$ and  \eqref{e4.4} holds for all $s>t>0$.
\end{cor}

\begin{proof}
Let $\mu\in\mathscr P(X)$. For every $n\in\nn$,
by the argument before Corollary \ref{c4.1},
we  see that $T_{1/n}\mu$ is absolutely continuous with respect to $m$
 and that the Radon-Nikodym derivative  $f_{1/n}=\frac{d}{dm}T_{1/n}\mu$
belongs to
$L^1(X)$,
 that is, $T_{1/n}\mu=f_{1/n}m\in\mathscr P_2^\ast(X)$.
Hence Proposition \ref{p4.3} ensures that $\{ T_t(f_{1/n}m) \}_{t\in[0, \infty)}$
is an absolutely
continuous curve in  $\mathscr P_2^\ast(X).$
By Theorem \ref{t4.1},
$\{T_t(f_{1/n}m)\}_{t\in[0,\,\fz)}$ gives the unique heat flow with initial
value $T_{1/n}\mu=f_{1/n}m$,
which means that for all $s>t\ge0$, \eqref{e4.4} holds with
$\mu_r=(T_rf_{1/n})m$ when $t\le r\le s$.
Observe that $\{T_t\mu\}_{t\in[\frac1n,\,\fz)}=\{T_t(T_{1/n}\mu)\}_{t\in[0,\,\fz)}=\{T_t(f_{1/n}m)\}_{t\in[0,\,\fz)}$.
We further obtain $\mathscr U_\fz(T_t\mu )<\fz$ for all $t\ge\frac1n$,
$\{T_t\mu\}_{t\in[1/n,\,\fz)}$ is an absolutely continuous curve in $\mathscr P^\ast(X)$,
and for all $\frac1n\le t<s$, \eqref{e4.4} holds with $\mu_r=T_r\mu$ when $t\le r\le s$.
By the arbitrariness of $n$, we finally have that $\{T_t\mu\}_{t\ge(0,\,\fz)}$ is a locally absolutely continuous curve in $\mathscr P^\ast(X)$, $\mathscr U_\fz(T_t\mu)<\fz$ for all $t>0$
and for all $0< t<s$, \eqref{e4.4} holds with $\mu_r=T_r\mu$ when $t\le r\le s$.
\end{proof}

Furthermore, for $1< N<\fz$,  associated to the convex function
$U_N\equiv Nr-Nr^{1-1/N}$ for $r\ge0$, we have  the
functional $\mathscr U_N:\ \mathscr P_2( X)\to\rr\cup\{+\fz\}$,
which is well defined; see for example \cite{lv09}.
Assume that $\mathscr U_N$ is weakly $\lz$-displacement convex for some $N\in[1,\,\fz)$ and $\lz\ge0$.
Then, as proved in \cite[Theorem 5.31]{lv09}, $( X,\,d,\,m) $ satisfies a doubling property.

\begin{cor}\label{c4.2}
Assume  that $( X,\,\mathscr E,\,m)$ satisfies the Newtonian property.
If $\mathscr U_N$ is weakly $0$-displacement convex for some $N\in(1,\,\fz)$,
and $\mathscr U_\fz$ is weakly $\lz$-displacement convex for some $\lz\in\rr$,
then the heat flow gives the unique gradient flow of $\mathscr U_\fz$.
\end{cor}

\begin{rem}\rm
As pointed out by the referee, the weak $0$-displacement convexity assumption
on
$\mathscr U_N$
and the weak $\lz$-displacement convexity assumption on $\mathscr U_\fz$
in Corollary
\ref{c4.2}
can be replaced by a weaker condition, the curvature-dimension
condition $CD(\lz,\,N)$ (see \cite[Definition 1.3]{s06b}).
Indeed, under the condition $CD(\lz,\,N)$,
by \cite[Proposition 1.6]{s06b}, $\mathscr U_\fz$
is weakly $\lz$-displacement convex.
By the compactness of $X$ and \cite[Corollary 2.4]{s06b},
  $m$ satisfies the doubling property, and
  by \cite{r11},  $(X,
  \,d,\,m)$ supports a weak $(1,\,
  1)$-Poincar\'e inequality.
  With these, the conclusion of Corollary 5.2 follows from
  Theorem
  \ref{t4.1}.
%
\end{rem}

By the linearity and symmetry of heat flows,
we also have the following property of the gradient flow
of $\mathscr U_\fz$.

\begin{cor}\label{c4.3}
Let all the assumptions be as in Theorem \ref{t4.1}.
For every $\nu\in\mathscr P^\ast_2( X)$, let $\{\mu^\nu_t\}_{t\ge0}$
be the gradient flow of $\mathscr U_\fz$  with $\mu_0^\nu=\nu$

Then

(i) for all $\nu_0,\,\nu_1\in\mathscr P^\ast_2( X)$,  $t\ge0$ and $\lz\in[0,\,1]$,
$$\mu_t^{(1-\lz)\nu_0+\lz\nu_1}= (1-\lz)\mu_t^{\nu_0}+ \lz\mu_t^{\nu_1};$$

(ii) for all nonnegative $f,\,g\in L^1( X)$ with $\|f\|_{L^1( X)}=\|g\|_{L^1( X)}=1$ and $t\ge0$,
$$\int_ X f\,d\mu_t^{gm}=\int_ X g\,d\mu_t^{fm}.$$
\end{cor}

Recall that, under a non-branching condition,
 additional semiconcavity and local angle conditions,
the  linearity property  in Corollary \ref{c4.3} (i) was proved in \cite{s07}.
For the definitions of {\it $K$-semiconcavity} and {\it local angle condition}
introduced in \cite{s07},  see  Section 6 below.
We do not know if these conditions hold under the assumptions of Theorem \ref{t4.1}.
Also recall that the linearity property fails on Finsler manifolds as pointed out in \cite{os09}.

After we obtained Theorem \ref{t4.1},
we learned about a related result  established in \cite[Theorem 9.3]{ags11}.
Indeed, instead of the Dirichlet energy form $\mathscr E$,  Ambrosio, Gigli and Savar\'e \cite{ags11} considered the Cheeger energy functional
${\bf {Ch}}$ on $L^2(X)$, which is not necessarily Hilbertian.
They showed that, under the convexity of $\mathscr U_\fz$ and very few assumptions on $X$,
the gradient flow of  ${\bf Ch}$ coincides with the gradient flow of the entropy $\mathscr U_\fz$;
see \cite[Theorem 9.3]{ags11}.
Their proof also relies on the procedure outlined in \cite{gko} but works at a  high level of generality.

Assume that $X$ satisfies a doubling property and  supports a weak $(1,\,2)$-Poincar\'e inequality.
Then,  with the aid of Lemma \ref{l2.4}, ${\bf Ch}$ can be  written as
$${\bf {Ch}}(f)=\int_X(\aplip f)^2\,dm.$$
Moreover, by Theorem \ref{t2.2} (ii),
\begin{equation}\label{e4.x2}
 \mathscr E(f,\,f)\le {\bf Ch}(f)\le C_1\mathscr E(f,\,f),
\end{equation}
while $C_1=1$ if we further assume that $(X,\,\mathscr E,\,m)$ supports a Newtonian property.
We point out that, under the assumptions of Theorem \ref{t4.1}, the conclusion of Theorem \ref{t4.1} follows from
\eqref{e4.x2} with $C_1=1$ and \cite[Theorem 9.3]{ags11}.

Recall that in Theorem \ref{t4.1}, we showed that the Newtonian property  is a sufficient condition
to identify the heat flow and the gradient flow of entropy. Combining Theorem \ref{t2.2}, Proposition \ref{p4.3}
and \cite[Theorem 9.3]{ags11}, we  will show that the Newtonian property  is also necessary in the following sense.

\begin{thm}\label{t4.2}
Assume that $( X,\,d,\,m)$ is compact and satisfies a doubling property, and
that $\mathscr U_\fz$ is weakly $\lz$-displacement convex for some $\lz\in\rr$.
Then the following are equivalent:

(i) For every $\mu\in\mathscr P_2^\ast( X)$, the heat flow $\{T_t\mu\}_{t\in[0,\,\fz)}$
gives the unique gradient flow of $\mathscr U_\fz$ with initial value $\mu$.

(ii) $( X,\,\mathscr E,\,m)$ satisfies the Newtonian property.

(iii) For all $u\in\bd$, $\frac{d}{dm}\Gamma(u,\,u)=(\aplip u)^2$ almost everywhere.
\end{thm}

\begin{proof}
We recall again that, by \cite{r11,r12},
 the $\lambda$-displacement convexity of $\mathscr U_\fz$ implies
 that $(X,\,d,\,m)$ supports a weak $(1,\,2)$-Poincar\'e inequality.
Then the equivalence of (ii) and (iii)  follows from Theorem \ref{t2.2} and Lemma \ref{l2.5}.
If (ii) holds, then by Theorem \ref{t4.1}, we have (i).
Now assume that (i) holds.
Let $f\in\lip (X)$ be a positive function and set $\mu=fm\in\mathscr P_2^\ast( X)$.
By  Proposition \ref{p4.3}, we have that
 for almost all $t\in[0,\,\fz)$,
\begin{equation}\label{e4.x1}
|\dot T_t\mu|^2\le\int_ X\frac1{T_tf}\,d\Gamma(T_tf,\,T_tf).
\end{equation}
Moveover the assumption (i) says that
$\{T_t\mu\}_{t\in[0,\,\fz)}$ is the gradient flow of  $\mathscr U_\fz$.
By this, the convexity of $\mathscr U_\fz$, and Theorem 9.3 and Theorem 8.5 of \cite{ags11},
we know that $ T_t\mu$
coincides with the gradient flow of ${\bf Ch}$ and  satisfies that for almost all $t\in(0,\,\fz)$,
$$|\dot T_t\mu|^2=\int_ X\frac1{T_tf}   (\aplip T_tf)^2\,dm.$$
This and \eqref{e4.x1}, with the aid of $ \Gamma(T_tf,\,T_tf)\le(\aplip T_tf)^2m$ given in \eqref{e2.23},
further gives  $ \Gamma(T_tf,\,T_tf)=(\aplip T_tf)^2m$ for almost all $t\in(0,\fz)$.
Therefore,
$$\mathscr E(f,\,f)=\lim_{t\to\fz}\mathscr E(T_tf,\,T_tf)\ge \liminf_{t\to0}\int_X(\aplip T_tf)^2\,dm.$$
Observing that $$
\int_X(\aplip  f-\aplip T_tf)^2\,dm \le \int_X[\aplip  (f-T_tf)]^2\,dm
\ls \mathscr E(f-T_tf,\, f-T_tf)\to0,$$
we obtain $\mathscr E(f,\,f)= \int_X(\aplip  f )^2\,dm$.
With the help of $ \Gamma( f,\, f)\le(\aplip f)^2m$ given in \eqref{e2.23} again,
we have $ \Gamma( f,\, f)=(\aplip  f)^2m$ as desired.
Then a density argument yields (iii).
\end{proof}

\section{Applications to (coarse) Ricci curvatures}\label{s5}

In this section, we apply Corollary \ref{c2.1} to a variant of the dual
formula of Kuwada \cite{k10}
and the coarse Ricci curvature of Ollivier in Theorem \ref{t5.1} and
Corollary 6.1,
and then apply Theorem 5.1 to the Ricci curvatures  of Bakry-Emery  and
Lott-Sturm-Villani in  Corollary \ref{p5.3}.
We always let $ \mathscr E $ be a regular, strongly local Dirichlet form on $L^2(X,\,m)$,
assume that the topology induced by the intrinsic $d$ coincides with the
original topology on $X$ and that $(X,\,d,\,m)$ satisfies the
doubling property.


We begin with a variant of the dual formula established in
\cite[Theorem 2.2]{k10}, which is closely related to the coarse Ricci
curvature of Ollivier \cite{o09}.
Let $\{P_x\}_{x\in X}\subset\mathscr P( X)$ be a family of
probability measures on $X,$ so that the map
$x\to P_x$ from $ X$ to $ \mathscr P( X)$ is continuous. Then
$\{P_x\}_{x\in X} $ defines a bounded linear operator $P$ on $\mathscr C( X)$
by $Pf(x)=\int_ X f(y)\,dP_x(y)$ and we denote its dual
operator by $P^\ast:\, \mathscr P( X)\to \mathscr P ( X)$.
We also assume that $P_x$ is absolutely continuous with respect to $m$ with
density $P_x(y)$ for all $x\in X$,
and that $P_x(y)$ is a continuous function of $x$ for almost all $y\in X$.
Observe that  we do not assume that $\{P_x\}_{x\in X}$ has any
relation with the Dirichlet form $\mathscr E$.

By Corollary \ref{c2.1} and \cite{k10}, we have the following result.

\begin{thm}\label{t5.1}
Assume that $\mathscr E$ is a regular, strongly local Dirichlet form on
$L^2( X)$,
 the topology induced by the intrinsic distance is equivalent to the original
topology on  the locally compact space $X$,
and $( X,\,d,\,m)$ satisfies the doubling property.
Let  $K_1\ge 0$ be a positive constant. Then the following are equivalent:

(i) For all $\mu,\,\nu\in\mathscr P( X)$, $ W_1(P^\ast\mu,\,P^\ast\nu)\le K_1  W_1( \mu,\, \nu).$

(ii) For all $f\in \lip( X)\cap L^\fz( X)$,  $Pf\in \lip(X)$ and
$\|Pf\|_{\lip( X)}\le K_1  \| f\|_{\lip( X)}.$

(iii) For all $f\in \lip( X)\cap L^\fz( X)$,  $Pf\in \lip(X)$ and
$$\|\lip\, Pf\|_{L^\fz( X)}\le K_1  \|\lip\, f\|_{L^\fz( X)}.$$

(iv) For all $f\in \lip( X)\cap L^\fz( X)$, $Pf\in \lip(X)$ and
$$ \lf\|\frac d{dm}\Gamma(Pf,\,Pf)\r\|_{L^\fz( X)}\le (K_1) ^2   \lf\| \frac d{dm}\Gamma(f,\,f)\r\|_{L^\fz( X)}.$$
\end{thm}

\begin{proof}
For $K_1>0$,
the equivalence between (i) and (ii) follows from \cite[Theorem 2]{k10}
(by taking $\wz d=K_1d$ there).
Notice that  the proof of \cite[Theorem 2]{k10} for the case $p=1$ does not
require any weak Poincar\'e inequality.
The equivalence of (ii), (iii) and (iv) follows from
$$\|u\|^2_{\lip(X)} =\|\lip\,u\|^2_{L^\fz(X)}=\lf\|\frac{d}{dm}\Gamma(u,\,u)\r\|_{L^\fz(X)}$$
with $u=f\in L^\fz(X)$ and $u=Pf\in L^\fz(X)$, see Corollary \ref{c2.1}.
For $K_1=0$, the equivalence of (i) through (iv) follows from the case $K_1+\ez$
with $\ez>0$ and an approximation argument.
We omit the details.
\end{proof}

Associated to $(X,\,d,\,P)$ with $P$ as above, Ollivier \cite{ol09} introduced the
{\it coarse Ricci curvature} via
$$\kz(x,\,y)=1-\frac{W_1(P^\ast\dz_x,\,P^\ast\dz_y)}{d(x,\,y)}.$$
$(X,\,d,\,P)$ is said to have the {\it coarse Ricci curvature bounded from below by constant $K$} if
  $\kz(x,\,y)\ge K$ for all $x,\,y\in X$.
Obviously, $K\le1$.
Applying Theorem \ref{t5.1}, we have the following result.

\begin{cor}\label{c5.x1}
Under the assumptions of Theorem \ref{t5.1}, the following are equivalent:

(i) $(X,\,d,\,P)$  has the coarse Ricci  curvature bounded from below by  $K\le 1 $.

(ii)  For all $\mu,\,\nu\in\mathscr P( X)$, $ W_1(P^\ast\mu,\,P^\ast\nu)\le (1-K)  W_1( \mu,\, \nu) $.

(iii) For all $f\in \lip( X)\cap L^\fz( X)$, $Pf\in \lip(X)$ and
$$ \lf\|\frac d{dm}\Gamma(Pf,\,Pf)\r\|_{L^\fz( X)}\le ( 1-K) ^2   \lf\| \frac d{dm}\Gamma(f,\,f)\r\|_{L^\fz( X)}.$$
\end{cor}

\begin{proof}
By Theorem \ref{t5.1}, (i) follows from (ii) or (iii).
Conversely, if (i) holds, then  (ii) holds with $\mu=\dz_x$ and $\nu=\dz_y$, which together with \cite[Lemma 3.3]{k10}
further yields that (ii) holds with $\mu $ and $\nu \in\mathscr P(X)$.
\end{proof}

On the other hand, combining \cite[Theorem 1]{s07}, \cite{k10}, and
Theorems  \ref{t4.1} and \ref{t2.2} of our paper, and following the procedure
of \cite{gko},
we know that, under some additional conditions, a Ricci curvature bound from
below in the sense of
Lott-Sturm-Villani \cite{lv09,s06a,s06b} implies that in the sense of
Bakry-Emery \cite{b97,be83}.
Recall that $(X,\,\mathscr E,\,m)$ is said to have Ricci curvature bounded from below by $\lz\in\rr$ in the sense of Bakry-Emery
if  for all  $f\in \bd$ and $t\ge0$,
and for almost all $x\in X$,
\begin{equation}\label{e5.x3}
\frac d{dm}\Gamma(T_tf,\,T_tf)(x)\le e^{-2\lz  t} T_t\lf(\frac d{dm}\Gamma(f,\,f)\r)(x)
\end{equation}
Indeed, Savar\'e \cite{s07}  obtained  the contraction property of the  gradient  flow of the entropy with the aid of the semiconcavity and local angle
conditions.
Recall that $ X$ is {\it $K$-semiconcave} if $K\ge 1$ and for every geodesic $\gz$ and $y\in X$,
 $$d^2(\gz(t),\,y)\ge (1-t)d^2(\gz(0),\,y)+td^2(\gz(1),\,y)-Kt(1-t)d^2(\gz(0),\,\gz(1)).$$
Moreover, $X$ satisfies the {\it local angle condition} if for every triplet of geodesics $\gz_i$,
$i=1,\,2,\,3$, emanating from the same initial point $x_0$, the corresponding angles $\angle(\gz_i,\,\gz_j) \in[0,\,\pi]$
satisfy $$\angle(\gz_1,\,\gz_2)+\angle(\gz_2,\,\gz_3)+\angle(\gz_3,\,\gz_1)\le 2\pi,$$
where
$$\angle(\gz_i,\,\gz_j)\equiv\liminf_{s,\,t\to0+}\frac{d^2(x_0,\,\gz_i(s))+d^2(x_0,\,\gz_j(t))-d^2( \gz_i(s),\,\gz_j(s))}{
2d(x_0,\,\gz_i(s))d(x_0,\,\gz_j(t))}.$$
Kuwada  established a dual relation between contraction of the
gradient flow in Wassertein distance
and its pointwise Lipchitz constant estimate (see \cite{k10}).
Under our assumptions, Theorem  \ref{t2.2} identifies the pointwise Lipschitz constant with length of the gradient,
while Theorem \ref{t4.1} identifies the heat flow and gradient flow.

\begin{cor}\label{p5.3}
Under the assumptions of Theorem \ref{t4.1},
and further assuming that $(X,\,d)$ is compact,
$( X,\,d,\,m)$ is $K$-semiconcave for some
$K\ge1$ and satisfies a local angle condition,
if $\mathscr U_\fz$ is weakly $\lz$-displacement convex for some $\lz\in\rr$,
then
the following hold:

(i)  For all $\mu,\,\nu\in\mathscr P( X)$,
$ W_2(T_t\mu,\,T_t\nu)\le e^{- \lz t} W_2( \mu,\, \nu),$

(ii) For all  $f\in \bd$ and $t\ge0$, $T_tf\in\lip(X)$
and for all $x\in X$,
\begin{equation}\label{e5.x2}
[\lip\, T_tf (x) ]^2\le e^{-2\lz  t} T_t( \aplip f )^2(x).
\end{equation}

(iii) For all  $f\in \bd$  and $t\ge0$,   \eqref{e5.x3} holds for almost all $x\in X$.
\end{cor}

\begin{proof}
To see (i), since the heat flow coincides with gradient flow of $\mathscr U_\fz$ as given in \ref{t4.1},
it suffices to prove that for the gradient flows $\{\mu_t\}_{t\ge 0}$ and $\{\nu_t\}_{t\ge 0}$,
$ W_2 (\mu_t,\,\nu_t)\le e^{- \lz t} W_2 (\mu_0,\,\nu_0)$. But this was
already proved by Savar\'e \cite{s07}
and hence we have (i).

Obviously, applying (ii) and Theorem \ref{t2.2} (iii), we  have (iii).

Moreover, (ii) follows from (i), \cite[Theorem 2]{k10} and  an approximation argument.
Indeed, for $f\in\lip(X)$, by \cite[Theorem 2]{k10}, \eqref{e5.x2} follows from (i).
Generally, let $f\in \bd$.
By Theorem \ref{t2.2} (i), $\lip(X)$ is dense in $\bd$. Thus, there exists a sequence $f_i\in\lip(X)$
such that $f_i\to f$ in $\bd$ as $i\to\fz$.
For each $x\in X$,
\begin{eqnarray*}
|T_tf(x)-T_tf_i(x)|&&\le  \int_XT_t(x,\,y)|f(y)-f_i(y)|\,dy\\
&&\le  \|T_t(x,\,\cdot)\|_{L^2(X)} \|f -f_i \|_{L^2(X)} \le C(t)\|f -f_i \|_{L^2(X)},
\end{eqnarray*}
where $C(t)=\sup_{x\in X} \|T_t(x,\,\cdot)\|_{L^2(X)}<\fz$. Thus for each pair of $x,\,y\in X$,
$$|T_tf(x)-T_tf(y)|\le 2C(t)\|f -f_i \|_{L^2(X)}+ |T_tf_i(x)-T_tf_i(y)|.$$
Notice that for every rectifiable curve $\gz$,
$$
|T_tf_i(x)-T_tf_i(y)|\le \int_\gz  \lip\, T_tf_i  \,ds\le e^{- \lz t}\int_\gz [T_t(\lip f_i)^2]^{1/2}\,ds
 $$
and by Theorem \ref{t2.2} (iii),
\begin{eqnarray*}
 [T_t(\lip f_i)^2]^{1/2} &&\le [T_t(\lip f_i-\aplip f)^2]^{1/2}+[T_t(\aplip f)^2]^{1/2}\\
&&
\le [T_t(\aplip (f_i-  f))^2]^{1/2}+[T_t(\aplip f)^2]^{1/2}\\
&&\le \wz C(t) \|\aplip(f -f_i) \|_{L^2(X)}+[T_t(\aplip f)^2]^{1/2}\\
&&\le \wz C(t) \| f -f_i  \|_{\bd}+[T_t(\aplip f)^2]^{1/2},
\end{eqnarray*}
where $\wz C(t)=\sup_{x,\,y\in X} T_t(x,\,y)<\fz$.
Then
$$|T_tf(x)-T_tf(y)|\le [2C(t)+\wz C(t) e^{- \lz t}\ell(\gz)]\| f -f_i  \|_{\bd}+ e^{- \lz t}\int_\gz [T_t(\aplip f )^2]^{1/2}\,ds$$
and hence
$$|T_tf(x)-T_tf(y)|\le  e^{- \lz t}\int_\gz [T_t(\aplip f )^2]^{1/2}\,ds\le \wz C(t) \ell(\gz)\|f\|_\bd.$$
Choosing $\gz$ to be a geodesic joining $x$ and $y$, we see that $T_tf\in\lip(X)$.
Moreover, by the continuity of the heat kernel and hence of $e^{- \lz t}[T_t(\aplip f )^2]^{1/2}$,
we have that   $\lip\, T_tf(x)\le e^{- \lz t}[T_t(\aplip f )^2(x)]^{1/2}$ for all $x\in X$.
\end{proof}

\begin{rem}\rm
Notice that,  by \cite{s07,o09}, compact Aleksandrov spaces with curvature
bounded from below
satisfy the assumptions of Corollary 6.2 (in particular, the $K$-semiconcavity
and the local angle condition) and
thus they have Ricci curvature bounded from below in the sense of Bakry-Emery.
This conclusion can also be found in \cite{gko}.
\end{rem}

\section{Asymptotics of the gradient of the heat kernel}\label{s6}

We are going to give a characterization for the condition that
$\Gamma(d_x,\,d_x)=m$ for all
$x\in X$
via the short time asymptotics of the gradient of the heat semigroup; see
Theorem \ref{t6.1} below.

Assume that $X$ is compact and $(X,\,\mathscr E,\,m)$ has a spectral gap,
that is, there exists a positive constant $C_{\rm spec}$
such that for all $u\in \bd$,
$$\int_X|u-\bint_X u\,dm|^2\,dm\le C_{\rm spec}\mathscr E(u,\,u). $$
Obviously, if $(X,\,\mathscr E,\,m)$ satisfies a weak Poincar\'e  inequality in the sense of Section 2,
then it has a spectral gap.
Then the Varahdan asymptotic behavior of heat kernels
was established in \cite{r01}:
for all $x,\,y\in X$,
\begin{equation}\label{e6.7}
\lim_{t\to0}-4t\log T_t(x,\,y)= d^2(x,\,y);
\end{equation}
see \cite{n97} for Lipschitz manifolds and \cite{hr03} for general local and conservative Dirichlet forms.

On the other hand, on a Riemannian manifold,  Malliavin and Stroock \cite{ms96} (see also \cite{st97})
proved that
\begin{equation}\label{e6.x7}
\lim_{t\to0}-4t[\nabla \log T_t(\cdot ,\,y)](x)= [\nabla d^2(\cdot ,\,y)](x),
\end{equation}
for all $y\in M$ and all $x\in M$ outside the cut locus of $y$,
where $\nabla$  denotes the gradient on a Riemannian manifold.
On $\rn$, the Gaussian kernel $h_t(x)=c_nt^{n/2}\exp(-\frac{|x|^2}{4t}) $ satisfies
  $  |\nabla|x|^2|=  {4t|\nabla \log h_t(x)| } $ for all $t\in(0,\,\fz)$.

We show that a weak variant of \eqref{e6.x7} will reflect
a connection  between the length structure and gradient structure of Dirichlet forms.

\begin{thm}\label{t6.1}
Let $\mathscr E$ be a regular, strongly local Dirichlet form on $L^2(X,\,m)$.
Assume that $X$ is compact, the topology induced by $d$ coincides with the original topology,
 and that $(X,\,\mathscr E,\,m)$ has a spectral gap.
Then the following are equivalent:

(i) For all $x\in X$,
 $\Gamma(d_{x},\,d_{x})= m$.

(ii)   For every Borel measurable set  $A$ with positive measure and  each
$\vz\in\bd\cap\mathscr C_0( X)$,
 \begin{equation}\label{e6.3}
 \int_ X\vz d\Gamma(t\log T_t1_A,\,t\log T_t1_A )\to
\int_ X\vz d\Gamma(d_A^2/4,\, d_A^2/4 ).
\end{equation}
\end{thm}

Notice that \eqref{e6.3} is a weak variant of \eqref{e6.x7} while \eqref{e6.7} has a weak variant
as established in \cite[Theorem 3.10]{r01}.

\begin{prop}\label{p6.1}
Under the assumptions of Theorem \ref{t6.1},
  for every Borel measurable set  $A$ with positive measure and  each
$\vz\in\bd\cap\mathscr C_0( X)$,
\begin{equation}\label{e6.2}
 \int_ X  (-t\log T_t1_A)\vz\,dm\to
\int_ X\vz d_A^2/4\, dm.
\end{equation}
\end{prop}

\begin{proof}[Proof of Theorem \ref{t6.1}]
We first prove that (ii) implies (i).
Let $A$ be a Borel measurable set in $ X$, $u_t=-t\log T_t1_A$ for all $t>0$,
and $u_0=d_A ^2/4$. Then $u_t,\,u_0\in\bd_\loc$.  From  $\Gamma(d_{x },\,d_{x })\le m$,
it follows that $\Gamma(u_0,\,u_0)\le  u_0\,m.$ It suffices to prove the converse inequality.

Notice that the strong  locality of $\mathscr E$ implies that
 $\Gamma$ satisfies the Lebniz rule, namely, for $F\in C^1(\rr)$, every $\phi \in\bd\cap L^\fz( X)$
and  $\vz\in\bd\cap\mathscr C_0( X)$,  we have
$$\int_ X \vz\,d\Gamma(F\circ\phi,\,F\circ\phi)= \int_ X \vz (F'\circ\phi)^2\,d\Gamma( \phi,\,\phi),$$
and if  $F\in C^2(\rr)$,
$$\int_ X\vz\Delta (F\circ\phi)\,d\mu=\int_ X \vz (F'\circ\phi) \Delta\phi\,d\mu+
 \int_ X \vz (F''\circ\phi)\,d\Gamma( \phi,\,\phi).
$$

Then   for every   $\vz\in\bd\cap\mathscr C_0( X)$, by
$$\int_ X \vz\lf[\frac{d}{dt}T_t1_A+\Delta T_t1_A\r]\,d\mu=0,$$
we have
$$t\int_ X \frac{d u_t}{dt}\vz\,d\mu-t\int_ X\Delta u_t\vz\,d\mu=\int_ X u_t\vz\,d\mu-
 \int_ X  \vz\,d\Gamma(u_t,\,u_t),$$
from which it follows that
$$t\int_ X\frac{d u_t}{dt}\vz\,d\mu+t\mathscr E(u_t,\,\vz)=\int_ X u_t\vz\,d\mu-
 \int_ X  \vz\,d\Gamma(u_t,\,u_t)$$
and that
\begin{eqnarray}
 &&\frac1{ t}\int_0^t\int_ X  \vz\,d\Gamma(u_s,\,u_s)\,ds\label{e6.4}\\
&&\quad=\frac1t\int_0^t\int_ X u_s\vz\,d\mu\,ds-\frac1t\int_0^t\int_ X   s\frac{du_s}{ds} \vz\,d\mu\,ds-\frac1t\int_0^ts\mathscr E(u_s,\,\vz)\,ds.\nonumber
\end{eqnarray}

Let $\wz\vz\in \bd\cap\mathscr C_0( X)$ such that $\wz\vz=1$ on the support of $\vz$.
Notice that
$$ |\mathscr E(u_s,\,\vz)|=\lf|\int_ X \wz\vz\,d\Gamma(u_s,\,\vz)\r|
\le  \mathscr E(\vz,\,\vz)\int_ X(\wz\vz)^2\,d\Gamma(u_s,\,u_s ). $$
Then, by \eqref{e6.3}, we know that
$\mathscr E(u_s,\,\phi)$ is uniformly  bounded with respect to $s$.
Hence
\begin{equation}\label{e6.5}
 \frac1t\int_0^ts\mathscr E(u_s,\,\vz)\,ds\to0
\end{equation}
 as $t\to0$.

By \eqref{e6.2},  for every $\vz\in\bd\cap\mathscr C_0( X)$,
$$ \int_ X u_0\vz\,d\mu=\lim_{s\to0}\int_ X u_s\vz\,d\mu.
$$
Hence
$$\frac1t\int_0^t\int_ X u_s\vz\,d\mu\,ds\to \int_ X u_0\vz\,d\mu$$
and
$$
\frac1t\int_0^t\int_ X s\frac{du_s}{ds} \vz\,d\mu=\lim_{\ez\to0}
\frac st\int_ X   u_s  \vz\,d\mu\Big|_{s=\ez}^t- \frac1t\int_0^t\int_ X   u_s  \vz\,d\mu\,ds
\to 0 $$
as $t\to0$.
From these two facts,   \eqref{e6.4} and \eqref{e6.5}, it follows that
$$\frac1{ t}\int_0^t\int_ X  \vz\,d\Gamma(u_s,\,u_s)\,ds\to \int_ X u_0\vz\,d\mu$$
as $t \to0$,
which together with \eqref{e6.3} implies
\begin{equation}\label{e6.6}
\int_ X  \vz\,d\Gamma(u_0,\,u_0)=  \int_ X  u_0\vz\,d\mu.\end{equation}
This  gives $\Gamma(u_0,\,u_0)=  u_0\,\mu$. Moreover, if $A$ is compact, then for all $\vz\in\mathscr C_0( X)$
with $\mathrm {supp}\, \vz\subset A^\complement$,
 $$\int_ X\vz\,d\Gamma(d_A,\,d_A)=  4 \int_ X\vz\,d\Gamma(\sqrt {u_0},\,\sqrt{u_0})=  \int_ X\vz\frac1{ u_0}\,
d\Gamma( {u_0},\,  {u_0})= \int_ X\vz\,d\,\mu,$$
which means that $\Gamma(d_A,\,d_A)=\mu$.
Since $d_{\overline{B(x_0,\,r)}}\to d_{x_0}$ in $\bd_\loc$ as $r\to0$, we have (i).

Now we turn to prove that (i) implies (ii).
 It suffices to prove that $\mathscr E(u_0,\,u_0)\ge \|u_0\|_{L^1(X)}$.
Indeed, from this and $\Gamma(u_0,\,u_0)\le  u_0m$, it follows that $ u_0-\frac{d}{dm}\Gamma(u_0,\,u_0)=0$ almost everywhere,
which further implies that $\frac{d}{dm}\Gamma(d_A,\,d_A)=1$ almost everywhere on $A^\complement$, and hence gives (i).
We first observe that, by our assumption (ii),
\begin{equation}\label{e6.x1}
 \mathscr E(u_0,\,u_0)= \lim_{s\to0}\mathscr E(u_s,\,u_s)= \lim_{t\to0}\frac1t\int_0^t \mathscr E(u_s,\,u_s)\,dt.
\end{equation}
But,  taking $\vz=1$, \eqref{e6.4} yields  that
\begin{eqnarray*}
 \frac1t\int_0^t \mathscr E(u_s,\,u_s)\,dt&&= \frac1t\int_0^t \|u_s\|_{L^1(X)}\,ds-
\frac1t\int_0^ts\frac d{ds}\|u_s\|_{L^1(X)}  \,ds  \\
&&=\frac1t\int_0^t \|u_s\|_{L^1(X)}\,ds- \frac1t(s\|u_s\|_{L^1(X)})\Big|_{s\to0}^{s=t}+
\frac1t\int_0^t \|u_s\|_{L^1(X)}  \,ds  \\
&&=2\frac1t\int_0^t \|u_s\|_{L^1(X)}\,ds- \|u_t\|_{L^1(X)}.
\end{eqnarray*}
Then, by \eqref{e6.2} as given in Proposition \ref{p6.1},  taking $\vz=1$, we  have
$\|u_t\|_{L^1(X)}\to \|u_0\|_{L^1(X)}$ as $t$ tends to  $0$, which yields that
 $$\lim_{t\to0}\frac1t\int_0^t \mathscr E(u_s,\,u_s)\,dt= \|u_0\|_{L^1(X)}.$$
Combining this with \eqref{e6.x1}, we have  $\mathscr E(u_0,\,u_0)\ge \|u_0\|_{L^1(X)}$
as desired.
\end{proof}

\begin{rem}\rm
There exist a large variety of $(X,\,\mathscr E,\,m)$ satisfying
$\Gamma (d_x,\,d_x)=m$ for all $x\in X$,
including compact Riemannian manifolds, compact Alexandrov spaces, and the
Sierpinski gasket considered in Section 3.
Theorem \ref{t6.1} (ii) then gives the short time asymptotics of the gradient of the heat kernel for them.
\end{rem}


\noindent Pekka Koskela

\noindent Department of Mathematics and Statistics,
P. O. Box 35 (MaD),
FI-40014, University of Jyv\"askyl\"a,
Finland
\smallskip

\noindent{\it E-mail address}:   \texttt{pkoskela@maths.jyu.fi}

\bigskip

\noindent Yuan Zhou

\medskip

\noindent
Department of Mathematics, Beijing University of Aeronautics and Astronautics, Beijing 100191, P. R. China

and

\noindent Department of Mathematics and Statistics,
P. O. Box 35 (MaD),
FI-40014, University of Jyv\"askyl\"a,
Finland

\smallskip

\noindent{\it E-mail address}:  \texttt{yuanzhou@buaa.edu.cn}

\end{document}